\documentclass{amsart}
\usepackage{IMjournal}
\usepackage[hidelinks]{hyperref}
\usepackage[english]{babel}
\usepackage{amsmath,fullpage,graphicx,subeqnarray,amsthm,math,amsfonts,amssymb,epsf,eepic,epsfig,bbm}     
\usepackage[utf8]{inputenc}
\usepackage{ifthen}
\usepackage{color}
\usepackage{graphicx}
\usepackage{psfrag}
\usepackage[inline]{enumitem}
\usepackage{leftidx}
\usepackage{cite}

\newif\ifarxiv
\arxivtrue

\newtheorem{theorem}{Theorem}[section]

\newtheorem{corollary}[theorem]{Corollary}

\newtheorem{definition}[theorem]{Definition}
\newtheorem{example}[theorem]{Example}

\newtheorem{lemma}[theorem]{Lemma}

\newtheorem{proposition}[theorem]{Proposition}
\newtheorem{remark}[theorem]{Remark}

\makeatletter\@addtoreset{equation}{section}\makeatother

\DeclareFontFamily{U}{mathx}{\hyphenchar\font45}
\DeclareFontShape{U}{mathx}{m}{n}{
      <5> <6> <7> <8> <9> <10>
      <10.95> <12> <14.4> <17.28> <20.74> <24.88>
      mathx10
      }{}
\DeclareSymbolFont{mathx}{U}{mathx}{m}{n}
\DeclareFontSubstitution{U}{mathx}{m}{n}
\DeclareMathAccent{\widecheck}{0}{mathx}{"71}

\newcommand{\eex}{\hbox{}\hfill{\rule{.8ex}{.8ex}}}
\newcommand{\eremk}{\eex}
\newenvironment{numberedproof}[2][Proof]{\noindent \emph{#1 #2} }{\hfill \qed}
\newcommand{\BbbR}{\mathbb R}
\newcommand{\BbbC}{\mathbb C}
\newcommand{\BbbN}{\mathbb N}
\newcommand{\BbbZ}{\mathbb Z}
\def\XXint#1#2#3{{\setbox0=\hbox{$#1{#2#3}{\int}$ }
\vcenter{\hbox{$#2#3$ }}\kern-.6\wd0}}

\begin{document}

\ifarxiv
\title{A note on the shift theorem for the Laplacian in polygonal domains (extended version)}
\else 
\title{A note on the shift theorem for the Laplacian in polygonal domains}
\fi

%

\author{J.M.~Melenk, Wien, C.~Rojik, Wien}

%
%
%
\ifarxiv
\textsl{\vspace*{2mm} \begin{center} in memoriam Ivo M.~Babu{\v s}ka (1926--2023)\end{center}}
\fi

\abstract
We present a shift theorem for solutions of the Poisson equation 
in a finite planar cone (and hence also on plane polygons) for Dirichlet, Neumann, and mixed  boundary conditions. 
The range in which the shift theorem holds depends on the angle of the cone. For the right endpoint of the range, 
the shift theorem is described in terms of Besov spaces rather than Sobolev spaces.  
\endabstract

\keywords
Besov spaces, corner domains, corner singularities, Mellin calculus
\endkeywords

\subjclass
35J25, 
35B65
\endsubjclass

\thanks
JMM acknowledges support by the Austrian Science Fund (FWF) project \href{https://doi.org/10.55776/F65}{10.55776/F65} and CR support
by the FWF under project P 28367-N35. 
\endthanks


\section{Introduction}

The classical shift theorem for second order elliptic boundary value problems 
expresses the observation that the regularity of the solution $u$ is 
two Sobolev orders better than the right-hand side $f$. For example, 
for the Laplacian with Dirichlet boundary conditions and smooth domains,  
this shift theorem takes the form 
\begin{equation}
\label{eq:shift-theorem-intro} 
\text{ $f \in H^{-1+s}$ implies $u \in H^{1+s}$ } 
\end{equation}
for any $s \ge 0$, \cite{Evans_2010}, \cite[Chap.~2]{Grisvard_2011}. 
In 2D polygonal domains or even Lipschitz domains, the shift theorem
(\ref{eq:shift-theorem-intro}) is still valid, however, for a restricted  
range of values $s \in [0,s_0)$, where $s_0 = 1/2$ for Lipschitz domains \cite{savare98} and $s_0 > 1/2$ for 
polygonal $\Omega$ depends on the interior angles 
of $\Omega$, \cite{Grisvard_2011}. While the shift theorem does not hold 
in the limiting case $s = s_0$ in the scale of Sobolev spaces,  
we show that it holds in suitable Besov spaces. 

We prove the shift theorem in the limiting case $s = s_0$ using the 
well-known expansion of the solution in terms of singularity functions. For the purpose of exposition, consider a 
cone ${\mathcal C}$ with apex at the origin and angle $\omega > \pi$. 
Then, near the origin, a solution $u \in H^1({\mathcal C})$ of the Dirichlet problem can be written as 
\begin{equation}
\label{eq:expansion-intro}
u = S(f) s^+ \chi + u_0
\end{equation}
where $s^+$ is a known singularity function 
(see  (\ref{eq:s1D}), where $s^+ = s_1^D$), $\chi$ is a smooth cut-off function with $\chi \equiv 1$ near the origin, $u_0 \in H^2$ for 
$f \in L^2$, and $f \mapsto S(f)$ is a linear functional. We prove the shift
theorem in Besov spaces for the limiting case using three ingredients: 
\begin{enumerate*}[label=(\roman*)]
\item 
we assert that $s^+ \in B^{\alpha}_{2,\infty}$ for a Besov space $B^\alpha_{2,\infty}$; 
\item 
we show that $f \mapsto S(f)$ is a linear functional on a Besov space of the type $B^{\alpha'}_{2,1}$; 
\item 
we use the Mellin calculus to get a shift theorem for the mapping $f \mapsto u_0$.  
\end{enumerate*}

Shift theorems involving Besov spaces for the endpoints of the Sobolev scale have been 
shown to be appropriate in \cite[Thm.~{2}]{savare98}. For Lipschitz domains and Dirichlet conditions, 
it is shown that the solution $u$ of the Poisson problem $-\nabla \cdot (a \nabla u) = f$ (with sufficiently regular
elliptic $a$) satisfies $\|u\|_{B^{1+1/2}_{2,\infty}(\Omega)} \lesssim \|f\|_{B^{-1+1/2}_{2,1}(\Omega)}$. 
(A similar result holds for Neumann boundary conditions.) The 
proof relies on difference quotient techniques that are adapted to Dirichlet or Neumann conditions; an extension to 
mixed boundary conditions has to impose convexity conditions on the geometry, \cite{EF1999,Eb2002}. 
The endpoint result of the shift theorem of \cite[Thm.~{2}]{savare98} implies by interpolation the regularity result 
for the Poisson problem of \cite[Thms.~{1.1}, {1.3}]{jerison-kenig95}, which was obtained by a completely different method, 
namely, tools from harmonic analysis, although interpolation spaces are employed \textsl{en route}; 
these tools from harmonic analysis allow one to show shift theorems up to $1/2$ in scales of Sobolev spaces 
for the Dirichlet or Neumann Laplace problem (i.e., homogeneous right-hand side but inhomogeneous boundary conditions) 
on Lipschitz domains \emph{including} the endpoint $1/2$, \cite{jerison-kenig81,jerison-kenig82}. 

Moving from general Lipschitz domains to polygonal (in 2D) or polyhedral (in 3D) gives the solution more structure. 
A powerful way to describe the solution structure consists in expansions 
of the form (\ref{eq:expansion-intro}) and the Mellin calculus to derive these expansions. Expansions
in corner domains started with the seminal work on 2D corner domains in \cite{kondratiev-1967}; a comprehensive 
discussion  of the 2D case was achieved in \cite{Grisvard_2011,grisvard92}. 
Settings in $L^p$-spaces, the much more complex higher-dimensional cases and higher order equations
and even certain nonlinear equations were addressed 
in \cite{mazya-plamenevskii78,mazya-plamenevskii78-translated}, in 
\cite{dauge88} and 
\cite{costabel-stephan-1985, CostabelbiLaplace, Grisvard_2011, dauge88, nicaise93,kozlov-mazya-rossmann-1997, mazja-plamenevsky-1977, nazarov-plamenevsky-1994,mazya-rossmann2010}. Formulas for the linear functionals $f \mapsto S(f)$ alluded to above go back to the work by Maz'ya and Plamenvskii \cite{mazja-plamenevsky-1977}. 

Describing solutions in terms of expansions of the form (\ref{eq:expansion-intro}) leads to a further possible regularity theory
for solutions of elliptic boundary value problems in corner domains, namely, the use of weighted spaces, which has applications
to finite element approximation theory on graded meshes, \cite{babuska-kellogg-pitkaranta79}. 
While corner weighted spaces 
of finite regularity are a natural habitat of solutions and data in the framework of the Mellin calculus, weighted analytic regularity for 
problems in corner domains was developed by Babu{\v s}ka and Guo in \cite{babuska-guo88,babuska-guo89} for polygonal domains 
and by Costabel, Dauge, and Nicaise in \cite{costabel-dauge-nicaise12} for polyhedra.

Elliptic shift theorems in Besov spaces have been derived 
in \cite{dahlke-devore97,dahlke-1999} with a view to characterize optimal
convergence rates for adaptive numerical methods. Our present focus on
the limiting case $s = s_0$ is close to the works 
\cite{bacuta-diss, bacuta-bramble-2002, bacuta-bramble-xu-2003}. 
Indeed, \cite{bacuta-bramble-xu-2003} obtains the same shift theorem 
as we do but effectively restricts the attention to convex domains 
with one corner with an interior angles between $\pi/2$ and $\pi$; 
 \cite{bacuta-bramble-2002} restricts to 
non-convex domains and right-hand sides $f$ in a Besov space that 
is the interpolation space between $H^{-1}$ and a subspace of $L^2$ of co-dimension $1$. 
In the present work, by analyzing the singularity function $s^+$ and the associated linear functional $f \mapsto S(f)$ 
in (\ref{eq:expansion-intro}), we are able to lift these restrictions of \cite{bacuta-bramble-2002,bacuta-bramble-xu-2003}
and show in Theorem~\ref{thm:shift-theorem-local-version} a local shift theorem near a corner without restrictions on the 
interior angle in the framework of standard Besov spaces.  Additionally, we explicitly consider Dirichlet, Neumann, and 
mixed boundary conditions. 

Our proof of the limiting case of the shift theorem relies on expansions in singularity functions and rather explicit
formulas for the stress intensity functions. Extensions to 3D might be possible for geometries with point singularities; 
the presence of edge singularities would require new tools.

Our main result, Theorem~\ref{thm:shift-theorem-local-version} is formulated in terms of $L^2$-based Besov spaces. 
Besov spaces based on $L^p$-spaces can alternatively be considered. In Section~\ref{sec:Lp} we briefly indicate 
that endpoint results in such $L^p$-based Besov spaces can be obtained by the same approach. 

\subsection{Notation}
\subsubsection{Interpolation spaces}
For Banach spaces 
$(X_0,\|\cdot\|_{X_0})$, $(X_1,\|\cdot\|_{X_1})$ with continuous embedding $X_1 \subset X_0$
and $\theta \in (0,1)$, $q \in [1,\infty]$, 
we define with the so-called ``real method''/``$K$-method'' the interpolation spaces 
$X_{\theta,q}:= (X_0,X_1)_{\theta,q} 
:= \{u \in X_0\,|\,  \|u\|_{(X_0,X_1)_{\theta,q}} < \infty\}$, 
where the norm 
$\|u\|_{X_{\theta,q}}:= \|u\|_{(X_0,X_1)_{\theta,q}}$ is given by  
\begin{equation}
\label{eq:interpolation-norm}
\|u\|_{X_{\theta,q}}:= \|u\|_{(X_0,X_1)_{\theta,q}} := 
\begin{cases}
\left( \int_{t=0}^\infty \left(t^{-\theta} K(t,u)  \right)^q \frac{dt}{t} \right)^{1/q}, & q \in [1,\infty)  \\
 \sup_{t > 0} t^{-\theta} K(t,u) & q = \infty
\end{cases}
\end{equation}
with the $K$-functional 
\begin{equation*}
K(t,u) = \inf_{v \in X_1} \|u - v\|_{X_0} + t \|v\|_{X_1} . 
\end{equation*} 
We refer to \cite{Mclean00,Tartar_2007,Triebel_2002} for discussions of interpolation spaces. 
We have the continuous embedding $X_{\theta,q} \subset X_{\theta',q'}$ if $\theta > \theta'$ ($q$, $q'$ arbitrary)
or $\theta = \theta'$ and $q \leq q'$. 
We highlight that in the present case of $X_1 \subset X_0$, the integral $\int_{0}^\infty$ 
in (\ref{eq:interpolation-norm}) can actually be replaced with the finite integral $\int_0^1$, 
\cite[Chap.~6, Sec.~7]{DEVore-Lorentz-93}. An important property of interpolation spaces is the Reiteration Theorem, \cite[Thm.~{26.3}]{Tartar_2007}, 
which states that for $0 \leq \theta_1 < \theta_2 \leq 1$ and arbitrary $\theta \in (0,1)$, $q_1$, $q_2$, $q \in [1,\infty]$
one has (with norm equivalence) 
$(X_{\theta_1,q_1}, X_{\theta_2,q_2})_{\theta_1 (1-\theta) + \theta_2 \theta,q} = X_{\theta,q}$.

\subsubsection{Sobolev and Besov spaces}
For domains $D \subset \BbbR^d$, $d \in \{1,2\}$, we employ the usual Sobolev spaces $H^s(D)$ 
and $\widetilde{H}^s(D)$ for $s \in \BbbR$ as described in, e.g., \cite{Mclean00} or \cite{Triebel2ndEd}. 
To be specific and following \cite{Mclean00}, with the space ${\mathcal S}^\star$ of tempered distributions and the Fourier transformation ${\mathcal F}$, 
the spaces $H^s(\BbbR^d)$ are given by $ H^s(\BbbR^d) = \{u \in {\mathcal S}^\star \colon \|u\|^2_{H^s(\BbbR^d)}:=\int_{\xi \in \BbbR^d} (1 + |\xi|^2)^s |{\mathcal F} u|^2 \, d\xi < \infty\}$. 
We set $H^s(D)  := \{ u \in {\mathcal D}^\star(D)\colon u = U|_D \text{ for some $U \in H^s(\BbbR^d)$}\}$ with the norm 
$\|u\|_{H^s(D)}:= \inf\{ \|U\|_{H^s(\BbbR^d)}\,|\, U|_D = u\}$, where ${\mathcal D}^\star(D)$ denotes the space of distributions on $D$. 
We set $\widetilde{H}^s(D):= \{u \in H^s(\BbbR^d)\,|\,  \operatorname{supp} u \subset \overline{D}\}$
with norm $\|u\|_{\widetilde{H}^s(D)}=  \|u\|_{H^s(\BbbR^d)}$. (The space $\widetilde{H}^s(D)$ is denoted $H^s_{\overline{D}}$ in \cite[p.~76]{Mclean00} but 
coincides with the space $\widetilde{H}^s(D)$ defined in \cite[p.~77]{Mclean00} by \cite[Thm.~{3.29}]{Mclean00}.) An important relation of these 
spaces is the duality relation \cite[Thm.~{3.30}]{Mclean00}
\begin{align*}
\widetilde{H}^{-s}(D) &= \left(H^s(D) \right)^\star, 
&
H^{-s}(D) &= \left(\widetilde{H}^s(D) \right)^\star, 
\qquad s \in \BbbR. 
\end{align*}
Furthermore, one has $H^0(D) = \widetilde{H}^0(D) = L^2(D)$ and, 
by \cite[Cor.~{1.4.4.5}]{Grisvard_2011} for $s \in (0,1/2)$ and by duality for $s \in (-1/2,0)$, 
\begin{align}
\label{eq:Hs=tildeHs}
H^s(D) = \widetilde{H}^s(D), \quad |s| < 1/2. 
\end{align}

The two scales $H^s(D)$, $\widetilde{H}^s(D)$, $s \in \BbbR$, of Sobolev spaces are scales of interpolation spaces: By \cite[Thms.~{B.8}, {B.9}]{Mclean00} 
we have for $s_1$, $s_2 \in \BbbR$, $\theta \in (0,1)$, 
\begin{align}
\label{eq:sobolev-space-interpolation-scale}
(H^{s_1}(D), H^{s_2}(D))_{\theta,2} &= H^{(1-\theta) s_1 + \theta s_2}(D) , 
&
(\widetilde{H}^{s_1}(D), \widetilde{H}^{s_2}(D))_{\theta,2} &= \widetilde{H}^{(1-\theta) s_1 + \theta s_2}(D). 
\end{align}
The scales of Besov spaces $B^s_{2,q}(D)$ and $\widetilde{B}^s_{2,q}(D)$ are defined by interpolating between Sobolev spaces: given $s \in \BbbR$, select
$s_1 < s < s_2$ and set 
\begin{align}
\label{eq:besov-space}
B^s_{2,q}(D)&:= (H^{s_1}(D), H^{s_2}(D))_{\theta,q}, 
&
\widetilde{B}^s_{2,q}(D)&:= (\widetilde{H}^{s_1}(D), \widetilde{H}^{s_2}(D))_{\theta,q}, 
\qquad \theta:= \frac{s - s_1}{s_2 - s_1}. 
\end{align}
The Reiteration Theorem \cite[Thm.~{26.3}]{Tartar_2007} asserts that the precise choice of $s_1$, $s_2$ is immaterial. 
(\ref{eq:Hs=tildeHs}) also implies 
\begin{equation}
\label{eq:B=tildeB}
B^s_{2,q}(D)=  \widetilde{B}^s_{2,q}(D), 
\qquad |s| < 1/2, \quad q \in [1,\infty]. 
\end{equation}

\subsection{Setting and main results}
We study the regularity of solutions of elliptic problems
in a cone. For an angle $\omega \in (0,2\pi)$ we therefore introduce 
in polar coordinates\footnote{throughout, we will freely identify points $\mathbf{x} = (x,y) \in \BbbR^2$ 
either in Cartesian or polar coordinates $(r,\phi)$} $(r,\phi)$ 
the cone $\mathcal{C}$  and the truncated cones $\mathcal{C}_R$ by 
\begin{align} 
\label{eq:cone} 
\mathcal{C}&:= \{(r\cos \phi,r \sin\phi)\,|\,  r > 0, \phi \in G\}, 
\qquad G:= (0,\omega), \\
\mathcal{C}_R &:= \mathcal{C} \cap B_R(0), 
\end{align} 
where $B_r(0)\subset \BbbR^2$ denotes the (open) ball of radius $r > 0$ centered at $0$. 
The two lateral sides of $\mathcal{C}$ are $\Gamma_0 = \{(r,0)\,|\,  r > 0\}$ 
and $\Gamma_\omega = \{(r \cos \omega,r \sin\omega)\,|\,  r > 0\}$. The three
boundary parts of $\mathcal{C}_R$ are 
$\Gamma_{0,R} = \Gamma_0 \cap B_R(0)$, 
$\Gamma_{\omega,R} = \Gamma_\omega \cap B_R(0)$, and  
$\widetilde \Gamma_R:= \{(R\cos \phi,R\sin\phi)\,|\,  \phi \in (0,\omega)\}$. 
We consider 
$H^1(\mathcal{C}_R)$-functions $u$ that satisfy 
\begin{subequations}
\label{eq:problem}
\begin{align}
\label{eq:problem-a}
- \Delta u & = f \quad \mbox{ in $\mathcal {C}_R$}, \\
\label{eq:problem-b}
 u & = 0 \quad \mbox{ on $\Gamma_D$}, \\
\label{eq:problem-c}
 \partial_n u & = 0 \quad \mbox{ on $\Gamma_N$}. 
\end{align}
\end{subequations}
Concerning the boundary conditions, we consider three cases: 
\begin{itemize}
\item 
\emph{Dirichlet case}:  
$\Gamma_D = \Gamma_{0,R} \cup \Gamma_{\omega,R}$ and $\Gamma_N = \emptyset$; 
\item\emph{Neumann case}: 
$\Gamma_N = \Gamma_{0,R} \cup \Gamma_{\omega,R}$ and $\Gamma_D = \emptyset$; 
\item\emph{mixed case}: 
$\Gamma_D = \Gamma_{0,R}$ and $\Gamma_N = \Gamma_{\omega,R}$. 
\end{itemize} 
Equation (\ref{eq:problem}) is understood in a weak sense. That is, 
we define the space 
$\displaystyle H^1_D(\mathcal{C}_R):= \{v \in H^1(\mathcal{C}_R) \,|\,  v|_{\widetilde \Gamma_R \cup \Gamma_D} = 0\} 
$
and its dual $H^{-1}_D(\mathcal{C}_R):= \left(H^1_D(\mathcal{C}_R)\right)^\star$. 
The minimal regularity assumption for (\ref{eq:problem}) is $f \in H^{-1}_D(\mathcal{C}_R)$. 
Then, 
$u \in H^1(\mathcal{C}_R)$ solves (\ref{eq:problem}) if 
$u|_{\Gamma_D} = 0$ (in the sense of traces) and 
the equations (\ref{eq:problem-a}), (\ref{eq:problem-c}) are satisfied in a weak sense, i.e., 
\begin{equation*}
\label{eq:problem-weak} 
\int_{\mathcal{C}_R} \nabla u \cdot \nabla v = \langle f, v\rangle_{H^{-1}_D(\mathcal{C}_R) \times H^1_D(\mathcal{C}_R)} 
\qquad \forall v \in H^1_D(\mathcal{C}_R). 
\end{equation*}
For solutions $u$ of (\ref{eq:problem}), we have the following result: 
\begin{theorem}[Shift theorem, Besov spaces]
\label{thm:shift-theorem-local-version}
Let $\omega \in (0,2\pi)$. 
Fix $0 < R'< R$. Let $f\in H_D^{-1}(\mathcal{C}_R)$, and let $\chi_{R} \in C^\infty_0(B_R(0))$
with $\chi_{R} \equiv 1$ on $B_{R'}(0)$. 
Then for a solution $u \in H^1(\mathcal{C}_R)$ of (\ref{eq:problem}) the following statements hold with implied constants 
depending only on $\omega$, $R$, $R'$, and $\chi_R$: 
\begin{enumerate}[label=(\roman*)]
\item \label{item:thm:shift-theorem-local-version-dirichlet} Dirichlet case: 
For $\chi_{R} f\in B_{2,1}^{\frac{\pi}{\omega}-1}(\mathcal{C}_R)$ one has 
$u\in B_{2,\infty}^{\frac{\pi}{\omega}+1}(\mathcal{C}_{R'})$ with the estimate
\begin{align}
\label{eq:thm:shift-theorem-local-version-ii}
\|u\|_{B_{2,\infty}^{\frac{\pi}{\omega}+1}(\mathcal{C}_{R'})} \lesssim \|\chi_R f\|_{B_{2,1}^{\frac{\pi}{\omega}-1}(\mathcal{C}_R)} + \|u\|_{H^1(\mathcal{C}_R)}.
\end{align}
\item \label{item:thm:shift-theorem-local-version-neumann} Neumann case:  
For $\chi_{R} f\in {B}_{2,1}^{\frac{\pi}{\omega}-1}(\mathcal{C}_R)$ 
one has 
$u\in B_{2,\infty}^{\pi/\omega+1}(\mathcal{C}_{R'})$ with the estimate
\begin{align}
\label{eq:thm:shift-theorem-local-version-iii}
\|u\|_{B_{2,\infty}^{\frac{\pi}{\omega}+1}(\mathcal{C}_{R'})} \lesssim 
\|\chi_R f\|_{{B}_{2,1}^{\frac{\pi}{\omega}-1}(\mathcal{C}_R)} + \|u\|_{H^1(\mathcal{C}_R)}. 
\end{align}
\item \label{item:thm:shift-theorem-local-version-mixed} Mixed case:  
If $\chi_{R} f\in \widetilde{B}_{2,1}^{\frac{\pi}{2\omega}-1}(\mathcal{C}_R)$ (if $\omega \ge \pi/2$) or 
$\chi_{R} f\in {B}_{2,1}^{\frac{\pi}{2\omega}-1}(\mathcal{C}_R)$ (if $\omega < \pi/2$) 
one has 
$u\in B_{2,\infty}^{1+\frac{\pi}{2\omega}}(\mathcal{C}_{R'})$ with the estimate
\begin{align}
\label{eq:thm:shift-theorem-local-version-iv}
\|u\|_{B_{2,\infty}^{1+\frac{\pi}{2\omega}}(\mathcal{C}_{R'})} \lesssim 
\begin{cases} 
\|\chi_{R} f\|_{B^{-1+\frac{\pi}{2\omega}}_{2,1}(\mathcal{C}_R)} + \|u\|_{H^1(\mathcal{C}_R)} & \text{if $\omega <\pi/2$}  \\
\|\chi_{R} f\|_{\widetilde{B}^{-1+\frac{\pi}{2\omega}}_{2,1}(\mathcal{C}_R)} + \|u\|_{H^1(\mathcal{C}_R)} & \text{if $\omega \ge \pi/2$}.  
\end{cases}
\end{align}
\end{enumerate}
\end{theorem}

\begin{proof}
Item \ref{item:thm:shift-theorem-local-version-dirichlet} is shown 
in Section~\ref{sec:dirichlet},  
item \ref{item:thm:shift-theorem-local-version-neumann} is discussed in 
Section~\ref{sec:neumann}, and 
item \ref{item:thm:shift-theorem-local-version-mixed} in Section~\ref{sec:mixed}. 
\end{proof}
\begin{remark}
\label{rem:omega=pi}
The cases $\omega = \pi$ for Dirichlet and Neumann boundary conditions can be sharpened. 
This case corresponds to a smooth geometry
so that by standard elliptic regularity theory \cite{Evans_2010, Gilbarg} the solution is as smooth as the right-hand side $f$ permits near the origin, 
i.e., one has estimates of the form 
$$
\|u\|_{H^{1+s}(\mathcal{C}_{R'})} \lesssim \|\chi_R f\|_{H^{-1+s}(\mathcal{C}_R)} + \|u\|_{H^1(\mathcal{C}_R)} 
$$
for all $s \ge 0$. The implied constant additionally depends on $s$. 
\eremk
\end{remark}
Theorem~\ref{thm:shift-theorem-local-version} discusses a limiting case of the shift theorem. With similar techniques as those
used in the proof of Theorem~\ref{thm:shift-theorem-local-version}, one can show a shift theorem in a range of regularity indices: 
\begin{corollary}
\label{cor:shift-theorem} 
Assume the hypotheses and notation of Theorem~\ref{thm:shift-theorem-local-version}. 
\begin{enumerate}[label=(\roman*)]
\item Dirichlet case: A solution $u \in H^1(\mathcal{C}_R)$ of (\ref{eq:problem}) 
satisfies for $0 < s < \pi/\omega$ and $q \in [1,\infty]$ 
\begin{equation}
\label{eq:cor:shift-theorem-dirichlet} 
\|u\|_{B^{s +1}_{2,q}({\mathcal C}_{R'})} \lesssim \|\chi_R f\|_{B^{s-1}_{2,q}({\mathcal C}_R)} + \|u\|_{H^1({\mathcal C}_R)}. 
\end{equation}
\item Neumann case: A solution $u \in H^1(\mathcal{C}_R)$ of (\ref{eq:problem}) 
satisfies for $0 < s < \pi/\omega$ and $q \in [1,\infty]$ 
\begin{equation}
\label{eq:cor:shift-theorem-neumann} 
\|u\|_{B^{s +1}_{2,q}({\mathcal C}_{R'})} \lesssim 
\begin{cases} \|\chi_R f\|_{\widetilde{B}^{s-1}_{2,q}({\mathcal C}_R)} + \|u\|_{H^1({\mathcal C}_R)} & \mbox{ if $s < 1$} \\
              \|\chi_R f\|_{{B}^{s-1}_{2,q}({\mathcal C}_R)} + \|u\|_{H^1({\mathcal C}_R)} & \mbox{ if $s \ge 1$}.
\end{cases}
\end{equation}
\item Mixed case: A solution $u \in H^1(\mathcal{C}_R)$ of (\ref{eq:problem}) 
satisfies for $0 < s < \pi/(2\omega)$ and $q \in [1,\infty]$ 
\begin{equation}
\label{eq:cor:shift-theorem-mixed} 
\|u\|_{B^{s +1}_{2,q}({\mathcal C}_{R'})} \lesssim 
\begin{cases} \|\chi_R f\|_{\widetilde{B}^{s-1}_{2,q}({\mathcal C}_R)} + \|u\|_{H^1({\mathcal C}_R)} & \mbox{ if $s < 1$} \\
              \|\chi_R f\|_{{B}^{s-1}_{2,q}({\mathcal C}_R)} + \|u\|_{H^1({\mathcal C}_R)} & \mbox{ if $s \ge 1$}.
\end{cases} 
\end{equation}
\end{enumerate}
\end{corollary}
\begin{proof}
The result follows by inspection of the proof of Theorem~\ref{thm:shift-theorem-local-version}. For example, for the case of Dirichlet 
conditions, the proof rests on two ingredients: 
\begin{enumerate*}[label=(\alph*)]
\item the shift theorem for the operator $\widetilde{T}$ of (\ref{eq:operator-T}) and
\item the estimate of the function $\widetilde{f}$ in (\ref{eq:thm:shift-theorem-local-version-10}).  
\end{enumerate*}
The operator $\widetilde{T}$ of (\ref{eq:operator-T}) is directly amenable to interpolation arguments as it maps
$H^{-1} \rightarrow H^1$ and, by (\ref{eq:dirichlet-case-goal}), $B^{\pi/\omega-1}_{2,1} \rightarrow B^{\pi/\omega+1}_{2,\infty}$. 
Inspection of the proof of (\ref{eq:thm:shift-theorem-local-version-10}) leads to having to control
$\|\nabla \chi_R \cdot \nabla u \|_{B^{-1+s}_{2,q}}$  and $\|\Delta \chi_R u \|_{B^{-1+s}_{2,q}}$. These terms can be estimated 
with Lemma~\ref{lemma:regularity-lemma-annuli}. 

For the Neumann case (and similarly for the mixed case), the analysis is also reduced to understanding the mapping properties
of the corresponding operator $\widetilde{T}$. If $\omega > \pi$ (i.e., $\pi/\omega - 1 \in (-1/2,0)$) one observes that 
$\widetilde{B}^{\pi/\omega-1}_{2,1} = B^{\pi/\omega-1}_{2,1}$ so that one has by (\ref{eq:neumann-case-goal}) the mapping properties
$\widetilde{T}: \widetilde{H}^{-1} \rightarrow H^1$ and $\widetilde{T}: \widetilde{B}^{\pi/\omega-1}_{2,1} \rightarrow B^{\pi/\omega+1}_{2,\infty}$. 
An interpolation argument as in the Dirichlet case concludes the argument. 
If $\omega < \pi$, one splits the argument into two interpolation steps. First, one observes from 
Corollary~\ref{cor:solution-with-polynomial-neumann} for $k = 0$ and $\epsilon \in( 0,1/2)$ 
sufficiently small that $\widetilde{T}:H^\epsilon \rightarrow H^{2+\epsilon}$. Hence, 
$\widetilde{T}:\widetilde{H}^{-1} \rightarrow H^1$ and 
$\widetilde{T}:\widetilde{H}^{\epsilon} = H^\epsilon \rightarrow H^{2+\epsilon}$, which provides the desired result for $s \in (0,1+\epsilon)$ 
by interpolation. 
For $s \in (1,\pi/\omega)$, one interpolates using the mapping properties 
$\widetilde{T}:H^\epsilon \rightarrow H^{2+\epsilon}$ and $\widetilde{T}: B^{-1 + \pi/\omega}_{2,1} \rightarrow B^{1+\pi/\omega}_{2,\infty}$ 
provided by (\ref{eq:neumann-case-goal}). 
\end{proof} 
\section{Dirichlet boundary conditions}
\label{sec:dirichlet}
We start with introducing corner-weighted functions that are useful in connection with Mellin transform techniques: 

\begin{definition}[weighted spaces]
For $s\in\mathbb{N}_0$ and $\gamma\in\mathbb{R}$, we define
\begin{align*}
K_\gamma^s(\mathcal{C}):=\{u\in L_{loc}^2(\mathcal{C})\,|\, 
\|u\|_{K_\gamma^s(\mathcal{C})}^2 := \sum_{|\alpha|\leq s} \|r^{|\alpha|-s+\gamma} D^\alpha u\|_{L^2(\mathcal{C})}^2 < \infty, \quad |\alpha| \leq s\}. 
\end{align*}
The spaces $K_\gamma^s(\mathcal{C}_R)$ are defined in the same way, just by replacing $\mathcal{C}$ by $\mathcal{C}_R$.
\end{definition}

Fractional order Sobolev spaces of functions that are constrained to vanish to a certain order at the origin 
are shown in the following Lemma~\ref{lemma:Sobolev-to-cone-second-version} to be subspaces of suitable weighted Sobolev spaces 
of the $K_\gamma^s$-type; 
similar estimate with a focus on integer order Sobolev spaces 
are well-known in the literature, see., e.g., \cite[Chap.~{7.1}]{kozlov-mazya-rossmann-1997}. 

\begin{lemma}
\label{lemma:Sobolev-to-cone-second-version}
Let $f\in H^{k+\epsilon}(\mathcal{C})$ with $\operatorname{supp} f \subset B_1(0)$ for some $k\in\mathbb{N}_0$ and $\epsilon\in (0,1)$ and assume 
$\partial_x^i \partial_y^j f(0)=0$ for $i+j\leq k-1$. Then $f\in K_{-\epsilon}^k(\mathcal{C})$ with the norm estimate
\begin{align*}
\|f\|_{K_{-\epsilon}^k(\mathcal{C})} \lesssim \|f\|_{H^{k+\epsilon}(\mathcal{C}_1)}.
\end{align*}
\end{lemma}
\begin{proof}
See Appendix~\ref{sec:inequalities}. Note that $H^{k+\epsilon} \subset C^{k}$ by Sobolev embedding. 
\end{proof}

\subsection{A recap of regularity based on the Mellin calculus}
\label{sec:recap}
\subsubsection{Preliminaries}
The following properties of the Mellin transformation are at the heart of the analysis of \cite{kondratiev-1967,dauge88,kozlov-mazya-rossmann-1997, mazja-plamenevsky-1977, nazarov-plamenevsky-1994,mazya-rossmann2010} and are collected
in \cite[Sec.~{3}]{costabel-dauge-nicaise10}; we also refer to \cite[Chap.~3]{rojik19} 
for detailed proofs. 
\ifarxiv
We also point to Appendix~\ref{appendix:L2} for a rather self-contained proof of the results about the Mellin calculus presented
in this section. 
\fi

For a sufficiently regular function $u$ on the cone ${\mathcal C}$, we define its Mellin transform ${\mathcal M}[u]$ by 
\begin{equation}
\label{eq:mellin-transform}
{\mathcal M}[u](\zeta,\phi):= \frac{1}{\sqrt{2\pi}} \int_{r=0}^\infty r^{-i \zeta} \widetilde u(r,\phi)\, \frac{dr}{r}, 
\end{equation}
where $\widetilde u(r,\phi)  = u(r \cos \phi, r\sin \phi)$, i.e., the representation of $u$ in polar coordinates. 
We emphasize, however, that henceforth we will write $u$ for the function both in Cartesian and polar coordinates.  
The Mellin transformation is connected 
to the Fourier transformation in that one has with the change of variables $r = e^{t}$ 
\begin{equation*}
{\mathcal M} [u](\zeta,\phi) = \frac{1}{\sqrt{2\pi}} \int_{t=-\infty}^\infty e^{-i \zeta t} u(e^t,\phi)\, dt. 
\end{equation*}
This connection with the Fourier transformation is at the heart of the following norm equivalence: if $u(\cdot, \phi) \in L^2(0,\infty)$, then 
${\mathcal M}[u](\cdot,\phi)$ is in $L^2(\BbbR - i)$ with equivalent norms, and the inverse Mellin transformation correspondingly takes the form 
\begin{equation*}
u = \frac{1}{\sqrt{2\pi}} \int_{\operatorname{Im} \zeta = -i} r^{i \zeta} {\mathcal M} [u](\zeta,\phi)\, d\zeta. 
\end{equation*}
More generally, one has for $k \in \BbbN_0$ and $\gamma \in \BbbR$ the norm equivalence 
\begin{align*}
\|u\|^2_{K^k_\gamma({\mathcal C})} &\sim \int_{\xi \in \BbbR} \|{\mathcal M}[u](\xi - i\eta)\|^2_{H^k(G; |\xi|)}\, d\xi, 
\qquad \|v\|^2_{H^k(G; |\xi|)}:= \sum_{j\leq k} (1 + |\xi|^2)^{k-j} \|\partial^j_{\phi} v\|^2_{L^2(G)}, 
\quad \eta:= k - \gamma  -1, 
\end{align*}
where we view the Mellin transformation, which acts only on the variable $r$ (with dual variable $\zeta$), 
as a mapping from $K^k_\gamma({\mathcal C})$ into a space of $H^k(G)$-valued functions. 
A final important property of the Mellin transformation is that if $u \in K^k_\gamma({\mathcal C})$ satisfies additionally 
$\operatorname{supp} u \subset B_1(0)$, then, by a variant of the Paley-Wiener Theorem, 
${\mathcal M}[u]$ is actually holomorphic on $\{z \in \BbbC\,|\, \operatorname{Im} z > -\eta\}$. This property
allows one to use the Cauchy integral theorem/residue theorem, whose use leads to expansions in terms of corner singularity functions. 

\subsubsection{The isomorphism in weighted spaces and expansion in corner singularity functions}
\label{sec:regularity-in-weighted-spaces}
Let $k \in \BbbN_0$ and $\epsilon \in (0,1)$. 
Consider $f\in H^{k+\epsilon}(\mathcal{C})$ with $\operatorname{supp} f\subseteq B_1(0)$ and $\partial_x^i \partial_y^j f(0)=0$ for $i+j\leq k-1$. 
For convenience, we assume ${k+1+\epsilon} < 2 \frac{\pi}{\omega}$. Note that by Lemma~\ref{lemma:Sobolev-to-cone-second-version} the function $f \in K_{-\epsilon}^{k}(\mathcal{C})$. Assume that $u_1\in H^1(\mathcal{C})$ with $\operatorname{supp} u_1 \subseteq B_1(0)$ solves the problem
\begin{align}
\label{eq:elliptic2.25}
-\Delta u_1 &= f \in H^{k+\epsilon}(\mathcal{C}), 
\qquad 
u_1 = 0 \ \mbox{ on $\Gamma_0$ and $\Gamma_\omega$.}
\end{align}
Further we pose the auxiliary problem
\begin{align}
\label{eq:elliptic}
-\Delta u_0 &= f \in K_{-\epsilon}^k(\mathcal{C}), 
\qquad 
u_0 = 0 \ \mbox{ on $\Gamma_0$ and $\Gamma_\omega$.}
\end{align}
This latter problem admits a unique solution $u_0 \in K^{k+2}_{-\epsilon}({\mathcal C})$ by the Mellin calculus going back to 
\cite{kondratiev-1967}
with the norm estimate $\|u_0\|_{K^{k+2}_{-\epsilon}({\mathcal C})} \lesssim \|f\|_{K^k_{-\epsilon}(\mathcal{C})}$ 
(see, e.g., 
\cite[Sec.~{6.1.8}]{kozlov-mazya-rossmann-1997}, \cite{dauge88}, or \cite{rojik19} for more details).
By elliptic regularity based on a dyadic decomposition of the cone ${\mathcal C}$ (see Lemma~\ref{lemma:scaling-arguments-cone}), one actually has 
$u_0\in H^{k+2+\epsilon}(\mathcal{C}_R)$ with the estimate $\|u_0\|_{H^{k+2+\epsilon}({\mathcal C}_R)} \lesssim \|f\|_{H^{k+\epsilon}(\mathcal{C})}$
for each fixed $R > 1$. 
Following the classical path, we now analyze the relation of the solutions $u_0$ and $u_1$.
The Mellin transformation yields
\begin{align}
\label{eq:Mellin-transform-1}
\mathcal{L}(\zeta)\mathcal{M}[u_1]& =\mathcal{M}[g] \quad \text{on } \{\zeta\in\mathbb{C}:\operatorname{Im}\zeta > 0\} , \\
\label{eq:Mellin-transform-2}
\mathcal{L}(\zeta)\mathcal{M}[u_0]& =\mathcal{M}[g] \quad \text{on } \{\zeta\in\mathbb{C}:\operatorname{Im}\zeta = -1-k-\epsilon\}
\end{align}
with the operator $\mathcal{L}(\zeta) := (-\partial_\phi^2+\zeta^2)$ and $\mathcal{M}[g]$ 
being the Mellin transform of the function $g=r^2 f$. 
Note that $\mathcal{M}[g]$ is holomorphic on $\{\zeta\in\mathbb{C}:\operatorname{Im}\zeta > -1-k-\epsilon\}$ with values in $H^{k}(G)$ 
and that $\mathcal{M}[u_1]$ is holomorphic on $\{\zeta\in\mathbb{C}:\operatorname{Im}\zeta > 0\}$ with values in $H^2(G)$. 
Since the operator $(\mathcal{L}(\zeta))^{-1}$ is meromorphic on $\mathbb{C}$ with poles at the discrete set 
\begin{equation}
\pm i \sigma^D \quad \mbox{ with }  \sigma^D:= \{\lambda^D_n \,|\, n \in \BbbN\}, \qquad \lambda^D_n:= n \frac{\pi}{\omega}, 
\end{equation}
we observe that $\mathcal{M}[u_1]$ can be extended meromorphically to $\{\zeta\in\mathbb{C}:\operatorname{Im}\zeta>-1-k-\epsilon\}$ by
\begin{align*}
U(\zeta):=\mathcal{M}[u_1](\zeta):=\left(\mathcal{L}(\zeta)\right)^{-1}\mathcal{M}[g](\zeta).
\end{align*}
Let us mention that $U(\zeta)$ and $\mathcal{M}[u_0](\zeta)$ coincide on $\{\zeta \in \BbbC\,|\,  \operatorname{Im} \zeta= -1-k-\epsilon\}$, 
as well as $U(\zeta)$ and $\mathcal{M}u_1(\zeta)$ on $\{\operatorname{Im} \zeta = 0\}$. 
Inverse transformations and Residue Theorem then lead to
\begin{align}
\label{eq:residue-dirichlet-higher-order0.5}
u_0-u_1 = \sum_{\substack{\zeta_0\in -i\sigma^D \colon\\ \operatorname{Im}\zeta_0\in (-1-k-\epsilon,0)}} \frac{2\pi i}{\sqrt{2\pi}} \operatorname*{Res}_{\zeta = \zeta_0} \left(r^{i\zeta}\left(\mathcal{L}(\zeta)\right)^{-1} \mathcal{M}[g](\zeta)\right).
\end{align}
Since we assumed $k+1+\epsilon < 2 \frac{\pi}{\omega}$, the sum in \eqref{eq:residue-dirichlet-higher-order0.5} has at most one term. Determining the residue yields
\begin{align*}
u_1 = \begin{cases}
u_0, & \mbox{ if } \quad k+1+\epsilon < \frac{\pi}{\omega} \\
u_0 - \frac{1}{\pi} \left(\int_{\mathcal{C}} r^{-\lambda_1} \sin(\lambda_1\phi) f(x) \, dx\right) r^{\lambda_1} \sin(\lambda_1\phi), & \mbox{ if } \quad 1+k+\epsilon > \frac{\pi}{\omega} 
\end{cases}
\end{align*}
with $u_0\in H^{k+2+\epsilon}(\mathcal{C}_R)$ for any chosen $R > 0$. 
These observations are collected in the following result, cf.\ \cite[Sec.~{6.1.8}]{kozlov-mazya-rossmann-1997}.
\begin{proposition}
\label{prop:solution_u1}
Let $R > 0$. 
Let $k\in\mathbb{N}_0$ and $\epsilon\in (0,1)$ satisfy $k+1+\epsilon < \lambda^D_2 = 2 \frac{\pi}{\omega}$ and $k+1+\epsilon \ne \lambda^D_1 = \frac{\pi}{\omega}$. 
Let $f\in H^{k+\epsilon}(\mathcal{C})$ with $\operatorname{supp} f\subseteq B_1(0)$. Further assume $\partial_x^i \partial_y^j f(0)=0$ for $i+j \leq k-1$. Then $u_1\in H^1(\mathcal{C})$ with $\operatorname{supp} u_1 \subseteq B_1(0)$ solving
\begin{align}
-\Delta u_1 &= f \in H^{k+\epsilon}(\mathcal{C}), 
\qquad 
u_1 = 0 \mbox{ for $\phi\in \{0,\omega\}$}, 
\end{align}
has the form 
\begin{align}
\label{eq:prop_solution_u1}
u_1 = 
\begin{cases} 
u_0, & \mbox{ if } \quad \frac{\pi}{\omega} = \lambda^D_1 > k+\epsilon+1 \\ 
u_0 - \frac{1}{\pi} \left(\int_{\mathcal{C}} r^{-\lambda^D_1} \sin(\lambda^D_1\phi) f(\mathbf{x}) \, d\mathbf{x}\right) r^{\lambda^D_1} \sin(\lambda^D_1\phi), & 
\mbox{ if } \quad \frac{\pi}{\omega} = \lambda^D_1 < k+1+\epsilon < \lambda^D_2
\end{cases} 
\end{align}
for a $u_0\in H^{k+2+\epsilon}(\mathcal{C}_R )$ with the estimate
\begin{align*}
\|u_0\|_{H^{k+2+\epsilon}(\mathcal{C}_R )} \lesssim \|f\|_{H^{k+\epsilon}(\mathcal{C}_1)}.
\end{align*}
\end{proposition}
The next lemma allows us to remove from Proposition~\ref{prop:solution_u1} 
the condition that $f$ vanish to order $k-1$ at zero. 
\begin{lemma}
\label{lemma:polynomial-solution}
Let $i$, $j$, $k\in\mathbb{N}_0$ with $i+j=k$. Set $\Sigma^D_{k+2}:=\{n \in \{2,\ldots, k+2\}\,|\,  n \frac{\omega}{\pi} \in\BbbN\}$ 
and $S^D_{k+2}:= \operatorname{span} \{ r^n \left(\ln r \sin (n \phi) + \phi \cos (n \phi)\right)\,|\,  n \in \Sigma^D_{k+2}\}$. Then 
there is a polynomial $\widetilde p_{i,j}$ of degree $k+2$ and a harmonic function $p^\prime_{i,j} \in S^D_{k+2}$ such that  
$p^D_{i,j}:= \widetilde p_{i,j} + p^\prime_{i,j}$ satisfies 
\begin{align}
\label{eq:lemma:polynomial-solution}
\begin{split}
-\Delta p^D_{i,j} &= x^i y^j \quad \text{on } \mathcal{C}, \\
p^D_{i,j}|_\Gamma &= 0. 
\end{split}
\end{align}
In the special case $\omega = \pi$, the contribution $p^\prime_{i,j}$ may be taken to be zero. 
\end{lemma}

\begin{proof}
\emph{Step~1:} There is a polynomial $\hat p_{k+2}$ of degree $k+2$ such that $-\Delta \hat p_{k+2} = x^i y^j$. This is shown by induction on $j$: 
one observes for any $i \in \BbbN_0$ that $\Delta x^{i+2} y^0 = (i+2)(i+1) x^i y^0$ so the case $j = 0$ is shown; the formula 
$\Delta (x^{i+2} y^{j+1}) = (i+2)(i+1) x^i y^{j+1} + (j+1) j x^{i+2} y^{j-1}$ provides the induction step from $j$ to $j+1$. 

\emph{Step~2:} If $\omega = \pi$, then the boundary condition on the line $\{y = 0\}$ can be enforced by subtracting a suitable 
harmonic polynomial $\operatorname{Re} \sum_{n=0}^{k+2} a_n z^n$ with $z =x + iy$ and coefficients $a_n \in \BbbR$. 

\emph{Step 3:} If $\omega \ne \pi$, then the boundary condition at $\phi = 0$ is again enforced 
by subtracting a suitable harmonic polynomial of the form $\operatorname{Re} \sum_{n=0}^{k+2} a_n z^n$. To correct the boundary
condition at $\phi = \omega$, we note (see, e.g.,  \cite[Lem.~{6.1.1}]{Melenk_1995}) that $\operatorname{Im} z^n = r^n \sin (n\phi)$ 
and $\operatorname{Im} (z^n \ln z) = r^n \left(\ln r \sin (n \phi) + \phi \cos (n \phi)\right)$. Both functions vanish on $\phi = 0$ and are harmonic. 
Next, we can match a function of the form $r^n$ on the line $\phi = \omega$ by either $\operatorname{Im} z^n$ if $n\frac{\omega}{\pi} \notin\BbbN$ or 
by $\operatorname{Im} z^n \ln z$ if $n \frac{\omega}{\pi} \in \BbbN$. The function $\operatorname{Im} (z \ln z)$ is not required in the set $S^D_{k+2}$ 
since the case $1 \cdot \frac{\omega}{\pi} \in \BbbN$ can only arise for $\omega = \pi$. 
\end{proof}

\begin{corollary}
\label{cor:solution-with-polynomial}
Let $R > 0$. 
Let $k\in\mathbb{N}_0$ and $\epsilon\in (0,1)$ satisfy $k+1+\epsilon < \lambda^D_2 = \frac{2\pi}{\omega}$ and $k+1+\epsilon \ne \lambda_1^D = \frac{\pi}{\omega}$. 
Let $f\in H^{k+\epsilon}(\mathcal{C})$ with $\operatorname{supp} f\subseteq B_1(0)$. Let $\chi \in C^\infty_0(B_1(0))$ with $\chi\equiv 1$ near the origin. Then every function $u_1\in H^1(\mathcal{C})$ with $\operatorname{supp} u_1 \subseteq B_1(0)$ solving
\begin{align}
\label{eq:cor:solution-with-polynomial}
-\Delta u_1 &= f \in H^{k+\epsilon}(\mathcal{C}), 
\qquad 
u_1 = 0, \mbox{ for $\phi\in \{0,\omega\}$}, 
\end{align}
has the form $u_1 = u_0 + \chi P_{k-1} + \delta$ with 
$u_0\in H^{k+2+\epsilon}(\mathcal{C}_R )$,  
\begin{align}
\label{eq:cor:solution-with-polynomial-1.5}
\delta & = \begin{cases} 0, & \mbox{ if } \quad k+\epsilon+1<\lambda^D_1 = \frac{\pi}{\omega} \\ 
                           S^D_1(f) s^D_1, & \mbox{ if } \quad k+\epsilon+1> \lambda^D_1 = \frac{\pi}{\omega} , 
\end{cases} \\
S^D(f) &:= - \frac{1}{\pi} \left(\int_{\mathcal{C}} r^{-\lambda_1^D} \sin(\lambda_1^D\phi) (f(\mathbf{x})+\Delta(\chi(\mathbf{x}) P_{k-1}(\mathbf{x}))) \, d\mathbf{x}
\right), \\
\label{eq:s1D}
s_1^D & := r^{\lambda_1^D} \sin(\lambda_1^D\phi)  ,\\
\label{eq:Pkm1}
P_{k-1}(\mathbf{x}) &:=\sum_{i+j\leq k-1} \frac{1}{i!j!} p^D_{i,j}(\mathbf{x})(\partial_x^i \partial_y^j f)(0), 
\end{align}
and the $p^D_{i,j}$ are the fixed functions from Lemma~\ref{lemma:polynomial-solution}. Furthermore, 
\begin{align}
\label{eq:cor:solution-with-polynomial-1.75}
\|u_0\|_{H^{k+2+\epsilon}(\mathcal{C}_R) } &\lesssim \|f\|_{H^{k+\epsilon}(\mathcal{C})},  \\
\label{eq:cor:solution-with-polynomial-1.75a}
\|P_{k-1}\|_{H^{k+2+\epsilon}({\mathcal C}_R)} &\lesssim \|f\|_{B^{k}_{2,1}({\mathcal C})} \quad \mbox{ if $\Sigma^D_{k+1}  =\emptyset$}, \\
\label{eq:cor:solution-with-polynomial-1.75b}
\|P_{k-1}\|_{B^{n^\star+1}_{2,\infty}({\mathcal C}_R)} &\lesssim \|f\|_{B^{k}_{2,1}({\mathcal C})} \quad \mbox{ if $\Sigma^D_{k+1}  \ne \emptyset$}, 
\qquad n^\star:= \min\{n \in \{2,\ldots,k+1\}\,|\, n \frac{\omega}{\pi} \in \BbbN\}, \\ 
\label{eq:cor:solution-with-polynomial-1.75c}
\|\Delta (\chi P_{k-1})\|_{H^{k+\epsilon}({\mathcal C}_R)} &\lesssim \|f\|_{B^{k}_{2,1}({\mathcal C})} 
\lesssim \|f\|_{H^{k+\epsilon}(\mathcal{C})}. 
\end{align}
The implied constants depend only $k$, $\epsilon$, the angle $\omega$, and the choice of the cut-off function $\chi$.
\end{corollary}

\begin{proof}
We only consider $k\geq 1$, since the claim for $k=0$ is a restatement of Proposition~\ref{prop:solution_u1} and $P_{-1} = 0$. 

\emph{Step 1:}
Lemma~\ref{lemma:polynomial-solution} provides functions $p^D_{i,j}$ such that the function $P_{k-1}=\sum_{i+j\leq k-1} \frac{1}{i!j!} p^D_{i,j} (\partial_x^i \partial_y^j f)(0)$ solves the problem
\begin{align}
\label{eq:cor:solution-with-polynomial-2}
-\Delta P_{k-1} &= \sum_{i+j\leq k-1} \frac{1}{i!j!} x^i y^j (\partial_x^i \partial_y^j f)(0), 
\qquad 
P_{k-1} = 0, \ \mbox{ for $\phi\in \{0,\omega\}$}.
\end{align}
We have for $k \ge 1$ the embedding $B^{k}_{2,1} \subset C^{k-1}$, which is asserted in 
\cite[Thm.~{2.8.1(c)}]{Triebel2ndEd} and could, alternatively, be obtained by combining 
the classical Gagliardo inequality (in 2D) $\|\nabla^j u\|_{L^\infty} \lesssim 
\|u\|_{H^{k+1}}^\theta \|u\|_{L^2}^{1-\theta}$ with $\theta = (j+1)/(k+1)$ for $0 \leq j \leq k-1$ with 
the result \cite[Thm.~{25.3}]{Tartar_2007}. In view of (\ref{eq:Pkm1}), this embedding leads to 
the estimates (\ref{eq:cor:solution-with-polynomial-1.75a}), (\ref{eq:cor:solution-with-polynomial-1.75b}): 
In the first case, $\Sigma^D_{k+1} = \emptyset$, the functions $p^D_{i,j}$ are polynomials (and hence smooth). 
In the second case, $\Sigma^D_{k+1} \ne \emptyset$, the functions $p^D_{i,j}$ are sums of polynomials, which are smooth, 
and functions $p^\prime_{i,j} \in S^D_{k+1}$, which are in $B^{n^\star+1}_{2,\infty}(\mathcal{C}_1)$ by 
Lemma~\ref{lemma:bounded-linear-besov-2}\ref{item:lemma:bounded-linear-besov-2-iii} ahead.  

\emph{Step 2:}
We now define $\widetilde{u_1}:=u_1 - \chi P_{k-1}$, which also has support in $B_1(0)$. Since $u_1$ solves \eqref{eq:cor:solution-with-polynomial}, the function $\widetilde{u_1}\in H^1(\mathcal{C})$ solves the problem
\begin{align}
\label{eq:cor:solution-with-polynomial-3}
-\Delta \widetilde{u_1} &= \widetilde{f} := f+\Delta(\chi P_{k-1}) \in H^{k+\epsilon}(\mathcal{C}), 
\qquad 
\widetilde{u_1} = 0 \ \mbox{ for $\phi\in \{0,\omega\}$}.
\end{align}
By constuction of $P_{k-1}$, the right-hand side $\widetilde{f}$ satisfies 
\begin{align}
\label{eq:cor:solution-with-polynomial-4}
\partial_x^i \partial_y^j \widetilde{f}(0) = 0 \quad \text{for $i+j \leq k-1$.}
\end{align}
Thus Proposition~\ref{prop:solution_u1} can be applied to problem \eqref{eq:cor:solution-with-polynomial-3}, and we obtain
\begin{align*}
\widetilde{u_1} = \begin{cases} 
u_0, & \mbox{ if } \quad k+1+\epsilon<\lambda^D_1 = \frac{\pi}{\omega}, \\ 
u_0 - \frac{1}{\pi} \left(\int_{\mathcal{C}} r^{-\lambda_1^D} \sin(\lambda_1^D\phi) \widetilde{f}(\mathbf{x}) \, d\mathbf{x}\right) 
r^{\lambda_1^D} \sin(\lambda_1^D\phi), & \mbox{ if } \quad k+1+\epsilon>  \lambda^D_1 = \frac{\pi}{\omega}
\end{cases}
\end{align*}
with $\|{u_0}\|_{H^{k+2+\epsilon}(\mathcal{C}_R )} \lesssim \|\widetilde f\|_{H^{k+\epsilon}(\mathcal{C})}$. To complete the proof, we distinguish the cases
$\Sigma^D_{k+1}  = \emptyset$ and 
$\Sigma^D_{k+1}  \ne \emptyset$. If  
$\Sigma^D_{k+1}  = \emptyset$, then $P_{k-1}$ is a polynomial of degree $k+1$ and together with 
(\ref{eq:cor:solution-with-polynomial-1.75a}) we get $\|\widetilde f\|_{H^{k+\epsilon}(\mathcal{C}_R)}\lesssim \|f\|_{H^{k+\epsilon}(\mathcal{C}_R)}$. 
If 
$\Sigma^D_{k+1}  \ne \emptyset$, then $P_{k-1}$ is the sum of a polynomial of degree $k+1$, for which we can argue as in the case $\Sigma^D_{k+1} = \emptyset$, 
and a harmonic contribution $P^\prime_{k+1} \in S^D_{k+1}$. The function $P^\prime_{k+1}$ is smooth away from the origin and by harmonicity we have $\Delta (\chi P^\prime_{k+1}) = 0$ near the origin
so that in total we arrive again at the estimate (\ref{eq:cor:solution-with-polynomial-1.75c}) and thus 
$\|\widetilde f\|_{H^{k+\epsilon}(\mathcal{C})} \lesssim \|f\|_{H^{k+\epsilon}(\mathcal{C})}$. 
\end{proof}

\subsection{Regularity of the singularity function and the stress intensity functional}
The following Lemma~\ref{lemma:bounded-linear-besov-2} clarifies in which Besov spaces the singularity functions arising in corner domains lie. 
The proof of Lemma~\ref{lemma:bounded-linear-besov-2}\ref{item:lemma:bounded-linear-besov-2-ii}, \ref{item:lemma:bounded-linear-besov-2-iii} 
relies on arguments given in \cite{babuska-kellogg-pitkaranta79} or \cite[Thm.~{2.1}]{babuska-osborn91}. 
\begin{lemma}
\label{lemma:bounded-linear-besov-2}
The following statements hold.
\begin{enumerate}[leftmargin=8mm, label=(\roman*)]
\item \label{item:lemma:bounded-linear-besov-2-i} For $\alpha>1$, $\alpha\notin\mathbb{N}$, set $k:=\lfloor \alpha \rfloor-1$ 
and let $P_{k-1}$, $\chi$ be as in Corollary~\ref{cor:solution-with-polynomial}. 
Then the mapping
\begin{align*}
f \mapsto S(f):=\int_{\mathcal{C}_1} r^{-\alpha} \sin(\alpha\phi) (f+\Delta(\chi P_{k-1})) \, d\mathbf{x}
\end{align*}
is bounded and linear on $B_{2,1}^{\alpha-1}(\mathcal{C}_1)$. 
\item \label{item:lemma:bounded-linear-besov-2-ia} Let $0<\alpha\leq 1$. Then the mapping
\begin{align*}
f \mapsto S(f):=\int_{\mathcal{C}_1} r^{-\alpha} \sin(\alpha\phi) f \, d\mathbf{x}
\end{align*}
is bounded and linear on $\widetilde{B}_{2,1}^{\alpha-1}(\mathcal{C}_1)$. 
For $\alpha >1/2$, it is also bounded and linear on ${B}_{2,1}^{\alpha-1}(\mathcal{C}_1) = \widetilde{B}^{\alpha-1}_{2,1}(\mathcal{C}_1)$.  
\item \label{item:lemma:bounded-linear-besov-2-ii} For $\beta \ge -1$, the function $s^+(r,\phi)=r^\beta \sin(\beta\phi)$ is in the space $B_{2,\infty}^{1+\beta}(\mathcal{C}_1)$.
\item \label{item:lemma:bounded-linear-besov-2-iii} 
Let $\Phi \in C^\infty(\BbbR^2)$ with $|\Phi(x,y)| \leq C r^n$ as $r\rightarrow 0$. 
The functions $v(x,y) =\Phi(x,y) \ln r$ and $w(x,y) = \phi \Phi(x,y)$ 
are in the space $B_{2,\infty}^{n+1}(\mathcal{C}_1)$.
\item
The statements 
\ref{item:lemma:bounded-linear-besov-2-i}, 
\ref{item:lemma:bounded-linear-besov-2-ia}, 
\ref{item:lemma:bounded-linear-besov-2-ii} remain true if the function $\sin$ is replaced with $\cos$. 
\end{enumerate}
\end{lemma}

\begin{proof}
\emph{Proof of \ref{item:lemma:bounded-linear-besov-2-i}:} Choose $0<\epsilon\ll 1$ such that $k+1-\alpha+\epsilon<0$. 
By the Reiteration Theorem~\cite[Thm.~{26.3}]{Tartar_2007} 
\begin{align*}
B_{2,1}^{\alpha-1}(\mathcal{C}_1) &= (H^{k+\epsilon}(\mathcal{C}_1),H^{k+1}(\mathcal{C}_1))_{\frac{\alpha-k-1-\epsilon}{1-\epsilon},1} = (H^{k+\epsilon}(\mathcal{C}_1),(H^{k+\epsilon}(\mathcal{C}_1),H^{k+2+\epsilon}(\mathcal{C}_1))_{\frac{1-\epsilon}{2},2})_{\frac{\alpha-k-1-\epsilon}{1-\epsilon},1} \\
&= (H^{k+\epsilon}(\mathcal{C}_1),(H^{k+\epsilon}(\mathcal{C}_1),H^{k+2+\epsilon}(\mathcal{C}_1))_{\frac{1-\epsilon}{2},1})_{\frac{\alpha-k-1-\epsilon}{1-\epsilon},1} = (H^{k+\epsilon}(\mathcal{C}_1),B_{2,1}^{k+1}(\mathcal{C}_1))_{\frac{\alpha-k-1-\epsilon}{1-\epsilon},1}.  
\end{align*} 
Now assume $f\in C^\infty(\overline{\mathcal{C}_1})$, $f\neq 0$, the general statement will then follow by density arguments. For
\begin{align*}
\delta:=\operatorname{min}\left\{\|f\|_{H^{k+\epsilon}(\mathcal{C}_1)}^{\frac{1}{1-\epsilon}} \|f\|_{B_{2,1}^{k+1}(\mathcal{C}_1)}^{-\frac{1}{1-\epsilon}},\frac{1}{2}\operatorname{diam}\{\mathbf{x}\in\mathcal{C}:\chi(\mathbf{x})=1\}\right\},
\end{align*}
denote by $\chi_\delta$ a smooth cut-off function that equals zero for $r<\delta$ and one for $r>2\delta$. We have
\begin{align}
\nonumber 
S(f) & = \int_{\mathcal{C}_1} r^{-\alpha} \sin(\alpha\phi) \chi_\delta (f+\Delta(\chi P_{k-1})) \, d\mathbf{x} + \int_{\mathcal{C}_1} r^{-\alpha} \sin(\alpha\phi) (1-\chi_\delta) (f+\Delta(\chi P_{k-1})) \, d\mathbf{x} \\
& =:S_1 + S_2. 
\label{eq:lemma:bounded-linear-besov-2-1}
\end{align}
The first integral, $S_1$, is estimated by
\begin{align*}
S_1 & = \left|\int_{\mathcal{C}_1} r^{-\alpha+k+\epsilon} \sin(\alpha\phi) \chi_\delta \frac{f+\Delta(\chi P_{k-1})}{r^{k+\epsilon}} \, dx\right| \\
&\quad \lesssim \left\|\frac{f+\Delta(\chi P_{k-1})}{r^{k+\epsilon}}\right\|_{L^2(\mathcal{C}_1)} \left|\int_\delta^1 r^{-2\alpha+2k+2\epsilon+1} \chi_\delta^2 \, dr\right|^{\frac{1}{2}} \lesssim \left\|\frac{f+\Delta(\chi P_{k-1})}{r^{k+\epsilon}}\right\|_{L^2(\mathcal{C}_1)} \delta^{k+1-\alpha+\epsilon}.
\end{align*}
The remaining $L^2$-norm can be handled with Lemma~\ref{lemma:Sobolev-to-cone-second-version}: We have, 
since $\partial_x^i\partial_y^j (f+\Delta(\chi P_{k-1}))(0)=0$ for $i+j\leq k-1$, cf.\ \eqref{eq:cor:solution-with-polynomial-4}, 
\begin{align*}
\left\|\frac{f+\Delta(\chi P_{k-1})}{r^{k+\epsilon}}\right\|_{L^2(\mathcal{C}_1)}^2 \leq \|f+\Delta(\chi P_{k-1})\|_{K_{-\epsilon}^k(\mathcal{C}_1)}^2 \lesssim \|f+\Delta(\chi P_{k-1})\|_{H^{k+\epsilon}(\mathcal{C}_1)}^2 \stackrel{\eqref{eq:cor:solution-with-polynomial-1.75c}}{\lesssim} \|f\|_{H^{k+\epsilon}(\mathcal{C}_1)}^2. 
\end{align*}
Since the expression $\delta$ involves a minimum, we analyze two cases. If $\|f\|_{H^{k+\epsilon}(\mathcal{C}_1)}^{\frac{1}{1-\epsilon}} \|f\|_{B_{2,1}^{k+1}(\mathcal{C}_1)}^{-\frac{1}{1-\epsilon}} \leq \frac{1}{2}\operatorname{diam}\{\mathbf{x}\in\mathcal{C}:\chi(\mathbf{x})=1\}$, we get directly
\begin{align*}
\delta^{k+1-\alpha+\epsilon} = \left(\|f\|_{H^{k+\epsilon}(\mathcal{C}_1)}^{\frac{1}{1-\epsilon}} \|f\|_{B_{2,1}^{k+1}(\mathcal{C}_1)}^{-\frac{1}{1-\epsilon}}\right)^{k+1-\alpha+\epsilon},
\end{align*}
since the exponent $k+1-\alpha+\epsilon < 0$; if, on the other hand, $\|f\|_{H^{k+\epsilon}(\mathcal{C}_1)}^{\frac{1}{1-\epsilon}} \|f\|_{B_{2,1}^{k+1}(\mathcal{C}_1)}^{-\frac{1}{1-\epsilon}} > \frac{1}{2}\operatorname{diam}\{\mathbf{x}\in\mathcal{C}:\chi(\mathbf{x})=1\}$, then 
the continuous embedding 
$B_{2,1}^{k+1}(\mathcal{C}_1) \subset H^{k+\epsilon}(\mathcal{C}_1)$ implies 
$\|f\|_{H^{k+\epsilon}(\mathcal{C}_1)}^{\frac{1}{1-\epsilon}} \|f\|_{B_{2,1}^{k+1}(\mathcal{C}_1)}^{-\frac{1}{1-\epsilon}} \lesssim 1$, 
and  we arrive at 
\begin{align*}
\delta^{k+1-\alpha+\epsilon} \lesssim \left(\frac{1}{2}\operatorname{diam}\{\mathbf{x}\in\mathcal{C}:\chi(\mathbf{x})=1\}\right)^{k+1-\alpha+\epsilon} \lesssim \left(\|f\|_{H^{k+\epsilon}(\mathcal{C}_1)}^{\frac{1}{1-\epsilon}} \|f\|_{B_{2,1}^{k+1}(\mathcal{C}_1)}^{-\frac{1}{1-\epsilon}}\right)^{k+1-\alpha+\epsilon}.
\end{align*}
For the second integral of \eqref{eq:lemma:bounded-linear-besov-2-1}, $S_2$,  we obtain
\begin{align*}
S_2 &
= \left|\int_{\mathcal{C}_1} \! r^{-\alpha+k} \sin(\alpha\phi) (1-\chi_\delta) \frac{f\!+\!\Delta(\chi P_{k-1})}{r^k} \, dx\right| 
\lesssim \left\|\frac{f+\Delta(\chi P_{k-1})}{r^k}\right\|_{L^\infty(\mathcal{C}_{2\delta}) } \int_0^{2\delta} r^{-\alpha+k+1} \, dr.
\end{align*}
Since $\Delta(\chi P_{k-1})=\Delta P_{k-1}=-\sum_{i+j\leq k-1} \frac{1}{i!j!} x^i y^j (\partial_x^i \partial_y^j f)(0)$ in the region where $\chi\equiv 1$, it follows with the embedding $B_{2,1}^1(\mathcal{C}_1) \subseteq C(\overline{\mathcal{C}_1})$, cf.~\cite[Thm.~{4.6.1}]{Triebel2ndEd},
\begin{align*}
S_2 &
\lesssim \left\|\frac{f-\sum_{i+j\leq k-1} \frac{1}{i!j!} x^i y^j (\partial_x^i \partial_y^j f)(0)}{r^k}\right\|_{L^\infty(\mathcal{C}_1)} \delta^{k+2-\alpha} \\
&\qquad \lesssim \|D^k f\|_{L^\infty(\mathcal{C}_1)} \delta^{k+2-\alpha} \lesssim \|f\|_{B_{2,1}^{k+1}(\mathcal{C}_1)} \left(\|f\|_{H^{k+\epsilon}(\mathcal{C}_1)}^{\frac{1}{1-\epsilon}} \|f\|_{B_{2,1}^{k+1}(\mathcal{C}_1)}^{-\frac{1}{1-\epsilon}}\right)^{k+2-\alpha}.
\end{align*}
In total, we have arrived at
\begin{align}
\label{eq:foo-10}
|S(f)| \lesssim \|f\|_{H^{k+\epsilon}(\mathcal{C}_1)}^{\frac{k+2-\alpha}{1-\epsilon}} \|f\|_{B_{2,1}^{k+1}(\mathcal{C}_1)}^{\frac{\alpha-k-1-\epsilon}{1-\epsilon}}.
\end{align}
By \cite[Lem.~{25.2}]{Tartar_2007}, the estimate (\ref{eq:foo-10}) implies $S\in \left((H^{k+\epsilon}(\mathcal{C}_1),B_{2,1}^{k+1}(\mathcal{C}_1))_{\frac{\alpha-k-1-\epsilon}{1-\epsilon},1}\right)^\star = (B_{2,1}^{\alpha-1}(\mathcal{C}_1))^\star$.


\emph{Proof of \ref{item:lemma:bounded-linear-besov-2-ia}:} 
By \ref{item:lemma:bounded-linear-besov-2-ii} we have for $\alpha \in (0,1]$ 
that $s^-(r,\phi):=r^{-\alpha} \sin(\alpha\phi) \in  B^{-\alpha+1}_{2,\infty}(\mathcal{C}_1)$. From the 
characterization of dual spaces of interpolation spaces, \cite[Lem.~{41.3}]{Tartar_2007}, \cite[Thm.~{1.11.2}]{Triebel2ndEd}, we have 
for any $\epsilon \in (0,1/2)$ in view of $-\alpha + 1 \in [0,1)$ with $\theta = (-\alpha+1+\epsilon)/(1+\epsilon)$
$$
B^{-\alpha+1}_{2,\infty} =  (H^{-\epsilon}, H^1)_{\theta,\infty} = 
\left(((H^1)^\star, (H^{-\epsilon})^\star)_{1-\theta,1}\right)^\star = 
\left((\widetilde{H}^{-1}, \widetilde{H}^{\epsilon})_{1-\theta,1}\right)^\star  = \left(\widetilde{B}^{\alpha-1}_{2,1}\right)^\star. 
$$
For $\alpha > 1/2$, we note that $\alpha - 1 \in (-1/2,0)$ so that by (\ref{eq:B=tildeB}) we have $\widetilde{B}^{\alpha-1}_{2,1} = B^{\alpha-1}_{2,1}$. 
%
%

\emph{Proof of \ref{item:lemma:bounded-linear-besov-2-ii}:} 

\emph{Step~1 ($\beta \in \BbbN_0$):} For $\beta \in \BbbN_0$, the function $(x,y) \mapsto r^\beta \sin (\beta \phi)$ is a polynomial and thus smooth. 

\emph{Step~2 ($\beta > -1$, $\beta \not\in\BbbN_0$):} 
We write the Besov space as the interpolation space
\begin{align*}
B_{2,\infty}^{1+\beta}(\mathcal{C}_1) = (L^2(\mathcal{C}_1),H^{\lfloor\beta\rfloor+2}(\mathcal{C}_1))_{\frac{1+\beta}{\lfloor\beta\rfloor+2},\infty}.
\end{align*}
Next we select a smooth cut-off function $\chi_t$ with $\chi_t \equiv 0$ on $B_{t^{\frac{1}{\lfloor\beta\rfloor+2}}/2}(0)$ and $\chi_t \equiv 1$ on $B_1(0) \backslash B_{t^{\frac{1}{\lfloor\beta\rfloor+2}}}(0)$, and whose derivatives satisfy $\|\nabla^k \chi_t\|_{L^\infty(\mathcal{C}_1)} \lesssim t^{-\frac{k}{\lfloor\beta\rfloor+2}}$. We then get
\begin{align}
\label{eq:lemma:regularity-splus-2}
\|(1-\chi_t)s^+\|_{L^2(\mathcal{C}_1)}^2 \lesssim \int_0^1 (1-\chi_t)^2 r^{2\beta} r \, dr \lesssim \int_0^{t^{\frac{1}{\lfloor\beta\rfloor+2}}} r^{2\beta+1} \, dr \lesssim t^{\frac{2(\beta+1)}{\lfloor\beta\rfloor+2}}.
\end{align}
For the derivatives we obtain
\begin{align}
\label{eq:lemma:regularity-splus-3}
\begin{split}
\|\nabla^{\lfloor\beta\rfloor+2}(\chi_t s^+)\|_{L^2(\mathcal{C}_1)}^2 &\lesssim \int_0^1 \sum_{s=0}^{\lfloor\beta\rfloor+2} |\nabla^s \chi_t(r)|^2 |\nabla^{\lfloor\beta\rfloor+2-s} r^\beta|^2 r \, dr \\
&\lesssim \int_{t^{\frac{1}{\lfloor\beta\rfloor+2}}/2}^1 r^{2(\beta-\lfloor\beta\rfloor-2)+1} \, dr + \sum_{s=1}^{\lfloor\beta\rfloor+2} t^{-\frac{2s}{\lfloor\beta\rfloor+2}} \left(t^{\frac{1}{\lfloor\beta\rfloor+2}}\right)^{2(\beta+s-\lfloor\beta\rfloor-2)+2} \\
&\lesssim t^{\frac{2\beta-2\lfloor\beta\rfloor-2}{\lfloor\beta\rfloor+2}} + \sum_{s=1}^{\lfloor\beta\rfloor+2} t^{-\frac{2s}{\lfloor\beta\rfloor+2}} \left(t^{\frac{1}{\lfloor\beta\rfloor+2}}\right)^{2(\beta+s-\lfloor\beta\rfloor-2)+2} \lesssim t^{\frac{2(\beta+1)}{\lfloor\beta\rfloor+2}-2}.
\end{split}
\end{align}
The $L^2$-norm satisfies
\begin{align}
\label{eq:lemma:regularity-splus-4}
\|\chi_t s^+\|_{L^2(\mathcal{C}_1)}^2 \lesssim \int_{t^{\frac{1}{\lfloor\beta\rfloor+2}}/2}^1 r^{2\beta+1} \, dr \lesssim t^{\frac{2(\beta+1)}{\lfloor\beta\rfloor+2}-2},
\end{align}
since the integral is bounded and $\frac{2(\beta+1)}{\lfloor\beta\rfloor+2}-2 < 0$. Then \eqref{eq:lemma:regularity-splus-2}, \eqref{eq:lemma:regularity-splus-3}, and \eqref{eq:lemma:regularity-splus-4} imply
$K(t,s^+) \lesssim t^{\frac{1+\beta}{\lfloor\beta\rfloor+2}}
$
and thus $s^+\in B_{2,\infty}^{1+\beta}(\mathcal{C}_1)$.

\emph{Step~3 ($\beta = -1$):} Fix $\epsilon \in (0,1/2)$. We assert $s^+ \in B^0_{2,1}(\mathcal{C}_1) = (H^{-\epsilon}(\mathcal{C}_1),H^{\epsilon}(\mathcal{C}_1))_{1/2,1}$ using the 
``Babu\v{s}ka trick'', \cite[Thm.~{1.4.5.3}]{Grisvard_2011}. Let $\chi_t$ be the cut-off function of Step~2 (with $\beta = -1$) and split $s^+ = r^{-1} \sin(\phi)$ as 
$s^+ = (1-\chi_t) s^+ + \chi_t s^+$. We estimate with the Cauchy-Schwarz inequality and Lemma~\ref{lemma:Sobolev-to-cone-second-version}
\begin{align*}
\|(1-\chi_t) s^+\|_{H^{-\epsilon}(\mathcal{C}_1)} & = \sup_{v \in \widetilde{H}^\epsilon(\mathcal{C}_1)\,|\,  \|v\|_{{\widetilde{H}}^\epsilon} = 1} \int_{\mathcal{C}_1} (1-\chi_t) r^{\epsilon} s^+ r^{-\epsilon} v\, d\mathbf{x}
\stackrel{\text{C.S.}}{\lesssim} t^\epsilon \sup_{\|v\|_{{\widetilde{H}}^\epsilon = 1}} \|r^{-\epsilon} v\|_{L^2(\mathcal{C}_1)} 
\stackrel{\text{Lemma~\ref{lemma:Sobolev-to-cone-second-version}}}{\lesssim} t^\epsilon.  
\end{align*}
For the term $\|\chi_t s^+\|_{H^{\epsilon}(\mathcal{C}_1)}$, we employ the ``Babu{\v s}ka trick'', i.e., we use the Sobolev embedding $W^{1,p} \subset H^{\epsilon}$ for $1/p = 1-\epsilon/2$, 
\cite[Thm.~{1.4.5.2}]{Grisvard_2011}, \cite[Thm.~{2.8.1}]{Triebel2ndEd}, to conveniently estimate the norm $\|\cdot\|_{H^\epsilon}$. A calculation shows for $t < 1$ 
\begin{align*}
\|\chi_t s^+\|_{H^\epsilon(\mathcal{C}_1)} \lesssim \|\chi_t s^+ \|_{W^{1,p}(\mathcal{C}_1)} \lesssim t^{-\epsilon}. 
\end{align*}
In conclusion, the $K$-functional of $s^+$ for the pair $(H^{-\epsilon}(\mathcal{C}_1), H^\epsilon(\mathcal{C}_1))$ satisfies 
$K(\tau,s^+)\lesssim t^\epsilon + \tau t^{-\epsilon}$. Selecting $t = \tau^{1/(2\epsilon)}$ shows $K(\tau,s^+) \lesssim \tau^{1/2}$ 
so that $s^+ \in (H^{-\epsilon}(\mathcal{C}_1), H^\epsilon(\mathcal{C}_1))_{1/2,\infty} = B^{0}_{2,\infty}(\mathcal{C}_1) = \widetilde{B}^0_{2,\infty}(\mathcal{C}_1)$. 


\emph{Proof of \ref{item:lemma:bounded-linear-besov-2-iii}:} First we consider $v=\Phi(x,y) \ln r$ with $\Phi = O(r^n)$ as $r \rightarrow 0$. 
We define for every $0 < t<1$ the function
\begin{align*}
v_t := \Phi(x,y) \int_r^1 - \chi_t(\tau) \tau^{-1} \, d\tau,
\end{align*}
where $\chi_t$ is a smooth cut-off function with the properties $\chi_t(\tau)= 1$ for $\tau>t$ and $\chi_t(\tau)=0$ for $\tau<t/2$
and $|\nabla^j \chi| \lesssim t^{-j}$. Note that $v_t\in C^\infty(\mathcal{C}_1)$ and $v_t=v$ for $r>t$. We obtain
\begin{align*}
\|v-v_t\|_{L^2(\mathcal{C})}^2 &= \int_{\phi = 0}^\omega \int_0^t \Phi^2\left|\int_r^1 \frac{1}{\tau}(\chi_t(\tau)-1) \, d\tau\right|^2 r \, dr \lesssim \int_0^t r^{2n+1} \left|\int_r^t \frac{1}{\tau} \, d\tau\right|^2 \, dr \\
&\lesssim \int_0^t r^{2n} t \frac{r}{t} \ln^2 \frac{t}{r} \, dr = t^{2n+2} \int_0^1 r^{2n+1} \ln^2 r  \, dr \lesssim t^{2n+2}.
\end{align*}
By the smoothness of $\Phi$, we have $|\nabla^j \Phi|  = O(r^{n-j})$ for $j \in \{0,\ldots,n\}$ and $|\nabla^j \Phi| = O(1)$ for $j > n$. 
The product rule then implies (using that $r \sim t$ on $\operatorname{supp} \nabla \chi_t \subset \overline{B}_t(0)\setminus B_{t/2}(0)$)
\begin{align*}
|\nabla^{n+2} v_t| &\lesssim \sum_{\mu=0}^{n} |\nabla^\mu \Phi|\,|\nabla^{n+1-\mu} \frac{\chi_t}{r}| + |r^{-1} \chi_t| + |\int_{r}^1 \chi_t \tau^{-1}\,d\tau| \\
& \lesssim \sum_{\mu=0}^n r^{n-\mu} r^{-(n+2-\mu)} \chi_{\overline{B}_t(0)\setminus B_{t/2}(0))} + r^{-1} \chi_{r > t/2} + |\ln r| 
 \lesssim r^{-2} \chi_{\overline{B}_t(0)\setminus B_{t/2}(0)} + r^{-1} \chi_{r > t/2} + |\ln r|. 
\end{align*}
This implies 
\begin{align*}
\|\nabla^{n+2}v_t\|^2_{L^2({\mathcal C})} = \int_{\phi=0}^\omega \int_{r=0}^1 |\nabla^{n+2} v_t|^2r\,dr\,d\phi \lesssim t^{-2}. 
\end{align*}
Since the above calculations hold for every $t\in(0,1)$, we choose $t=\tau^{1/(n+2)}$ and get for the $K$-functional
\begin{align*}
K(\tau,v)^2 \lesssim \|v-v_t\|_{L^2(\mathcal{C}_1)}^2 + \tau^2 \|v_t\|_{H^{n+2}(\mathcal{C}_1)}^2 \lesssim \tau^{\frac{2(n+1)}{n+2}},
\end{align*}
which shows $v\in \left(L^2(\mathcal{C}_1),H^{n+2}(\mathcal{C}_1)\right)_{\frac{n+1}{n+2},\infty} = B_{2,\infty}^{n+1}(\mathcal{C}_1)$. 

The proof that the function $w = \Phi(x,y) \phi$ is in $B^{n+1}_{2,\infty}(\mathcal{C}_1)$ 
follows similar lines using the function $v_t:= \Phi(x,y) \phi \chi_t$. 
\end{proof}

The core of the proof of the shift Theorem~\ref{thm:shift-theorem-local-version}
is the following abstract result: 
\begin{lemma}
\label{lemma:abstract}
Let $(X_1, \|\cdot\|_{X_1})  \subset (X_0,\|\cdot\|_{X_0})$ and $(Y_1, \|\cdot\|_{Y_1}) \subset (Y_0,\|\cdot\|_{Y_0})$ be Banach spaces with continuous embeddings. 
Let $q_j$, $p_j \in [1,\infty]$, $\theta_j\in (0,1)$, $j=1,\ldots,J$, with $0 < \theta_1 < \theta_2 < \cdots < \theta_J < 1$.  
Let $\widetilde{T}: X_0 \rightarrow Y_0$, 
$S_1: X_1 \rightarrow Y_1$, $S_{\theta_j}: X_{\theta,p_j} \rightarrow Y_{\theta,q_j}$, $j \in \{1,\ldots,J\}$, be bounded linear. 
Assume $\widetilde{T} f = S_1(f) + \sum_{j=1}^J S_{\theta_j}(f)$ for all $f \in X_1$. Then, 
$\widetilde{T} : (X_0,X_1)_{\theta_1,p_1} \rightarrow (Y_0,Y_1)_{\theta_1,\infty}$ is bounded linear. 
\end{lemma}
\begin{proof}
For $t > 0$, decompose $f \in X_{\theta_1,p_1}$ as $f = f_0 + f_1$ with 
$f_0 \in X_0$, $f_1 \in X_1$, and 
\begin{equation}
\label{eq:abstract-10}
\|f_0\|_{X_0} + t\|f_1\|_{X_1} \leq 2 K(t,f) \lesssim t^{\theta_1} \|f\|_{X_{\theta_1,\infty}} \lesssim t^{\theta_1} \|f\|_{X_{\theta_1,p_1}}. 
\end{equation}
By \cite[Lemma]{Bramble_Scott_1978} and the interpolation inequality, we have additionally 
\begin{align} 
\label{eq:abstract-20}
\|f_1\|_{X_{\theta_1,p_1}} &\stackrel{\text{\cite[Lemma]{Bramble_Scott_1978}}}{\leq} 3 \|f\|_{X_{\theta_1,p_1}}, \\
\label{eq:abstract-30}
\|f_1\|_{X_{\theta_j,p_j}} &\lesssim \|f_1\|_{X_{\theta_1,p_1}}^{(1-\theta_j)/(1-\theta_1)} \|f_1\|_{X_1}^{(\theta_j-\theta_1)/(1-\theta_1)} 
\stackrel{(\ref{eq:abstract-10})}{\lesssim} t^{-(\theta_j - \theta_1)} \|f_1\|_{X_{\theta_1},p_1}, \qquad j=2,\ldots,J. 
\end{align} 
We write $\widetilde{T} f = \widetilde{T} f_0 + \widetilde{T} f_1 = \widetilde{T} f_0 + S_1(f_1) + \sum_{j=1}^JS_{\theta_j}(f_1)$.  
For $t > 0$ and $j \in \{1,\ldots,J\}$ decompose $S_{\theta_j}(f_1) = s_{j,0}(f_1) + s_{j,1}(f_1)$ with 
$s_{j,0}(f_1) \in Y_0$, $s_{j,1}(f_1) \in Y_1$ and 
\begin{align*}
\|s_{1,0}(f_1)\|_{Y_0} + t \|s_{1,1}(f_1)\|_{Y_1} &\leq 2 K(t,S_{\theta_1}(f_1)) \lesssim t^{\theta_1} \|S_{\theta_1}(f_1)\|_{Y_{\theta_1,q_1}} 
\lesssim t^{\theta_1} \|f_1\|_{X_{\theta_1,p_1}} \stackrel{(\ref{eq:abstract-20})}{\lesssim} t^{\theta_1}\|f\|_{X_{\theta_1,p_1}}, \\
\|s_{j,0}(f_1)\|_{Y_0} + t \|s_{j,1}(f_1)\|_{Y_1} &\leq 2 K(t,S_{\theta_j}(f_1)) \lesssim t^{\theta_j} \|S_{\theta_j}(f_1)\|_{Y_{\theta_j,q_j}} 
\lesssim t^{\theta_j} \|f_1\|_{X_{\theta_j,p_j}} 
\stackrel{(\ref{eq:abstract-30})}{\lesssim} 
t^{\theta_1} \|f\|_{X_{\theta_1,p_1}}. 
\end{align*}
This implies the decomposition $\widetilde{T} f = \left(\widetilde{T} f_0 + \sum_{j=1}^J s_{j,0}(f_1) \right) + \left(S_1(f_1) + \sum_{j=1}^Js_{j,1}(f_1) \right)
=: y_0 + y_1$ with 
\begin{align*}
\|y_0\|_{Y_0} + t \|y_1\|_{Y_1} &\lesssim \|f_0\|_{X_0} + t^{\theta_1} \|f\|_{X_{\theta_1,p_1}} 
+ t \left(\|f_1\|_{X_1} + t^{\theta_1-1} \|f\|_{X_{\theta_1,p_1}}\right)
\stackrel{ (\ref{eq:abstract-30})} {\lesssim }
t^{\theta_1} \|f\|_{X_{\theta_1,p_1}}. 
\end{align*}
Hence, $\widetilde{T} f \in X_{\theta_1,\infty}$ with $\|\widetilde{T} f\|_{X_{\theta_1,\infty}} \lesssim \|f\|_{X_{\theta_1,p_1}}$. 
\end{proof}

We are now in position to prove the shift Theorem~\ref{thm:shift-theorem-local-version} for Dirichlet boundary conditions.

\begin{numberedproof}{of Theorem~\ref{thm:shift-theorem-local-version}\ref{item:thm:shift-theorem-local-version-dirichlet} (Dirichlet conditions):}
We denote by $\chi_r$, $r > 0$, a smooth cutoff function with $\operatorname{supp} \chi_r \subset B_r(0)$ and $\chi_r \equiv 1$ near $0$. 
We assume for simplicity $R < 1$. 
By Remark~\ref{rem:omega=pi} a stronger shift theorem holds  $\omega \ne \pi$ so that we will assume $\omega \ne \pi$. 

\textit{Step 0: (local regularity)} Since the two lines $\{\phi = 0\}$ and $\{\phi = \omega\}$ are smooth, local elliptic regularity gives for any $0 < R_1 < R_2 <R_3 < R$ that for any 
$s \in \BbbN_0$ (see, e.g., \cite[Sec.~{6}]{Evans_2010}; this is even true for any $s \ge 0$, see Lemma~\ref{lemma:regularity-lemma-annuli} for details)
\begin{align}
\label{eq:thm:shift-theorem-local-version-1000}
\|u\|_{H^{s+2}(\mathcal{C}_{R_2} \setminus \mathcal{C}_{R_1})} \lesssim \|f\|_{H^s(\mathcal{C}_{R_3})} + \|u\|_{H^1(\mathcal{C}_{R_3})}. 
\end{align}

\textit{Step 1: (localized equation)} Since $R' < R$ and we are interested in the regularity of $u$ in $\mathcal{C}_{R'}$, 
we fix a smooth cut-off function $\chi_{\widetilde R} \in C^\infty_0(B_{\widetilde R}(0))$ with $\chi_{\widetilde{R}} \equiv 1$ on $\mathcal{C}_{R'}$, 
where $R_2 < R' < \widetilde{R} < R_3 < R < 1$ are such that $\chi_{R} \equiv 1$ on $\mathcal{C}_{R_3}$. We set $\widetilde{u}:= \chi_{\widetilde{R}} u$, and we note that $\widetilde{u}$ satisfies 
\begin{align}
\label{eq:thm:shift-theorem-local-version-5}
-\Delta \widetilde u &= -\chi_{\widetilde R} \Delta u - 2\nabla\chi_{\widetilde R} \cdot \nabla u - \Delta\chi_{\widetilde R} u 
= \chi_{\widetilde R} f - 2\nabla\chi_{\widetilde R}\cdot \nabla u - \Delta\chi_{\widetilde R} u =: \widetilde{f} \quad \mbox{ in $\mathcal{C}_R$}, \\
\widetilde{u} & = 0 \quad \mbox{ on $\Gamma_D \cup \widetilde  \Gamma_R = \partial\mathcal{C}_R$.}
\end{align}
\emph{Claim:}
\begin{align}
\label{eq:thm:shift-theorem-local-version-10}
\|\widetilde f\|_{B^{\pi/\omega-1}_{2,1}(\mathcal{C}_R)} \lesssim \|\chi_R f\|_{B^{\pi/\omega-1}_{2,1}(\mathcal{C}_R)} + \|u\|_{H^1(\mathcal{C}_R)}. 
\end{align}
To see this, we consider the cases $\omega < \pi$, and $\omega \ge \pi$ separately. 

\emph{Proof of (\ref{eq:thm:shift-theorem-local-version-10}) for $\omega \leq \pi$:}
Let $s:= \lfloor \lambda^D_1\rfloor = \lfloor \pi/\omega\rfloor \ge 1$ and note the continuous embedding $H^{s+1} \subset B^{\pi/\omega}_{2,1}$ so that together with the support properties of $\nabla \chi_{\widetilde R}$  
\begin{align}
\nonumber 
\|\widetilde{f}\|_{B_{2,1}^{\pi/\omega-1}(\mathcal{C}_R)} &\stackrel{H^{s+1} \subset B^{\pi/\omega}_{2,1}}{\lesssim} 
\|\chi_{\widetilde R} f\|_{B_{2,1}^{\pi/\omega-1}(\mathcal{C}_R)} + \|u\|_{H^{s+1}(\mathcal{C}_{\widetilde R}\setminus\mathcal{C}_{R'} )} \\
&\stackrel{\eqref{eq:thm:shift-theorem-local-version-1000}}{\lesssim}
\|\chi_{\widetilde R} f\|_{B_{2,1}^{\pi/\omega-1}(\mathcal{C}_R)} + \|f\|_{H^{s-1}(\mathcal{C}_{R_3}\setminus\mathcal{C}_{R_2})} + \|u\|_{H^1(\mathcal{C}_R)} 
\label{eq:thm:shift-theorem-local-version-5.5}
\lesssim \|\chi_R f\|_{B_{2,1}^{\pi/\omega-1}(\mathcal{C}_R)} + \|u\|_{H^1(\mathcal{C}_R)},
\end{align}
where, in the last step, we used $\chi_{\widetilde R} \chi_{R}  = \chi_{\widetilde R}$ and the continuity of the multiplication with smooth functions in Sobolev (and hence, by interpolation, in Besov spaces)
so that $\|\chi_{\widetilde R} f\|_{B^{\pi/\omega-1}_{2,1}(\mathcal{C}_R)} = \|\chi_{\widetilde R} \chi_R f\|_{B^{\pi/\omega-1}_{2,1}(\mathcal{C}_R)} 
\lesssim \| \chi_{R} f\|_{B^{\pi/\omega-1}_{2,1}(\mathcal{C}_R)}$ and 
$\|f\|_{H^{s-1}(\mathcal{C}_{R_3}\setminus\mathcal{C}_{R_2})} \leq \|f\|_{H^{s-1}(\mathcal{C}_{R_3})} \leq \|\chi_R f\|_{H^{s-1}(\mathcal{C}_{R})} \leq \|\chi_R f\|_{B^{\pi/\omega-1}_{2,1}(\mathcal{C}_R)}$
as $s - 1 \in \BbbN_0$. 

\emph{Proof of (\ref{eq:thm:shift-theorem-local-version-10}) for $\omega > \pi$:} 
Again, set 
$s:= \lfloor \lambda^D_1\rfloor = \lfloor \pi/\omega\rfloor = 0$ and note the continuous embedding 
$L^{2} \subset B^{\pi/\omega-1}_{2,1}$. One may then argue as in the first line of (\ref{eq:thm:shift-theorem-local-version-5.5}), which gives the result. 

\textit{Step 2:} In order to analyze the regularity of $\widetilde u$ on $\mathcal{C}_{R'}$, 
we select $\overline{R} < R$ and $\chi_{\overline{R}} \in C^\infty_0(B_{\overline{R}}(0))$ with $\chi_{\overline{R}} \equiv 1$ on 
$\operatorname{supp} \chi_{R} \subset B_{R}(0)$ and introduce the operators $T$, $\widetilde{T}$ by 
\begin{align}
\label{eq:operator-T}
T:\left\{\begin{array}{ccc} H^{-1}(\mathcal{C}_R) & \rightarrow & H_0^1(\mathcal{C}_R) \\ f & \mapsto & v\end{array}\right. , \qquad \widetilde{T}:\left\{\begin{array}{ccc} H^{-1}(\mathcal{C}_R) & \rightarrow & H_0^1(\mathcal{C}_R) \\ f & \mapsto & \chi_{R'} T \chi_{\overline{R}} f\end{array}\right. ,
\end{align}
where $v = Tf$ solves 
\begin{align*}
-\Delta v &= f \quad \mbox{in $\mathcal{C}_R$,} & v & = 0 \quad \mbox{ on $\Gamma_D \cup \widetilde \Gamma_R = \partial \mathcal{C}_R$}. 
\end{align*}
Then, since $\operatorname{supp} \widetilde{f} \subset B_{\widetilde{R}}(0)$, we have $\chi_{\overline{R}} \widetilde{f} = \widetilde{f}$ and 
therefore $\widetilde{u} = \widetilde{T} \widetilde{f}$. The proof of 
Theorem~\ref{thm:shift-theorem-local-version}\ref{item:thm:shift-theorem-local-version-dirichlet} is complete once we ascertain
\begin{align}
\label{eq:dirichlet-case-goal}
\|\widetilde{T} \widetilde{f} \|_{B^{\pi/\omega+1}_{2,\infty}(\mathcal{C}_{R'})} \lesssim \|\widetilde f\|_{B^{\pi/\omega-1}_{2,1}(\mathcal{C}_R)}, 
\end{align}
which we will do in the ensuing Steps~3--5. 

\textit{Step 3:} We show (\ref{eq:dirichlet-case-goal})
for the case $\lambda^D_1 = \frac{\pi}{\omega} \notin \mathbb{N}$ and $\omega<\pi$ using Lemma~\ref{lemma:abstract}
and Corollary~\ref{cor:solution-with-polynomial}. 
Since $\omega < \pi$ and $\lambda^D_1 \not\in\BbbN$, we can find 
$(k,\epsilon) \in \BbbN_0 \times (0,1)$ such that 
$\lambda^D_1 < k + \epsilon + 1 < \lambda^D_2$ together with 
$\lfloor k+\epsilon+1\rfloor = \lfloor \lambda^D_1 \rfloor  \ge 1$. 
Set $\theta_1:= \frac{\pi}{\omega(k+\epsilon + 1)} \in (0,1)$ as well as 
\label{eq:X0X1Y0Y1-dirichlet}
\begin{align}
\label{eq:dirichlet-case-interpolation-couples}
X_0 &:= H^{-1}(\mathcal{C}_R), & X_1 &= H^{k+\epsilon}(\mathcal{C}_R) , 
& 
Y_0 &:= H^{1}_0(\mathcal{C}_R), & Y_1 &= H^{2+k+\epsilon}(\mathcal{C}_R) . 
\end{align}
We have $B^{\pi/\omega-1}_{2,1}({\mathcal C}_R) = (X_0,X_1)_{\theta_1,1}$ and $B^{\pi/\omega+1}_{2,\infty}({\mathcal C}_R) = (Y_0,Y_1)_{\theta_1,\infty}$. 
By Lax-Milgram, we have $\widetilde{T}:X_0 \rightarrow Y_0$ is bounded, linear. 
With $u_0$, $S^D_1$, $s^D_1$, $\chi P_{k-1}$ given by Corollary~\ref{cor:solution-with-polynomial}, we set 
$S_1(f) = u_0(f) + \chi P_{k-1}$, $S_{\theta_1}(f) = S^D(f) s^D_1$. 
We note that $\Sigma^D_{k+1} 
= \{n \in \{2,\ldots,\lfloor\frac{\pi}{\omega}\rfloor\}\,|\,  n \frac{\omega}{\pi} \in \BbbN\} = \emptyset$ in view of 
$ \lfloor \frac{\pi}{\omega} \rfloor \frac{\omega}{\pi} < 1$. Corollary~\ref{cor:solution-with-polynomial} shows that 
$S_1:X_1 \rightarrow Y_1$ is bounded, linear; Lemma~\ref{lemma:bounded-linear-besov-2}\ref{item:lemma:bounded-linear-besov-2-i} asserts that 
$S^D_1 \in \left((X_0,X_1)_{\theta_1,1}\right)^\star$ and $s^D_1 \in (X_0,X_1)_{\theta,\infty}$. The desired 
assertion (\ref{eq:dirichlet-case-goal}) now follows from Lemma~\ref{lemma:abstract}.

\textit{Step 4:} We show (\ref{eq:dirichlet-case-goal})
for the case $\lambda^D_1 = \frac{\pi}{\omega} \in \mathbb{N}$ and $\omega < \pi$ 
using 
Lemma~\ref{lemma:abstract} and Corollary~\ref{cor:solution-with-polynomial}.  
As in Step~3, choose $(k,\epsilon) \in \BbbN_0 \times (0,1)$ with $ \lambda^D_1 = \frac{\pi}{\omega} < k+\epsilon + 1 < \lambda^D_1 + 1 < \lambda^D_2$ 
such that $\lfloor k+\epsilon + 1\rfloor = \lambda^D_1$ 
and take with this choice of $k$, $\epsilon$ the spaces $X_0$, $X_1$, $Y_0$, $Y_1$ as in (\ref{eq:X0X1Y0Y1-dirichlet})
and set $\theta_1 = \frac{\pi}{\omega (k+\epsilon+1)}$
so that $B^{\pi/\omega-1}_{2,1}(\mathcal{C}_R) = (X_0,X_1)_{\theta_1,1}$ and 
$B^{\pi/\omega+1}_{2,1}(\mathcal{C}_R) = (Y_0,Y_1)_{\theta_1,\infty}$. 
In Corollary~\ref{cor:solution-with-polynomial}, our choice of $k$ corresponds to $\Sigma^D_{k+1} = \{\frac{\pi}{\omega}\}$ and 
$n^\star = \frac{\pi}{\omega} = \lambda^D_1$. Again, $\widetilde{T}: X_0 \rightarrow Y_0$ is bounded by Lax-Milgram. 
With the functions $u_0$, $S^D$, $s^D_1$, $P_{k-1}$ from Corollary~\ref{cor:solution-with-polynomial}, 
Corollary~\ref{cor:solution-with-polynomial} gives the decomposition $\widetilde{T}f = (u_0(f) + S^D(f) s^D_1) + \chi P_{k-1} 
=: S_1(f) + S_{\theta_1}(f)$. Since $s^D_1$ is a polynomial (and thus smooth), 
Corollary~\ref{cor:solution-with-polynomial} asserts the boundedness of $S_1: X_1 \rightarrow Y_1$ and $S_{\theta_1}: X_{\theta_1,1} \rightarrow Y_{\theta_1,\infty}$. 
Lemma~\ref{lemma:abstract} then implies (\ref{eq:dirichlet-case-goal}). 

\textit{Step 5:} We show (\ref{eq:dirichlet-case-goal})
for the case ${\omega} > \pi$.  The procedure is similar to that of the preceding Steps~{3}, {4}. Since $\omega \in (\pi,2\pi)$, 
we have $\lambda^D_1 = \frac{\pi}{\omega} < 1 < 2 \frac{\pi}{\omega} = \lambda^D_2$ so that 
we may select $k = 0$ and $\epsilon \in (0,1)$ such that  $\lambda^D_1 < 1 < k + \epsilon + 1< \lambda^D_2$.  
Then $\Sigma^D_{k+1} = \emptyset$, $P_{k-1} \equiv 0$ in Corollary~\ref{cor:solution-with-polynomial}, and we may argue as in Step~3 with
the choice $X_0$, $X_1$, $Y_0$, $Y_1$ from (\ref{eq:X0X1Y0Y1-dirichlet}) and $\theta_1 = \frac{\pi}{\omega(k+\epsilon+1)}$; note that 
$1/2 <  \pi/\omega < 1$, so that Lemma~\ref{lemma:bounded-linear-besov-2}\ref{item:lemma:bounded-linear-besov-2-ia} implies $S^D_1 \in \left(B^{\pi/\omega-1}_{2,1}(\mathcal{C}_R)\right)^\star
 = \left((X_0,X_1)_{\theta_1,1}\right)^\star$. 
\end{numberedproof}

\section{Neumann boundary conditions}
\label{sec:neumann} 
The case of Neumann boundary conditions is handled in way similar to the Dirichlet case of Section~\ref{sec:regularity-in-weighted-spaces}. 
We define 
\begin{equation}
\sigma^N:=\{\lambda^N_n\,|\,  n = 0,1,\ldots\} 
\qquad \mbox{ with } \lambda^N_n:= n \frac{\pi}{\omega}. 
\end{equation}
Analogous to Proposition~\ref{prop:solution_u1}, we have: 
\begin{proposition}
\label{prop:elliptic-Neumann}
Let $R > 0$. 
Let $k\in\mathbb{N}_0$ and $\epsilon\in (0,1)$ satisfy $k+1+\epsilon < \lambda^N_2 = 2 \frac{\pi}{\omega}$, $k+1+\epsilon \ne \lambda^N_1 = \frac{\pi}{\omega}$, and let $f\in H^{k+\epsilon}(\mathcal{C})$ with $\operatorname{supp} f\subseteq B_1(0)$. Further assume $\partial_x^i \partial_y^j f(0)=0$ for $i+j\leq k-1$. Then every function $u_1\in H^1(\mathcal{C})$ with $\operatorname{supp} u_1 \subseteq B_1(0)$ solving
\begin{align}
\label{eq:elliptic3.25}
-\Delta u_1 &= f \in H^{k+\epsilon}(\mathcal{C}), 
\quad 
\partial_n u_1 = 0 \ \mbox{ for $\phi\in \{0,\omega\}$}
\end{align}
has the form 
\begin{align}
\label{eq:elliptic-Neumann-5}
u_1 = \begin{cases} u_0, & \mbox{ if } \quad \frac{\pi}{\omega} = \lambda^N_1 > k+\epsilon+1 \\ 
      u_0 - \frac{1}{\pi} \left(\int_{\mathcal{C}} r^{-\lambda^N_1} \cos(\lambda^N_1\phi) f(\mathbf{x}) \, d\mathbf{x}\right) r^{\lambda^N_1} \cos (\lambda^N_1\phi), & \mbox{ if } \quad \frac{\pi}{\omega} = \lambda^N_1 < k+\epsilon+1<\lambda^N_2 
\end{cases}
\end{align}
for a $u_0\in H^{k+2+\epsilon}(\mathcal{C}_R)$ with the estimate
\begin{align*}
\|u_0\|_{H^{k+2+\epsilon}(\mathcal{C}_R)} \lesssim \|f\|_{H^{k+\epsilon}(\mathcal{C}_1)}.
\end{align*}
\end{proposition}
\begin{proof} 
The procedure is as in the Dirichlet case of Section~\ref{sec:regularity-in-weighted-spaces}: Mellin transformation yields 
the equations \eqref{eq:Mellin-transform-1}, \eqref{eq:Mellin-transform-2} for ${\mathcal M}[u_1]$ and ${\mathcal M}[u_0]$
together with the Neumann boundary conditions $\partial_\phi {\mathcal M}[u_1] = \partial_\phi {\mathcal M}[u_0] = 0$ on 
$\{\phi = 0\}$ and $\{\phi = \omega\}$. 
The operator ${\mathcal L}(\zeta)$ is meromorphic on $\BbbC$ with poles at 
$
\pm i \sigma^N. 
$
As in the Dirichlet case, $u_0 \in K^{k+2}_{-\epsilon}(\mathcal{C})$ so that $u_0 \in H^{k+1+\epsilon}(\mathcal{C}_R)$. 
Since $0 \in \pm i \sigma^N$, the inverse Mellin transformation cannot be performed on the line $\{\operatorname{Im} \zeta = 0\}$.  
In contrast to the Dirichlet case, where $u_1 \in K^1_0(\Gamma)$ due to the vanishing of $u_1$ on $\{\phi = 0\}$
and $\{\phi = \omega\}$ we only have $u_1 \in K^1_\delta(\mathcal{C})$, $\delta > 0$ arbitrary, in the case of Neumann boundary conditions. 
This implies that the inverse Mellin transformation has to be done on a line $\{\operatorname{Im} \zeta = \delta\}$ for chosen $\delta > 0$. 
The Cauchy integral formula relating $u_0$ and $u_1$ now uses the lines 
$\{\operatorname{Im} \zeta = \delta\}$ and $\{\operatorname{Im} \zeta = - (k+1+\epsilon)\}$ and leads to 
\begin{align*}
u_0-u_1 = \sum_{\substack{\zeta_0\in -i\sigma^N\\ \operatorname{Im}\zeta_0\in (-1-k-\epsilon,\delta)}} \frac{2\pi i}{\sqrt{2\pi}} \operatorname*{Res}_{\zeta = \zeta_0} \left(r^{i\zeta}\left(\mathcal{L}(\zeta)\right)^{-1} \mathcal{M}g(\zeta)\right).
\end{align*}
For the evaluation of the residues, we note that the double pole of ${\mathcal L}^{-1}$ at $\zeta_0=0$ leads to two contributions to the sum; 
if $k + 1 + \epsilon > \pi/\omega$, then a third contribution arises in the sum. The residues can be evaluated explicitly: 
\begin{align*}
\text{at $\zeta_0 = 0$:} & &&\left(\frac{i}{\omega}\frac{1}{\sqrt{2\pi}} \int_{\mathcal{C}} f \,d \mathbf{x}\right) \ln r + 
                           \left(\frac{i}{\omega}\frac{1}{\sqrt{2\pi}} \int_{\mathcal{C}} \ln r f \,d \mathbf{x}\right) 1 =: s_0 + s_1, \\
\text{at $\zeta_0 = -i \lambda^N_1$:} & && -\left(\frac{1}{\pi} \int_{\mathcal{C}} r^{-\lambda_1^N} \cos (\lambda_1^N \phi) f\, d\mathbf{x}\right) r^{\lambda_1^N} \cos (\lambda_1^N \phi). 
\end{align*}
By assumption $u_1 \in H^1(\mathcal{C})$ so that the contribution $s_0$ has to vanish. The contribution $s_1$ is a constant function and hence 
smooth. Additionally, the function $(x,y) \mapsto \ln r$ is in $B^{1}_{2,\infty}(\mathcal{C}_1)$ 
by Lemma~\ref{lemma:bounded-linear-besov-2} so that $f \mapsto \int_{\mathcal{C}} \ln r f$ is 
a bounded linear functional on $(B^1_{2,\infty}(\mathcal{C}_1))^\star \supset H^{k+\epsilon}(\mathcal{C}_1)$ and thus 
the sum $u_0 + s_1$ is in $H^{k+2+\epsilon}(\mathcal{C}_1)$ with the stated estimate. 
\end{proof}

The Neumann analog of Lemma~\ref{lemma:polynomial-solution} is: 
\begin{lemma}
\label{lemma:polynomial-solution-neumann}
Let $i$, $j$, $k\in\mathbb{N}_0$ with $i+j=k$. Set $\Sigma^N_{k+2}:=\{n \in \{2,\ldots, k+2\}\,|\,  n \frac{\omega}{\pi} \in\BbbN\}$ 
and $S^N_{k+2}:= \operatorname{span} \{ r^n \left(\ln r \cos (n \phi) - \phi \sin (n \phi)\right)\,|\, n \in \Sigma^N_{k+2}\}$. Then 
there is a polynomial $\widetilde p_{i,j}$ of degree $k+2$ and a harmonic function $p^\prime_{i,j} \in S^N_{k+2}$ such that  
$p^N_{i,j}:= \widetilde p_{i,j} + p^\prime_{i,j}$ satisfies 
\begin{align}
\label{eq:lemma:polynomial-solution-neumann}
-\Delta p^N_{i,j} &= x^i y^j \quad \text{on } \mathcal{C}, 
\qquad 
\partial_n p^N_{i,j}|_\Gamma = 0. 
\end{align}
In the special case $\omega = \pi$, the contribution $p^\prime_{i,j}$ may be taken to be zero. 
\end{lemma}
\begin{proof}
The proof is very similar to that of Lemma~\ref{lemma:polynomial-solution}. The correction on $\{\phi = 0\}$ is now done
with polynomials of the form $\operatorname{Im} \sum_{n=0}^{k+2} a_n z^n$, $a_n \in \BbbR$. For the correction 
on $\{\phi = \omega\}$ one uses the functions $\operatorname{Re} z^n = r^n \cos (n \phi)$ if $n \frac{\omega}{\pi} \not\in \BbbN_0$ and 
the function $\operatorname{Re} (z^n \ln z)= r^n \left( \ln r \cos (n \phi) - \phi \sin (n \phi)\right)$ if $n \frac{\omega }{\pi} \in \BbbN_0$. 
\end{proof}

\begin{corollary}
\label{cor:solution-with-polynomial-neumann}
Let $R > 0$. 
Let $k\in\mathbb{N}_0$ and $\epsilon\in (0,1)$ satisfy $k+1+\epsilon < 2 \frac{\pi}{\omega} = \lambda^N_2$, $k+1+\epsilon \neq \frac{\pi}{\omega} = \lambda^N_1$
and $f\in H^{k+\epsilon}(\mathcal{C})$ with $\operatorname{supp} f\subseteq B_1(0)$. Let $\chi \in C^\infty_0(B_1(0))$ with $\chi\equiv 1$ near the origin. Then every function $u_1\in H^1(\mathcal{C})$ with $\operatorname{supp} u_1 \subseteq B_1(0)$ solving
\begin{align}
\label{eq:cor:solution-with-polynomial-neumann}
-\Delta u_1 &= f \in H^{k+\epsilon}(\mathcal{C}), 
\qquad 
\partial_n u_1 = 0, \ \mbox{ for $\phi\in \{0,\omega\}$},
\end{align}
has the form $u_1 = u_0 + \chi P_{k-1} + \delta$ with 
$u_0\in H^{k+2+\epsilon}(\mathcal{C}_R )$,  
\begin{align}
\label{eq:cor:solution-with-polynomial-1.5-neumann}
\delta & = \begin{cases} 0, &  \mbox{ if } \quad \frac{\pi}{\omega} = \lambda^N_1 > k+\epsilon+1\\ 
                           S^N(f) s_1^N,  &  \mbox{ if } \quad \frac{\pi}{\omega} = \lambda^N_1 < k+\epsilon+1<\lambda^N_2 , 
\end{cases} 
\\
S^N(f) &:= - \frac{1}{\pi} \left(\int_{\mathcal{C}} r^{-\lambda_1^N} \cos(\lambda_1^N\phi) (f(\mathbf{x})+\Delta(\chi(\mathbf{x}) P_{k-1}(\mathbf{x}))) \, d\mathbf{x}
\right), \\
\label{eq:s1N}
s_1^N & := r^{\lambda_1^N} \cos(\lambda_1^N\phi)  ,\\
P_{k-1}(\mathbf{x}) &:=\sum_{i+j\leq k-1} \frac{1}{i!j!} p^N_{i,j}(\mathbf{x})(\partial_x^i \partial_y^j f)(0), 
\end{align}
and the $p^N_{i,j}$ are the fixed functions from Lemma~\ref{lemma:polynomial-solution-neumann}. Furthermore, the following estimates hold: 
\begin{align}
\label{eq:cor:solution-with-polynomial-1.75-neumann}
\|u_0\|_{H^{k+2+\epsilon}(\mathcal{C}_R) } &\lesssim \|f\|_{H^{k+\epsilon}(\mathcal{C})},  \\
\label{eq:cor:solution-with-polynomial-1.75a-neumann}
\|P_{k-1}\|_{H^{k+2+\epsilon}({\mathcal C}_R)} &\lesssim \|f\|_{B^{k}_{2,1}({\mathcal C})} \quad \mbox{ if $\Sigma^N_{k+1}  =\emptyset$}, \\
\label{eq:cor:solution-with-polynomial-1.75b-neumann}
\|P_{k-1}\|_{B^{n^\star+1}_{2,\infty}({\mathcal C}_R)} &\lesssim \|f\|_{B^{k}_{2,1}({\mathcal C})} \quad \mbox{ if $\Sigma^N_{k+1}  \ne \emptyset$}, 
\qquad n^\star:= \min\{n \in \{2,\ldots,k+1\}\,|\,  n \frac{\omega}{\pi} \in \BbbN\}, \\ 
\label{eq:cor:solution-with-polynomial-1.75c-neumann}
\|\Delta (\chi P_{k-1})\|_{H^{k+\epsilon}({\mathcal C}_R)} &\lesssim 
\|f\|_{B^{k}_{2,1}({\mathcal C})} \lesssim \|f\|_{H^{k+\epsilon}({\mathcal C})}.  
\end{align}
The implied constants depend only $k$, $\epsilon$, the angle $\omega$, and the choice of the cut-off function $\chi$.
\end{corollary}
\begin{proof}
Follows in the same way as that of Corollary~\ref{cor:solution-with-polynomial}. 
\end{proof}
\begin{numberedproof}{of Theorem~\ref{thm:shift-theorem-local-version}\ref{item:thm:shift-theorem-local-version-neumann}:}
The proof follows the strategy developed for the case of Dirichlet boundary conditions; the main difference lies in the 
fact that the basic stability estimate is now $\widetilde{T}:\widetilde{H}^{-1}(\mathcal{C}_R) \rightarrow H^1(\mathcal{C}_R)$
instead of $\widetilde{T}:H^{-1}(\mathcal{C}_R) \rightarrow H^1(\mathcal{C}_R)$ for the Dirichlet case. 

The regularity assertions of Step~0 still hold. The mapping $T$ has to be replaced with 
$T: \widetilde{H}^{-1}(\mathcal{C}_R) \rightarrow H^1_D(\mathcal{C}_R)$ where $Tf$ solves
$$
-\Delta Tf = f \quad \mbox{ in $\mathcal{C}_R$}, \qquad \partial_n (Tf) = 0\quad \mbox{ on $\Gamma_N$}, \quad Tf = 0 \quad \mbox{ on $\widetilde\Gamma_R$}
$$
and $\Gamma_N = \{\phi = 0\} \cup \{\phi = \omega\}$. 
We set $\widetilde{T}:= (\chi_{R'} T \chi_{\overline{R}})$ as in Step~2 of the Dirichlet case. 
The modified right-hand side $\widetilde{f}$ is defined as in (\ref{eq:thm:shift-theorem-local-version-5}), and it suffices to 
ascertain
\begin{align}
\label{eq:neumann-case-goal}
\|\widetilde{T} \widetilde{f} \|_{B^{\pi/\omega+1}_{2,\infty}(\mathcal{C}_{R'})} &\lesssim 
\|\widetilde{f}\|_{B^{\pi/\omega-1}_{2,1}(\mathcal{C}_R)}.  
\end{align}
As in the proof of
Theorem~\ref{thm:shift-theorem-local-version}\ref{item:thm:shift-theorem-local-version-dirichlet} for the Dirichlet case, we distinguish between
the cases $\omega < \pi$ and $\omega > \pi$, the case $\omega = \pi$ having already been discussed in Remark~\ref{rem:omega=pi}. 

\emph{Step~1 (preliminaries):} Let $\omega < \pi$. Then, $\lambda^N_1 > 1$. Select $k = 0$, $\epsilon \in (0,1)$ such that $1 < s':= k+\epsilon +1 < \lambda^N_1$.  
Corollary~\ref{cor:solution-with-polynomial-neumann} 
asserts that $\widetilde{T}:H^{s'-1}(\mathcal{C}_R) \rightarrow H^{s'+1}(\mathcal{C}_R)$ is a bounded operator. 

\emph{Step~2:} Let $\omega < \pi$ and $\pi/\omega \not\in\BbbN$. 
Select $(k,\epsilon) \in \BbbN_0 \times (0,1)$ with $\lambda^N_1 < k+\epsilon + 1 < \lambda^N_2$ 
such that $\lfloor k+\epsilon+1\rfloor = \lfloor \lambda^N_1\rfloor \ge 1$. 
Take with $s'$ from Step~1
\label{eq:X0X1Y0Y1-neumann}
\begin{align}
X_0 &:= H^{s'-1}(\mathcal{C}_R), & X_1 &= H^{k+\epsilon}(\mathcal{C}_R) 
& 
Y_0 &:= H^{s'+1}(\mathcal{C}_R), & Y_1 &= H^{2+k+\epsilon}(\mathcal{C}_R). 
\end{align}
Note that with $\theta_1:= (\pi/\omega -1 - (s'-1))/(k+\epsilon-(s'-1)) \in (0,1)$, we have 
$B^{\pi/\omega-1}_{2,1}(\mathcal{C}_R) = (X_0,X_1)_{\theta_1,1}$ and 
$B^{\pi/\omega+1}_{2,\infty}(\mathcal{C}_R) = (Y_0,Y_1)_{\theta_1,\infty}$.  
Then, $\widetilde{T}$ satisfies the assumptions of Lemma~\ref{lemma:abstract}: the mapping properties $X_0 \rightarrow Y_0$ are 
given by the above Step~1, and the mapping properties of the decomposition of $\widetilde{T} f = (u_0(f) + \chi P_{k-1}) + S^N_1(f) s^N_1 =: S_1(f) + S_{\theta_1}(f)$ for arguments $f \in X_1$ 
is provided by Corollary~\ref{cor:solution-with-polynomial-neumann} in conjunction with 
Lemma~\ref{lemma:bounded-linear-besov-2}\ref{item:lemma:bounded-linear-besov-2-i}. 

\emph{Step~3:} Let $\omega < \pi$ and $\pi/\omega \in\BbbN$. Choose $k$, $\epsilon$, $X_0$, $X_1$, $Y_0$, $Y_1$
as in Step~2 above. As in the Dirichlet case, choose the decomposition 
$\widetilde{T} f = (u_0(f) + S^N_1(f) s^N_1) + \chi P_{k-1} =: S_1(f) + S_\theta(f)$; Corollary~\ref{cor:solution-with-polynomial-neumann}
provides that $S_1$, $S_{\theta_1}$ satisfy the assumptions of Lemma~\ref{lemma:abstract}. 

\emph{Step~4:} Let $\omega > \pi$. Since $\omega \in (\pi,2\pi)$, we have $\lambda^N_1 = \frac{\pi}{\omega} < 1 < 2 \frac{\pi}{\omega} = \lambda^N_2$ 
and may therefore select $k = 0$ and $\epsilon \in (0,1/2)$ such that 
$\lambda^N_1 < 1 < k +\epsilon +1 < \lambda^N_2$. Set  
\label{eq:X0X1Y0Y1-neumann-}
\begin{align}
X_0 &:= \widetilde{H}^{-1}(\mathcal{C}_R), & X_1 &= H^{k+\epsilon}(\mathcal{C}_R) = \widetilde{H}^{k+\epsilon}(\mathcal{C}_R), 
& 
Y_0 &:= H^{1}(\mathcal{C}_R), & Y_1 &= H^{2+k+\epsilon}(\mathcal{C}_R), 
\end{align}
and $\theta_1 = (\pi/\omega-1+1)/(k+\epsilon+1) \in (0,1)$ so that 
$(X_0,X_1)_{\theta_1,1} = \widetilde{B}^{\pi/\omega-1}_{2,1}(\mathcal{C}_R)$ and 
$(Y_0,Y_1)_{\theta_1,1} = {B}^{\pi/\omega+1}_{2,1}(\mathcal{C}_R)$. As in the Dirichlet case, we note $\Sigma^N_{k+1} = \emptyset$
and that Corollary~\ref{cor:solution-with-polynomial-neumann} provides the assumptions of Lemma~\ref{lemma:abstract}
so that $\widetilde{T}:\widetilde{B}^{\pi/\omega-1}_{2,1}(\mathcal{C}_R) \rightarrow B^{\pi/\omega+1}_{2,\infty}(\mathcal{C}_R)$. 
Since $\omega \in (\pi,2\pi)$, we have $\pi/\omega - 1 \in (-1/2,0)$ 
and thus $\widetilde{B}^{\pi/\omega-1}_{2,1}(\mathcal{C}_R) = B^{\pi/\omega-1}_{2,1}(\mathcal{C}_R)$, cf.\ (\ref{eq:B=tildeB}). 
\end{numberedproof}
\section{Mixed boundary conditions}
\label{sec:mixed}
The case of mixed boundary conditions is similar to the Neumann case. 
We recall that $\Gamma_D = \{\phi = 0\}$, $\Gamma_N = \{\phi =\omega\}$. 
Set 
\begin{equation}
\sigma^M:=\{\lambda^M_n\,|\,  n \in \BbbN\}
\qquad \mbox{ with} \quad \lambda^M_n:= (n-1/2) \frac{\pi}{\omega}. 
\end{equation}
The operator $(\mathcal{L}(\zeta))^{-1}$ arising in the case of mixed boundary conditions is meromorphic on $\mathbb{C}$ with poles at $\pm i \sigma^M$. With similar arguments as in 
Proposition~\ref{prop:solution_u1} one obtains: 

\begin{proposition}
\label{prop:solution_u1-mixed}
Let $R > 0$. 
For $k\in\mathbb{N}_0$ and $\epsilon\in (0,1)$, let $k+1+\epsilon < \lambda^M_3$, $k+1+\epsilon \not\in \{\lambda^M_1, \lambda^M_2\} = \{\frac{\pi}{2 \omega}, \frac{3\pi}{2\omega}\}$, and $f\in H^{k+\epsilon}(\mathcal{C})$ with $\operatorname{supp} f\subseteq B_1(0)$. Further assume $\partial_x^i \partial_y^j f(0)=0$ for $i+j \leq k-1$. Then $u_1\in H^1(\mathcal{C})$ with $\operatorname{supp} u_1 \subseteq B_1(0)$ solving
\begin{align}
-\Delta u_1 &= f \in H^{k+\epsilon}(\mathcal{C}), 
\qquad 
u_1 = 0 \mbox{ on $\{\phi = 0\}$}, 
\quad 
\partial_n u_1 = 0 \mbox{ on $\{\phi = \omega\}$}, 
\end{align}
has the form 
\begin{align}
\label{eq:elliptic-mixed-5}
u_1 = 
\begin{cases} u_0 & \mbox{ if } \quad k+\epsilon+1<\frac{\pi}{2\omega} = \lambda^M_1 \\ 
u_0 +S_1^M(f)s_1^M, & \mbox{ if } \quad \lambda^M_1 = \frac{\pi}{2\omega}<k+\epsilon+1 <\frac{3\pi}{2\omega}  = \lambda^M_2\\ 
u_0 +S_1^M(f)s_1^M + S_2^M(f)s_2^M, & \mbox{ if } \quad \frac{3\pi}{2\omega} = \lambda^M_2 < k+\epsilon+1< \lambda^M_3, 
\end{cases}
\end{align}
where
\begin{subequations}
\label{eq:singularity-fcts-mixed}
\begin{align}
S_1^M(f)&=- \frac{1}{\pi} \left(\int_{\mathcal{C}} r^{\frac{-\pi}{2\omega}} \sin(\frac{\pi}{2\omega}\phi) f(x) \, dx\right), \quad s_1^M = r^{\frac{\pi}{2\omega}} \sin(\frac{\pi}{2\omega}\phi),  \\
S_2^M(f)&=- \frac{1}{\pi} \left(\int_{\mathcal{C}} r^{\frac{-3\pi}{2\omega}} \sin(\frac{3\pi}{2\omega}\phi) f(x) \, dx\right), \quad s_2^M = r^{\frac{3\pi}{2\omega}} \sin(\frac{3\pi}{2\omega}\phi)
\end{align}
\end{subequations}
together with the estimate 
\begin{align*}
\|u_0\|_{H^{k+2+\epsilon}(\mathcal{C}_R)} \lesssim \|f\|_{H^{k+\epsilon}(\mathcal{C}_1)}. 
\end{align*}
\end{proposition}

\begin{lemma}
\label{lemma:polynomial-solution-mixed}
Let $i$, $j$, $k\in\mathbb{N}_0$ with $i+j=k$. Set $\Sigma^M_{k+2}:=\{n \in \{1,\ldots, k+2\}\,|\,  n \frac{\omega}{\pi} + \frac{1}{2} \in\BbbN\}$ 
and $S^M_{k+2}:= \operatorname{span} \{ r^n \left(\ln r \sin (n \phi) + \phi \cos (n \phi)\right)\,|\, n \in \Sigma^M_{k+2}\}$. Then 
there is a polynomial $\widetilde p_{i,j}$ of degree $k+2$ and a harmonic function $p^\prime_{i,j} \in S^M_{k+2}$ such that  
$p^M_{i,j}:= \widetilde p_{i,j} + p^\prime_{i,j}$ satisfies 
\begin{align}
\label{eq:lemma:polynomial-solution}
-\Delta p^M_{i,j} &= x^i y^j \quad \text{on } \mathcal{C} , 
\qquad p^M_{i,j} = 0 \mbox{ on $\{\phi = 0\}$,} \qquad 
\partial_n p^M_{i,j} = 0 \mbox{ on $\{\phi = \omega\}$}. 
\end{align}
\end{lemma}
\begin{proof}
One proceeds as in the proof of Lemma~\ref{lemma:polynomial-solution}, the only difference being the correction 
on the line $\{\phi = \omega\}$. For that, one uses, for $n \ge 1$, the functions $\operatorname{Im} z^n$ if $n \frac{\omega}{\pi} + \frac{1}{2}\not\in \BbbN$ 
and $\operatorname{Im} (z^n \ln z)$ if $n\frac{\omega}{\pi} +\frac{1}{2} \in \BbbN$. 
\end{proof}
\begin{corollary}
\label{cor:solution-with-polynomial-mixed}
Let $R > 0$. 
Let $k\in\mathbb{N}_0$ and $\epsilon\in (0,1)$ satisfy $k+1+\epsilon < \lambda^M_3 = \frac{5\pi}{2\omega}$ and $k+1+\epsilon \not\in \{\lambda^M_1,\lambda^M_2\} = 
\{\frac{\pi}{2\omega}, \frac{3\pi}{2\omega}\}$. 
Let $f\in H^{k+\epsilon}(\mathcal{C})$ with $\operatorname{supp} f\subseteq B_1(0)$. Let $\chi \in C^\infty_0(B_1(0))$ with $\chi\equiv 1$ near the origin. Then every function $u_1\in H^1(\mathcal{C})$ with $\operatorname{supp} u_1 \subseteq B_1(0)$ solving
\begin{align}
\label{eq:cor:solution-with-polynomial-mixed}
-\Delta u_1 &= f \in H^{k+\epsilon}(\mathcal{C}), 
\qquad 
u_1 = 0  \mbox{ on $ \{\phi=0\}$}, \quad  \partial_n u_1  = 0 \mbox{ on $\{\phi = \omega\}$}
\end{align}
has the form $u_1 = u_0 + \chi P_{k-1} + \delta$ with 
$u_0\in H^{k+2+\epsilon}(\mathcal{C}_R )$,  
\begin{align}
\label{eq:cor:solution-with-polynomial-1.5-mixed}
\delta & = \begin{cases} 0, &  \mbox{ if } \quad k+\epsilon+1<\lambda^M_1 = \frac{\pi}{2\omega} \\ 
                           S_1^M(f + \Delta (\chi P_{k-1})) s_1^M, & \mbox{ if } \quad  \lambda^M_1 < k+\epsilon+1< \lambda^M_2 \\
                           S_1^M(f + \Delta (\chi P_{k-1})) s_1^M +  S_2^M(f+ \Delta (\chi P_{k-1})) s_2^M,&  \mbox{ if } \quad \lambda^M_2 < k+\epsilon+1< \lambda^M_3 
\end{cases} 
\\
P_{k-1}(\mathbf{x}) &:=\sum_{i+j\leq k-1} \frac{1}{i!j!} p^M_{i,j}(\mathbf{x})(\partial_x^i \partial_y^j f)(0), 
\end{align}
where $S_1^M$, $S_2^M$, $s_1^M$, $s_2^M$ are given in (\ref{eq:singularity-fcts-mixed}), 
and the $p^M_{i,j}$ are the fixed functions from Lemma~\ref{lemma:polynomial-solution-mixed}. Furthermore, 
\begin{align}
\label{eq:cor:solution-with-polynomial-1.75-mixed}
\|u_0\|_{H^{k+2+\epsilon}(\mathcal{C}_R) } &\lesssim \|f\|_{H^{k+\epsilon}(\mathcal{C}_1)},  \\
\label{eq:cor:solution-with-polynomial-1.75a-mixed}
\|P_{k-1}\|_{H^{k+2+\epsilon}({\mathcal C}_R)} &\lesssim \|f\|_{B^{k}_{2,1}({\mathcal C}_1)} \quad \mbox{ if $\Sigma^M_{k+1}  =\emptyset$}, \\
\label{eq:cor:solution-with-polynomial-1.75b-mixed}
\|P_{k-1}\|_{B^{n^\star+1}_{2,\infty}({\mathcal C}_R)} &\lesssim \|f\|_{B^{k}_{2,1}({\mathcal C}_1)} \quad \mbox{ if $\Sigma^M_{k+1}  \ne \emptyset$}, 
\quad n^\star:= \min\{n \in \{1,\ldots,k+1\}\,|\,  n\frac{\omega}{\pi} + \frac{1}{2} \in \BbbN\}, \\ 
\label{eq:cor:solution-with-polynomial-1.75c-mixed}
\|\Delta (\chi P_{k-1})\|_{H^{k+\epsilon}({\mathcal C}_R)} &\lesssim \|f\|_{B^{k}_{2,1}({\mathcal C}_1)} \lesssim \|f\|_{H^{k+\epsilon}({\mathcal C})}.  
\end{align}
The implied constants depend only $k$, $\epsilon$, the angle $\omega$, and the choice of the cut-off function $\chi$.
\end{corollary}

\begin{proof}
Follows as in the Dirichlet case. 
\end{proof}
As in the Neumann case, the proof of Theorem~\ref{thm:shift-theorem-local-version}\ref{item:thm:shift-theorem-local-version-mixed} comes down to showing
the mapping property $\widetilde{T}:B^{\pi/(2\omega)-1}_{2,1}(\mathcal{C}_R) \rightarrow B^{\pi/(2\omega)+1}_{2,\infty}(\mathcal{C}_R)$ for $\pi/(2\omega) -1 > -1/2$ 
and 
$\widetilde{T}:\widetilde{B}^{\pi/(2\omega)-1}_{2,1}(\mathcal{C}_R) \rightarrow B^{\pi/(2\omega)+1}_{2,\infty}(\mathcal{C}_R)$ for $\pi/(2\omega) -1 \leq -1/2$, 
where $\widetilde{T} = \chi_{R'} T \chi_{\overline{R}}$ with $Tf$ being the  solution of 
$$
-\Delta Tf = 0 \quad \mbox{ in $\mathcal{C}_R$}, \qquad Tf = 0 \quad \mbox{ on $\Gamma_{0,R} \cup \widetilde \Gamma_R$}, 
\quad \partial_n Tf = 0 \quad \mbox{ on $\Gamma_{\omega,R}$.}
$$

\begin{numberedproof}{of Theorem~\ref{thm:shift-theorem-local-version}\ref{item:thm:shift-theorem-local-version-mixed} for $\omega < \pi/2$:}
Since $\omega < \pi/2$, we may select $(k,\epsilon) \in \BbbN_0 \times (0,1)$ so that 
$1 < \lambda^M_1 < k+1+\epsilon < \lambda^M_2$ and $\lfloor k+\epsilon+1\rfloor = \lfloor \lambda^M_1\rfloor$. This is the setting, where exactly 
one singularity function appears in the representation of the solution in Corollary~\ref{cor:solution-with-polynomial-mixed} 
and the arguments of the Neumann case (Steps~2, 3 of the proof of Theorem~\ref{thm:shift-theorem-local-version}\ref{item:thm:shift-theorem-local-version-neumann}) 
apply.  
\end{numberedproof}

\begin{numberedproof}{of Theorem~\ref{thm:shift-theorem-local-version}\ref{item:thm:shift-theorem-local-version-mixed} for $\omega = \pi/2$:} 
Select $k = 0$, $\epsilon \in (0,1/2)$ such that $1 = \lambda^M_1 < k+\epsilon+1 < \lambda^M_2 = 3$. In Corollary~\ref{cor:solution-with-polynomial-mixed},
this corresponds to the case of one ``singularity'' function $S^M_1(f) s^M_1$, which is in fact a polynomial and thus smooth, 
and $P_{k-1}  = 0$. 
Lemma~\ref{lemma:bounded-linear-besov-2}\ref{item:lemma:bounded-linear-besov-2-ia} provides 
$S^M_1 \in \left(\widetilde{B}^0_{2,1}(\mathcal{C}_R)\right)^\star$. 
Set $X_0 = \widetilde{H}^{-1}(\mathcal{C}_R)$, $X_1 = \widetilde{H}^{\epsilon}(\mathcal{C}_R)$, 
$Y_0 = {H}^{1}(\mathcal{C}_R)$, $Y_1 = {H}^{2+\epsilon}(\mathcal{C}_R)$, $\theta = 1/(1+\epsilon)$. 
From Lemma~\ref{lemma:abstract}
and the solution representation of Corollary~\ref{cor:solution-with-polynomial-mixed} 
we infer $\widetilde{T}: \widetilde{B}^{\pi/(2\omega)-1}_{2,1} \rightarrow B^{\pi/(2\omega)+1}_{2,\infty}$. 
\end{numberedproof}

\begin{numberedproof}{of Theorem~\ref{thm:shift-theorem-local-version}\ref{item:thm:shift-theorem-local-version-mixed} for $\omega \in (\pi/2, 3\pi/2)$:} 
In this case, one can select $k = 0$ and $\epsilon \in (0,1)$ such that $\lambda^M_1 <1 < k+\epsilon+1 < \lambda^M_2$ so that in the application
of Corollary~\ref{cor:solution-with-polynomial-mixed} a single singularity function arises. This is handled as in the Neumann case 
with $\omega > \pi$ there. 
\end{numberedproof} 

\begin{numberedproof}{of Theorem~\ref{thm:shift-theorem-local-version}\ref{item:thm:shift-theorem-local-version-mixed} for $\omega \in (3\pi/2,2\pi)$:} 
We have $\lambda^M_1 < \lambda^M_2 < 1 < \frac{5}{4} < \lambda^M_3$. Select $k = 0$ and $\epsilon \in (0,1/2)$ such that 
$\lambda^M_2 < 1 < k + \epsilon+  1 < \lambda^M_3$. Set 
$X_0 = \widetilde{H}^{-1}(\mathcal{C}_R)$, 
$X_1 = \widetilde{H}^{\epsilon}(\mathcal{C}_R)$, 
$Y_0 = {H}^{1}(\mathcal{C}_R)$, 
$Y_1 = {H}^{2+k+\epsilon}(\mathcal{C}_R)$, $\theta_j = \lambda^M_j/(1+ k + \epsilon)$, $j \in \{1,2\}$, 
so that $X_{\theta_j,1} = \widetilde{B}^{\lambda^M_j-1}_{2,1}(\mathcal{C}_R)$ and 
$Y_{\theta_j,1} = {B}^{\lambda^M_j+1}_{2,1}(\mathcal{C}_R)$. 
Note 
$P_{k-1} \equiv 0$ in Corollary~\ref{cor:solution-with-polynomial-mixed}.  Corollary~\ref{cor:solution-with-polynomial-mixed}
yields for $f \in X_1$ the decomposition $\widetilde{T} f = u_0(f) + S^M_1(f) s^M_1 + S^M_2(f) s^M_2=: S_1(f) + S_{\theta_1}(f) + S_{\theta_2}(f)$
with $S_{\theta_j}: X_{\theta_j,1} \rightarrow Y_{\theta_j,\infty}$ (cf.\ Lemma~\ref{lemma:bounded-linear-besov-2}). The desired result now follows 
from Lemma~\ref{lemma:abstract}. 
\end{numberedproof}

\begin{numberedproof}{of Theorem~\ref{thm:shift-theorem-local-version}\ref{item:thm:shift-theorem-local-version-mixed} for $\omega = 3\pi/2$:} 
In contrast to the preceding case $\omega > 3\pi/2$, we have 
$\lambda^M_1 < \lambda^M_2 = 1 < \frac{5}{4} < \lambda^M_3$. We  
take $(k,\epsilon) = (0,\epsilon)$ as in the preceding case. 
The function 
$s^M_2$ is a polynomial and hence smooth. Corollary~\ref{cor:solution-with-polynomial-mixed} yields for $f \in X_1$ the decomposition
$\widetilde{T} f = \left(u_0(f) + S^M_2(f) s^M_2\right) +  S^M_1(f) s^M_1 =: S_1(f) + S_{\theta_1}(f)$.  
Lemma~\ref{lemma:bounded-linear-besov-2} shows that the operators $S_1$ and $S_{\theta_1}$ have the mapping properties required by 
Lemma~\ref{lemma:abstract}. This concludes the proof. 
\end{numberedproof}

\section{An extension to $L^p$-based Besov spaces}
\label{sec:Lp}
The Besov spaces $B^{s}_{2,q}$ considered so far are based on $L^2$-spaces. 
Several results obtained in the framework of these Besov spaces can be generalized
to $L^p$-based Besov spaces. We illustrate this in the present section for the 
case simplest case, that of Dirichlet boundary condition and the assumption of $W^{2,p}$-regularity.

For $k \in \BbbN_0$ and $p \in (1,\infty)$ and bounded domains $\Omega$, we define the spaces 
$W^{k,p}(\Omega)$ in the usual way by requiring that derivatives up to order $k$ be in $L^p(\Omega)$, 
\cite{Grisvard_2011}. 
For $s \ge 0$ with $s \not\in \BbbN_0$ one defines $W^{s,p}(\Omega):= (W^{\lfloor s \rfloor,p}(\Omega),W^{\lceil s\rceil ,p})_{s - \lfloor s \rfloor ,p}$. 
The norm $\|\cdot\|_{W^{s,p}(\Omega)}$ is in fact equivalent to the Aronstein-Slobodeckij norm, \cite[Chap.~{36}]{Tartar_2007}. 
A second important fact is that the Reiteration Theorem, \cite[Chap.~{26}]{Tartar_2007} allows one to show that 
for $k$, $m \in \BbbN_0$ and $k < s < m$ with $s \not\in \BbbN_0$ one has 
$W^{s,p}(\Omega) = (W^{k,p}(\Omega), W^{m,p}(\Omega))_{(s-k)/(m-k),p}$, \cite[Chap.~{34}]{Tartar_2007}. 
Analogous to the case of Sobolev spaces, one defines for $s \ge 0$, $s \not\in \BbbN_0$, Besov spaces $B^{s}_{p,q}(\Omega)$ 
by interpolation 
\begin{align*}
B^s_{p,q}(\Omega) = (W^{\lfloor{s}\rfloor ,p}(\Omega), W^{\lceil{s}\rceil,p}(\Omega))_{s - \lfloor{s}\rfloor, q} 
\qquad 
\end{align*}
and notes that the Reiteration theorem would allow us to represent these spaces by interpolating between the spaces corresponding
to $s_1$, $s_2$ with $0 \leq s_1 < s < s_2$, \cite[Chap.~{24}]{Tartar_2007}, \cite[Sec.~{4.3.2}]{Triebel2ndEd}. 
As is customary in connection with $L^p$ spaces, we denote by $p' = p/(p-1)$ the conjugate exponent.  

Analogous to the spaces $K^{s}_\gamma ({\mathcal C})$, it is convenient to define 
$K^{s,p}_\gamma({\mathcal C})$ by the norm 
\begin{equation}
\|u\|_{K_\gamma^{s,p}(\mathcal{C})}^p := \sum_{|\alpha|\leq s} \|r^{|\alpha|-s+\gamma} D^\alpha u\|_{L^p(\mathcal{C})}^p.
\end{equation}

We will consider solutions $u \in W^{1,p}({\mathcal C})$ of 
\begin{align}
\label{eq:Lp-Dirichlet}
-\Delta u &= f \quad \mbox{ in ${\mathcal C}$,} \quad  u = 0 \ \mbox{ on $\phi  \in \{0,\omega\}$}. 
\end{align}
\subsection{Regularity of the singularity functions and stress intensity functionals}
Let $k \in \BbbN_0$. If $k > 2/p$, then by \cite[Thm.~{4.6.1}]{Triebel2ndEd}, we have the embedding 
$W^{k,p}({\mathcal C}_R) \subset C^{\lfloor k-2/p\rfloor}({\mathcal C}_R)$
so that one may define as in (\ref{eq:Pkm1}) the polynomial $P_{\lfloor k-2/p\rfloor}$. We have: 
\begin{lemma}
\label{lemma:embedding-Wkp}
For $k \in \BbbN_0$, $p \in (1,2) \cup (2,\infty)$, $f \in W^{k,p}({\mathcal C}_1)$ 
let $P_{\lfloor{k-2/p}\rfloor}$ be given by (\ref{eq:Pkm1}), $\chi \in C^\infty_0(B_1(0))$ with $\chi \equiv 1$ near the origin. Then, 
$f - \Delta (\chi P_{\lfloor k - 2/p\rfloor}) \in K^{k,p}_0({\mathcal C}_1)$. 
\end{lemma}
\begin{proof}
The proof follows structurally that of Lemma~\ref{lemma:Sobolev-to-cone-second-version-appendix}. Note that $k - 2/p$ is not an integer, which 
allows one to avoid the introduction of $\epsilon > 0$ as in Lemma~\ref{lemma:Sobolev-to-cone-second-version-appendix}. 
\end{proof} 
In the following Lemma~\ref{lemma:singularity-fct-Lp}, which allows us to ascertain the regularity of the singularity functions 
and the stress intensity functional, the functions $P_{\lfloor \alpha - 2\rfloor}$ of (\ref{eq:Pkm1}) arise. These functions 
vanish for $\alpha- 2  <0$ and are well-defined for $f \in B^{\lfloor \alpha - 2\rfloor +2/p}_{p,1} \subset C^{\lfloor \alpha-2\rfloor}$ in view of 
\cite[Thm.~{4.6.1}]{Triebel2ndEd}. 
\begin{lemma}
\label{lemma:singularity-fct-Lp}
Let $p \in (1,2) \cup (2,\infty)$. 
\begin{enumerate}[label=(\roman*)]
\item 
\label{item:lemma:singularity-fct-Lp-i}
For $\beta + 2/p > 0$ the function $s^+(r,\phi) = r^\beta \sin(\beta \phi)$ is in the space $B^{\beta+2/p}_{p,\infty}({\mathcal C}_1)$. 
\item 
\label{item:lemma:singularity-fct-Lp-ii}
Let $\Phi \in C^\infty(\BbbR^2)$ with $|\Phi(x,y)| \leq C r^n$ as $r \rightarrow 0$ for some $n \in \BbbN_0$. 
Then the functions $v(x,y) = \Phi(x,y) \ln r$ 
and $w(x,y) = \phi \Phi(x,y)$ are in the space $B^{n+2/p}_{p,\infty}({\mathcal C}_1)$. 
\item 
\label{item:lemma:singularity-fct-Lp-iii}
Let $\alpha - 2/p' >  0$ with $\alpha- 2 \not\in \BbbN_0$. 
Let $P_{\lfloor{ \alpha - 2}\rfloor}$ be given by (\ref{eq:Pkm1}). Then
\begin{align*} 
f \mapsto S(f):= \int_{{\mathcal C}_1} r^{-\alpha} \sin(\alpha \phi) (f + \Delta (\chi P_{\lfloor{\alpha - 2}\rfloor}))
\end{align*}
is bounded linear on $B^{\alpha-2/p'}_{p,1}({\mathcal C}_1)$. 
%
%
\end{enumerate}
\end{lemma}
\begin{proof}
\emph{Proof of 
\ref{item:lemma:singularity-fct-Lp-i}, \ref{item:lemma:singularity-fct-Lp-ii}:}
\ref{item:lemma:singularity-fct-Lp-i} is shown similarly to the proof of 
Lemma~\ref{lemma:bounded-linear-besov-2}\ref{item:lemma:bounded-linear-besov-2-ii} by estimating the $K$-functional through the 
splitting $r^\beta = r^\beta \chi_t + r^\beta (1 - \chi_t)$ for a suitable $t$-dependent cut-off function $\chi_t$. 
\ref{item:lemma:singularity-fct-Lp-ii} is shown by appropriately modifying the proof of 
Lemma~\ref{lemma:bounded-linear-besov-2}\ref{item:lemma:bounded-linear-besov-2-iii}. 

\emph{Proof of \ref{item:lemma:singularity-fct-Lp-iii}:}
The proof follows that of 
Lemma~\ref{lemma:bounded-linear-besov-2}\ref{item:lemma:bounded-linear-besov-2-i} in that $B^{\alpha - 2/p'}_{p,1}$ is suitably written
as an interpolation space. We distinguish the cases $0 < \alpha < 2$ and $\alpha >  2$. 

\emph{The case $\alpha < 2$:} Then $P_{\lfloor \alpha - 2\rfloor} \equiv 0$. We write for $\epsilon > 0$ so small that $\alpha + \epsilon < 2$
\begin{align*} 
B^{\alpha-2/p'}_{p,1}({\mathcal C}_1) = (
L^{p}({\mathcal C}_1), B^{\alpha-2/p'+\epsilon}_{p,1}({\mathcal C}_1))_{\theta, 1}, 
\qquad \theta = \frac{\alpha - 2/p'} {\alpha - 2/p' +\epsilon }. 
\end{align*}
As in the proof Lemma~\ref{lemma:bounded-linear-besov-2}\ref{item:lemma:bounded-linear-besov-2-i}, one splits 
for $\delta > 0$ 
the expression $S(f) = S_1 + S_2$ as in (\ref{eq:lemma:bounded-linear-besov-2-1}). For $S_1$, the H\"older inequality yields 
\begin{align*}
|S_1| &\leq \|f\|_{L^p(\mathcal{C}_1)} \|\chi_\delta r^{-\alpha}\|_{L^{p'}(\mathcal{C}_1)} \lesssim \delta^{-\alpha + 2/p'} \|f\|_{L^p(\mathcal{C}_1)}. 
\end{align*}
Select $q > 1$ by the condition $2/q = 2/p-(\alpha-2/p'+\epsilon) = 2-\alpha - \epsilon \in (0,2)$. 
By \cite[Thm.~{4.6.1(c)}]{Triebel2ndEd}, we then have 
$B^{\alpha -2/p' + \epsilon}_{p,1}(\mathcal{C}_1) \subset 
B^{\alpha -2/p' + \epsilon}_{p,q}(\mathcal{C}_1) \subset L^q(\mathcal{C}_1)$ so that 
\begin{align*}
|S_2| \leq \|f\|_{L^q(\mathcal{C}_1)} \|(1 - \chi_\delta) r^{-\alpha}\|_{L^{q'}(\mathcal{C}_1)} 
\lesssim \delta^{\alpha + 2/q'} \|f\|_{B^{\alpha-2/p'+\epsilon}_{2,1}(\mathcal{C}_1)} 
\lesssim \delta^{\epsilon} \|f\|_{B^{\alpha-2/p'+\epsilon}_{2,1}(\mathcal{C}_1)} . 
\end{align*}
As in the proof Lemma~\ref{lemma:bounded-linear-besov-2}\ref{item:lemma:bounded-linear-besov-2-i}, we select 
\begin{equation}
\label{eq:lemma:singularity-fct-Lp-delta}
\delta = \min\left \{\frac{1}{2}\operatorname{diam} \{\mathbf{x} \in {\mathcal C}\,|\, \chi(\mathbf{x})  = 1\},  \left( \|f\|_{L^p(\mathcal{C}_1)} \|f\|^{-1} _{B^{\alpha-2/p' +\epsilon}_{2,1}(\mathcal{C}_1)}\right)^{1/(\alpha-2/p' + \epsilon)} \right\}. 
\end{equation}
For brevity, we only consider the case that the minimum in (\ref{eq:lemma:singularity-fct-Lp-delta}) is given by the second term. Then, 
\begin{align*}
S_1 + S_2 &\lesssim \|f\|_{L^p(\mathcal{C}_1)}^{1-\theta} \|f\|_{B^{\alpha-2/p'+\epsilon}_{p,1}(\mathcal{C}_1)}^\theta. 
\end{align*}
An appeal to \cite[Lemma~{25.2}]{Tartar_2007} concludes the proof. 

\emph{The case $\alpha > 2$:} Select $\epsilon > 0$ 
such that $\alpha - 2 - \lfloor \alpha - 2\rfloor +\epsilon  < 1$. 
Since $\lfloor \alpha - 2\rfloor + 2/p < \alpha - 2/p'$, we write 
\begin{align*} 
B^{\alpha-2/p'}_{p,1}({\mathcal C}_1) = (
B^{\lfloor \alpha -2 \rfloor + 2/p}_{p,1}({\mathcal C}_1), B^{\alpha-2/p'+\epsilon}_{p,1}({\mathcal C}_1))_{\theta, 1}, 
\qquad \theta = \frac{\alpha - 2 - \lfloor \alpha -2 \rfloor } {\alpha - 2 +\epsilon - \lfloor \alpha - 2\rfloor }. 
\end{align*}
As in the proof Lemma~\ref{lemma:bounded-linear-besov-2}\ref{item:lemma:bounded-linear-besov-2-i}, one splits 
for $\delta > 0$ 
the expression $S(f) = S_1 + S_2$ as in (\ref{eq:lemma:bounded-linear-besov-2-1}). For $S_1$, we note 
that $B^{\lfloor \alpha - 2\rfloor + 2/p}_{p,1}(\mathcal{C}_1) \subset C^{\lfloor \alpha - 2\rfloor}(\overline{\mathcal{C}_1})$
by \cite[Thm.~{4.6.1(f)}]{Triebel2ndEd}. Hence, 
\begin{align*}
|S_1| &\leq \|r^{-\lfloor \alpha - 2\rfloor} (f + \chi \Delta P_{\lfloor \alpha - 2\rfloor}) \|_{L^\infty(\mathcal{C}_1)} 
\|\chi_\delta r^{-\alpha+\lfloor \alpha - 2\rfloor }\|_{L^{1}(\mathcal{C}_1)} 
\lesssim \delta^{-(\alpha - 2) + \lfloor \alpha - 2 \rfloor } \|f\|_{B^{\lfloor \alpha-2\rfloor + 2/p}_{2,1}(\mathcal{C}_1)}. 
\end{align*}
For $S_2$, we have by \cite[Thm.~{4.6.1(f)}]{Triebel2ndEd} 
that $B^{\alpha - 2/p' +\epsilon}_{p,1}(\mathcal{C}_1)  = B^{\alpha - 2 + 2/p +\epsilon}_{p,1}(\mathcal{C}_1) \subset C^{\alpha - 2+\epsilon}(\overline{\mathcal{C}_1})$
so that in view of the fact that $-\Delta P_{\lfloor \alpha - 2\rfloor}$ is the Taylor expansion of $f$ at the origin 
of order $\lfloor \alpha - 2\rfloor $ and $0 < \alpha - 2 - \lfloor \alpha - 2\rfloor + \epsilon< 1$, we get  
\begin{align*}
|S_2| &\leq \|r^{-(\alpha - 2+\epsilon)} (f + \chi \Delta P_{\lfloor \alpha - 2\rfloor}) \|_{L^\infty(\mathcal{C}_1)} 
\|(1-\chi_\delta) r^{-\alpha+(\alpha - 2)+\epsilon}\|_{L^{1}(\mathcal{C}_1)} 
\lesssim \delta^{\epsilon} \|f\|_{B^{\alpha-2/p' +\epsilon}_{p,1}(\mathcal{C}_1)}. 
\end{align*}
We select $\delta$ as in (\ref{eq:lemma:singularity-fct-Lp-delta}) with the exponent replaced with 
$1/(\alpha - 2 + \epsilon - \lfloor \alpha - 2\rfloor)$. Then, as in the proof of 
Lemma~\ref{lemma:bounded-linear-besov-2}\ref{item:lemma:bounded-linear-besov-2-i}, one arrives at 
\begin{align*}
S_1 + S_2 &\lesssim \|f\|_{B^{\lfloor \alpha - 2\rfloor +2/p}_{p,1}(\mathcal{C}_1)}^{1-\theta} \|f\|_{B^{\alpha-2/p'+\epsilon}_{p,1}(\mathcal{C}_1)}^\theta. 
\end{align*}
An appeal to \cite[Lemma~{25.2}]{Tartar_2007} concludes the proof. 
\end{proof}
\subsection{Expansion in corner singularity functions}
The $L^p$-theory for elliptic problems in domains with conical points that was developed by Maz'ya and Plamenevskii leads to the following 
Proposition~\ref{prop:Lp}. 
\ifarxiv
We point the reader also to Appendix~\ref{appendix:Lp}, where some of the arguments developed 
by Maz'ya and Plamenenvsii are detailed and the main steps of the proof of Proposition~\ref{prop:Lp} are provided.
\fi
\begin{proposition} 
\label{prop:Lp}
Let $R > 0$. 
Let $p \in (1,2) \cup (2,\infty)$, $k \in \BbbN_0$, and assume $k+2-2/p = k + 2/p' \not\in \pm \sigma^D$. 
Let $f \in W^{k,p}({\mathcal C})$ satisfy $\partial^i_x\partial^j_y f(0) = 0$ 
for $i+j < k - 2/p$. Then, a solution $u_1 \in W^{1,p}({\mathcal C})$ with $\operatorname{supp} u_1 \subset B_1(0)$  solving 
\begin{equation}
\label{prop:Lp-10}
-\Delta u_1 = f, \quad u_1 = 0 \mbox{ for $\phi \in \{0,\omega\}$}
\end{equation}
has the form 
\begin{equation*}
u_1 = u_0 - \frac{1}{\pi} \sum_{j\colon \lambda^D_j < k+2-2/p} \int_{\mathcal C} r^{-\lambda^D_j} \sin (\lambda^D_j \phi)  f(x)\,d{\mathbf x} \  r^{\lambda^D_j} \sin (\lambda^D_j \phi)
\end{equation*}
with 
\begin{equation}
\label{prop:Lp-20}
\|u_0\|_{K^{k+2,p}_{0}({\mathcal C}_R)} \lesssim \|f\|_{K^{k,p}_0({\mathcal C})} \stackrel{\text{L.~\ref{lemma:embedding-Wkp}}}{\lesssim }
\|f\|_{W^{k,p}({\mathcal C})}. 
\end{equation}
\end{proposition} 
\begin{proof}
The procedure is similar to that in the $L^2$-based setting in Section~\ref{sec:recap}. Note $\operatorname{supp} f \subset B_1(0)$. 
By density, one may assume $f \in W^{k,p}(\mathcal{C})\cap C^\infty(\overline{\mathcal{C}})$ so that the formulas of the $L^2$-based setting are applicable. 
Since $\operatorname{supp} f \subset B_1(0)$, the Mellin transform ${\mathcal M}(g) = {\mathcal M}(r^2 f)$ is holomorphic in
the strip $\{\zeta \in \BbbC\,|\, \operatorname{Im} > -k - 2 + 2/p \}$. As in Section~\ref{sec:recap}, the function $u_0$ is defined by an appropriate
inverse Mellin transformation on the line $\operatorname{Im} = -k - 2 + 2/p $, and the residue theorem yields 
\begin{align}
\label{eq:representation-formula-Lp-1}
u_0 - u_1 = \sum_{\substack{ \zeta_0 \in -i \sigma^D \colon  \\ \operatorname{Im} \zeta_0 \in (-k-2+2/p,0)}} \operatorname{Res}_{\zeta = \zeta_0} 
\left( r^{i \zeta} ({\mathcal L}(\zeta))^{-1} {\mathcal M}[g](\zeta) \right)
\end{align}
Evaluating the residue yields the claimed representation. The estimate for $u_0$ is taken from \cite[Prop.~{2.3}]{nazarov-plamenevsky-1994} 
but goes back at least to \cite[Thm.~{4.1}]{mazya-plamenevskii78}. 
\end{proof}
As in Section~\ref{sec:recap}, the condition that $f$ vanish to sufficient order can be removed by adding additional polynomials or logarithmic
singularities: 
\begin{corollary}
\label{cor:Lp}
Let $R > 0$.
Let $p \in (1,2) \cup (2,\infty)$, $k \in \BbbN_0$, and assume $ k + 2 - 2/p = k + 2/p' \not\in \pm \sigma^D$. 
Let $f \in W^{k,p}({\mathcal C})$. Then, a solution $u_1 \in W^{1,p}({\mathcal C})$ with $\operatorname{supp} u_1 \subset B_1(0)$  solving 
(\ref{prop:Lp-10})
can be represented as $u_1 = u_0 + \chi P_{\lfloor k-2/p\rfloor} + S$ with
\begin{align*}
S & =  - \frac{1}{\pi} \sum_{j\colon \lambda^D_j < k+2-2/p} \int_{\mathcal C} r^{-\lambda^D_j} \sin(\lambda^D_j)\phi)  (f(x) + \Delta (\chi P_{\lfloor k-2/p\rfloor}) \,d{\mathbf{x}} \ r^{\lambda^D_j} \sin(\lambda^D_j \phi), \\
P_{\lfloor k-2/p\rfloor} & = \sum_{i+j < k-2/p} \frac{1}{i! j!} p^D_{i,j}(\mathbf{x}) \partial^i_x \partial^j_y f(0), 
\end{align*}
where the functions $p^D_{i,j}$ are from Lemma~\ref{lemma:polynomial-solution} and 
\begin{align*}
\|u_0\|_{W^{k+2,p}({\mathcal C}_R)} &\lesssim \|f\|_{W^{k,p}({\mathcal C}_1)}. 
\end{align*}
The function $P_{\lfloor{k-2/p}\rfloor} \equiv 0$ if $k - 2/p < 0$ and satisfies, for $k-2/p > 0$ the estimates 
\begin{align*}
\|P_{\lfloor {k-2/p}\rfloor }\|_{W^{k+2,p}({\mathcal C}_R)} &\lesssim \|f\|_{B^{\lfloor k-2/p\rfloor + 2/p}_{p,1}({\mathcal C}_1)}, 
\quad \mbox{ if $\Sigma^D_{\lfloor{k-2/p}\rfloor + 2} = \emptyset $}, \\
\|P_{\lfloor {k-2/p}\rfloor }\|_{B^{n^\ast +2/p}_{p,\infty}({\mathcal C}_R)} &\lesssim \|f\|_{B^{\lfloor k-2/p \rfloor +2/p}_{p,1}({\mathcal C}_1)}, 
\quad n^\ast:= \min\{n \in \Sigma^D_{\lfloor{k-2/p}\rfloor +2}\,|\,  n\frac{\omega}{\pi} \in \BbbN\}, 
\quad \mbox{ if $\Sigma^D_{\lfloor{k-2/p}\rfloor + 2} \ne  \emptyset $}. 
\end{align*}
\end{corollary}
\begin{proof}
Follows as in Corollary~\ref{cor:solution-with-polynomial}. 
\end{proof} 
\subsection{A shift theorem}
The representation formula of Corollary~\ref{cor:Lp} allows one to infer a shift theorem in Besov spaces as in $L^2$-case: 
\begin{theorem}
\label{thm:Lp}
Let $R > 0$.  Let $p \in (1,2) \cup (2,\infty)$ and $\chi \in C^\infty_0(B_1(0))$ with $\chi \equiv 1$ near $0$. 
Let $k:= \min\{n \in \BbbN\,|\, \lambda^D_1 + \frac{2}{p} < n + 2\}$. 
Assume that one of the following two conditions holds: 
\begin{enumerate} [label=(\roman*)]
\item 
\label{item:thm:Lp-i}
$2 < \lambda^D_1 + \frac{2}{p} < k + 2 < \lambda^D_2 + \frac{2}{p}$. 
\item 
\label{item:thm:Lp-ii}
$k < 2/p$ and $2 < \lambda^D_1 + \frac{2}{p} < k + 2$. 
\end{enumerate} 
Assume furthermore that $k + 2/p' \not\in \pm \sigma^D$. 
Then for $f \in B^{\lambda^D_1+2/p-2}_{2,1}({\mathcal C}_1)$ a solution $u$ of (\ref{eq:Lp-Dirichlet}) satisfies 
\begin{equation*}
\|u\|_{B^{\lambda^D_1 + 2/p}_{2,\infty}({\mathcal C}_R)} \lesssim \|\chi f\|_{B^{\lambda^D_1 + 2/p-2}_{2,1}({\mathcal C}_1)} + \|u\|_{W^{1,p}({\mathcal C}_1)}. 
\end{equation*}
\end{theorem}
\begin{proof}
In the following, we assume $\omega \ne \pi$ since in the case $\omega = \pi$ one has a full shift theorem analogous to the case discussed 
in Remark~\ref{rem:omega=pi}, see \cite[Sec.~9]{Gilbarg}. 

The two conditions \ref{item:thm:Lp-i}, \ref{item:thm:Lp-ii} are such that the procedure already used in the $L^2$ setting is applicable. 
Inspection of the proof of Theorem~\ref{thm:shift-theorem-local-version}\ref{item:thm:shift-theorem-local-version-dirichlet} shows that it relies 
on the following ingredients \ref{item:(a)}--\ref{item:(d)}: 
\begin{enumerate}[wide,label=(\Alph*)] 
\item \label{item:(a)} Local regularity assertions as in Steps~0--1 of that proof that underlie the estimate (\ref{eq:thm:shift-theorem-local-version-10}). 
The local regularity in $L^p$-spaces is available, e.g., \cite[Sec.~{9}]{Gilbarg}. 
\item \label{item:(b)} Solution operators $T$ and $\widetilde{T}$ for the Dirichlet problem as in (\ref{eq:operator-T}).  In contrast 
to (\ref{eq:dirichlet-case-interpolation-couples}), where $X_0=  H^{-1}(\mathcal{C}_R)$, $Y_0 = H^1_0({\mathcal C}_R)$, 
we view $T:L^p \rightarrow W^{2,p} \cap W^{1,p}_0$ and select 
\begin{align}
\label{eq:Lp-case-interpolation-couples} 
X_0 &= L^p(\mathcal{C}_R), 
&
X_1 &= W^{k,p}(\mathcal{C}_R), 
&
Y_0 &= W^{2,p}(\mathcal{C}_R), 
&
Y_1 &= W^{k+2,p}(\mathcal{C}_R), 
\end{align}
where we recall that is taken as the smallest integer with $k + 2 > \lambda^D_1+ 2/p$. 
One then has 
$B^{\lambda^D_1-2/p'}_{p,1} = (X_0,X_1)_{\theta,1}$ and 
$B^{\lambda^D_1+2/p}_{p,\infty} = (Y_0,Y_1)_{\theta,\infty}$ for $\theta = (\lambda^D_1 -2/p')/k$. 

The mapping property $\widetilde{T}: X_0 \rightarrow Y_0$ follows from the assumption $2 < \lambda^D_1 + 2/p$ since this implies
that the sum $S$ in the solution representation in Corollary~\ref{cor:Lp} is empty for $k = 0$ and that the function 
$P_{\lfloor k -2/p\rfloor}$ vanishes for $k=0$ so that $u_1 = u_0$. 
\item \label{item:(c)} The representation formula for the solution for data from $X_1$. This is provided by Corollary~\ref{cor:Lp}. 
\item \label{item:(d)} The interpolation argument of Lemma~\ref{lemma:abstract}. In the $L^2$-setting, it was applied
in two situations: 
  \begin{enumerate*}[label=(\alph*)] 
   \item The sum $S$ contains exactly one singularity function.  
   \item The sum $S$ contains several singularity functions but $P_{k-1} \equiv 0$;  this was the case
         in Theorem~\ref{thm:shift-theorem-local-version}\ref{item:thm:shift-theorem-local-version-mixed}
         for $\omega \in (3\pi/2,2\pi)$. 
  \end{enumerate*}
These two cases correspond to the conditions \ref{item:thm:Lp-i} and \ref{item:thm:Lp-ii}, respectively.
In the remainder of the proof, we discuss in more detail the application of the interpolation argument of Lemma~\ref{lemma:abstract}. 
\end{enumerate}

\emph{Proof under condition \ref{item:thm:Lp-i} and $\lambda^D_1 \not \in \BbbN$:}
The definition of $k$ reads 
\begin{equation}
\label{eq:smallest-k}
k+1-2/p \leq \lambda^D_1 < k+2 - 2/p, 
\end{equation}
 which gives 
\begin{equation}
\label{eq:thm-Lp-10}
\lfloor k+2-2/p\rfloor \ge \lfloor \lambda^D_1 \rfloor \ge \lfloor k+1-2/p\rfloor  = \lfloor k+2-2/p\rfloor - 1. 
\end{equation}
We claim that $\Sigma^D_{\lfloor k-2/p\rfloor + 2} = \emptyset$. To see this, note that
(\ref{eq:thm-Lp-10}) implies $\lfloor k-2/p\rfloor + 2 = \lfloor k + 2 -2/p \rfloor \leq \lfloor \lambda^D_1 \rfloor + 1$ so that 
for any $n \in \BbbN$ with $n \leq \lfloor k -2/p\rfloor + 2$ we have 
\begin{equation*}
\frac{n}{\lambda^D_1} \leq \frac{1+\lfloor \lambda^D_1\rfloor}{\lambda^D_1} \leq 1 + \frac{1}{\lambda^D_1} < 3. 
\end{equation*}
For $n \in \Sigma^D_{\lfloor k - 2/p \rfloor + 2}$, one has $n/\lambda^D_1 \in \BbbN$, leading to the possible cases $n= \lambda^D_1$ and 
$n = 2 \lambda^D_1$. The first case is excluded by $\lambda^D_1 \not\in \BbbN$. The second case is also excluded since
otherwise $2 \lambda^D_1  = n \leq \lfloor{k-2/p}\rfloor + 2 \leq 1 + \lfloor \lambda^D_1\rfloor \leq 1 + \lambda^D_1$, 
which implies $\lambda^D_1 \leq 1$, and then 
$2 \lambda^D_1 = n \in \BbbN$ leads to $\lambda^D_1 \in \{1/2,1\}$, which is not possible due to $\omega \in (0,2\pi)$ and $\omega \ne \pi$. 
Since $\Sigma^D_{\lfloor k - 2/p \rfloor +2} = \emptyset$, the interpolation argument based on 
Lemma~\ref{lemma:abstract} can be done as in the proof of 
         Theorem~\ref{thm:shift-theorem-local-version}\ref{item:thm:shift-theorem-local-version-dirichlet} for the case 
$\lambda^D_1 \not\in \BbbN$ there. 

\emph{Proof under condition \ref{item:thm:Lp-i} and $\lambda^D_1 \in \BbbN$:} The value $k$ satisfying (\ref{eq:smallest-k}) is given by 
\begin{align*}
k = 
\begin{cases}
\lambda^D_1 & \mbox{ if $p \in (1,2)$} \\
\lambda^D_1 -1 & \mbox{ if $p \in (2,\infty)$}
\end{cases}
\qquad 
\mbox{ leading to } 
\qquad 
\lfloor k - 2/p\rfloor + 2 = 
\begin{cases}
\lambda^D_1 & \mbox{ if $p \in (1,2)$} \\
\lambda^D_1 & \mbox{ if $p \in (2,\infty)$}. 
\end{cases}
\end{align*}
Hence, $\Sigma^D_{\lfloor{k-2/p}\rfloor+2} = \{\lambda^D_1\}$ and $n^\star = \lambda^D_1$ in the estimate for $P_{\lfloor{k-2/p}\rfloor}$ 
in Corollary~\ref{cor:Lp}. That is, Corollary~\ref{cor:Lp} ascertains 
$\|P_{\lfloor {k - 2/p} \rfloor} \|_{B^{\lambda^D_1 + 2/p}_{p,\infty}}
\lesssim \|f\|_{B^{\lambda^D_1 + 2/p - 2}_{p,1}}$ as used in the proof of 
         Theorem~\ref{thm:shift-theorem-local-version}\ref{item:thm:shift-theorem-local-version-dirichlet} for the case 
$\lambda^D_1 \in \BbbN$ there. 

\emph{Proof under condition \ref{item:thm:Lp-ii}:} 
Multiple singularity functions in the representation of Corollary~\ref{cor:Lp} are allowed. However, the 
condition $k < 2/p$ ensures that $P_{\lfloor k - 2/p\rfloor} = 0$ so that as in the proof of 
Theorem~\ref{thm:shift-theorem-local-version}\ref{item:thm:shift-theorem-local-version-mixed}
for $\omega \in (3\pi/2,2\pi)$ one can use Lemma~\ref{lemma:abstract} since the linear functionals
$f \mapsto \int_{\mathcal C} r^{-\lambda^D_j} \sin (\lambda^D_j \phi) f\, d{\mathbf x}$ are 
bounded, linear on $B^{\lambda^D_j-2/p'}_{2,1}$ for $j \in \BbbN$ with  $\lambda^D_j + 2/p < k+2$. 
\end{proof}
\begin{remark} 
Theorem~\ref{thm:Lp} is based on the assumption that the solution operator for the Dirichlet problem maps $L^p \rightarrow W^{2,p}$, 
which is expressed by the condition $2 < \lambda^D_1 + 2/p$. This restriction is due to our working with positive order Besov spaces 
$B^s_{p,q}$, $s > 0$. To remove this restriction, one would have to use negative order Besov spaces similarly to the way it is done
in the $L^2$-setting. 
\eremk
\end{remark}
\appendix

\section{Weighted spaces and elliptic regularity in weighted spaces}
\label{sec:inequalities}

We introduce for $0<\rho<\sigma$ the annuli
\begin{align}
\label{def:annuli}
A(\rho,\sigma):=\{x\in\mathcal{C}:\rho<|x|<\sigma\}.
\end{align}

\begin{lemma}
\label{lemma:Sobolev-to-cone-second-version-appendix}
Let $f\in H^{k+\epsilon}(\mathcal{C})$ with $\operatorname{supp} f \subset B_1(0)$ for some $k\in\mathbb{N}_0$ and $\epsilon\in (0,1)$ and assume 
$\partial_x^i \partial_y^j f(0)=0$ for $i+j\leq k-1$. Then $f\in K_{-\epsilon}^k(\mathcal{C})$ with the norm estimate
$
\|f\|_{K_{-\epsilon}^k(\mathcal{C})} \lesssim \|f\|_{H^{k+\epsilon}(\mathcal{C}_1)}.
$
\end{lemma}
\begin{numberedproof}{of Lemma~\ref{lemma:Sobolev-to-cone-second-version-appendix}/Lemma~\ref{lemma:Sobolev-to-cone-second-version}}
We will show the result for the cases $\epsilon \in (0,1/2)$ and $\epsilon \in (1/2,1)$ separately. The limiting case $\epsilon = 1/2$ is 
then obtained by an interpolation argument (cf.\ \cite[Chap.~{23}]{Tartar_2007} for the interpolation of $L^2$-based spaces with weights). 
We also flag that we may assume $\omega \ne \pi$ as in the case $\omega = \pi$ the cone $\mathcal{C}$ can be split into 2 cone with apertures $\ne \pi$ 
and each cone is considered separately. 

\emph{Step 1:} \emph{Claim:} For $\omega \ne \pi$ there holds 
\begin{align}
\label{eq:norm-equivalence-A12}
\|f\|_{H^{k+\epsilon}(A(1,2))} &\leq C |\nabla^k f|_{H^\epsilon(A(1,2))} + 
\begin{cases}
\sum_{\ell=1}^2 \sum_{j=0}^{k-1} \|\nabla^j f\|_{L^2(\Gamma^1_\ell)}, & k \ge 1 \\
\sum_{\ell=1}^2 \|f\|_{L^2(\Gamma^1_\ell)}, & k \ge 0 \quad \mbox{ and } \epsilon > 1/2, 
\end{cases}
\end{align}
where $\Gamma^1_\ell$, $\ell \in \{1,2\}$, are the two straight parts of $\partial A(1,2)$. 
This estimate is a variant of a classical Poincar\'e inequality. By the Deny-Lions lemma and the connectedness of $A(1,2)$, 
the full norm $\|f \|_{H^{k+\epsilon}}$ and the seminorm $|\nabla^f f|_{H^\epsilon}$ are equivalent up the polynomials of degree $k$. 
To see that, for $k \ge 1$, the map $f \mapsto \|f\|_*:= \sum_{\ell=1}^2 \sum_{j=0}^{k-1}\|\nabla^j f\|_{L^2(\Gamma^1_\ell)}$ 
is a norm on ${\mathcal P}_k$, the space of polynomials of degree $k$, 
we have to show that $\pi \in {\mathcal P}_k$ and $\|\pi\|_* = 0$ implies $\pi =  0$. 
Assuming, as we may, that  $\Gamma^1_1 \subset \BbbR \times \{0\}$ and writing $\pi(x,y) = \sum_{i+j \leq k} a_{ij} x^i y^j$, we 
see that $a_{ij} = 0$ for all $i$ and $j=0,\ldots,k-1$ so that $\pi(x,y) = a_{0k} y^k$. Since the line $\Gamma^1_2$ is not 
colinear with $\Gamma^1_1$ due to $\omega \ne \pi$, we finally conclude $\pi = 0$. The case $k  = 0$ is easy. 

\emph{Step 2: $k \ge 1$ and $\epsilon \ne 1/2$:}
For $d>0$ consider $A(d,2d)$, denote $\Gamma^d_\ell$, $\ell \in \{1,2\}$, the two straight parts of $\partial A(d,2d)$,  
and write $\widehat{f}$ for the function $f$ scaled to the reference element $A(1,2)$. 
Scaling to $A(1,2)$, using the norm equivalence (\ref{eq:norm-equivalence-A12}), and scaling back yields 
\begin{align}
\label{eq:lemma:Sobolev-to-cone-second-version-1}
&\int_{A(d,2d)} r^{-2k-2\epsilon} |f|^2 \lesssim d^{-2k-2\epsilon+2} \|\widehat{f}\|_{L^2(A(1,2))}^2 
\lesssim d^{-2k-2\epsilon+2} \|\widehat{f}\|_{H^{k+\epsilon}(A(1,2))} 
\\ &
\nonumber 
\qquad \lesssim d^{-2k-2\epsilon+2} \left(|D^k \widehat{f}|_{H^\epsilon(A(1,2))}^2 + \sum_{j=0}^{k-1} \sum_{l=1}^2 \|\nabla^j \widehat{f}\|_{L^2(\Gamma_l^1)}^2\right) \\
&\qquad \lesssim |D^k f|_{H^\epsilon(A(d,2d))}^2 + \sum_{j=0}^{k-1} \sum_{l=1}^2 \|r^{-(k+\epsilon-(1/2+j))} \nabla^j f\|_{L^2(\Gamma_l^d)}^2,
\nonumber 
\end{align}
where we used $d \sim r$ on $A(d,2d)$.  Covering $\mathcal{C}$ by annuli of the form $A(d,2d)$, we obtain
\begin{align*}
\int_{\mathcal{C}} r^{-2k-2\epsilon} |f|^2 \lesssim |D^k f|_{H^\epsilon(\mathcal{C}_1)}^2 + \sum_{j=0}^{k-1} \sum_{l=1}^2 \|r^{-(k+\epsilon-(1/2+j))} \nabla^j f\|_{L^2(\Gamma_\ell^{\mathcal{C}_1})}^2,
\end{align*}
where $\Gamma_\ell^{\mathcal{C}_1}$, $l=1$, $2$, denote the straight-lined parts of $\partial\mathcal{C}_1$. 
As $f\in H^{k+\epsilon}(\mathcal{C}_1)$, the trace theorem gives 
$\nabla^j f|_{\Gamma_\ell^{\mathcal{C}_1}}\in H^{k-j+\epsilon-1/2}(\Gamma_\ell^{\mathcal{C}_1})$ for $j \in \{0,\ldots,k-1\}$. 
Since $\partial_x^i \partial_y^{j'} f(0)=0$ for $i+j'\leq k-1$, we even have 
$\chi (\nabla^j f)|_{\Gamma_\ell^{\mathcal{C}_1}}\in H_0^{k-j+\epsilon-1/2}(\Gamma_l^{\mathcal{C}_1})$, cf.\ \cite[Thm.~{3.40}]{Mclean00},  
for smooth cut-off functions $\chi$ with $\chi\equiv 1$ near the origin. (The cut-off function is merely introduced for notational convenience 
to localize near the origin.) It follows by \cite[Thm.~{1.4.4.4}]{Grisvard_2011} 
\begin{align}
\label{eq:lemma:Sobolev-to-cone-second-version-2}
\sum_{j=0}^{k-1} \|r^{-(k+\epsilon-(1/2+j))} \nabla^j f\|_{L^2(\Gamma_\ell^{\mathcal{C}_1})}^2 \lesssim \|f\|_{H^{k+\epsilon-1/2}(\Gamma_\ell^{\mathcal{C}_1})}^2 \lesssim \|f\|_{H^{k+\epsilon}(\mathcal{C}_1)}^2. 
\end{align}
The higher derivatives of $f$ appearing in the norm $\|\cdot\|_{K^k_{-\epsilon}(\mathcal{C})}$ are treated by a similar argument. 

\emph{Step 3: $k =  0$ and $\epsilon \ne 1/2$:} 
The result in the case $k=0$ and $\epsilon \in (0,1/2)$ is found in \cite[Thm.~{1.4.4.3}]{Grisvard_2011}. For $\epsilon \in (1/2,1)$, we proceed 
as in Step~2, replacing the sum $\sum_{j=0}^{k-1}$ with the sum $\sum_{\ell=1}^2 \|f\|_{L^2(\Gamma^1_\ell)}$. The key estimate
\cite[Thm.~{1.4.4.4}]{Grisvard_2011} is again applicable,  which yields the result. 
\end{numberedproof}

\begin{lemma}
\label{lemma:regularity-lemma-annuli}
Let $\rho_1 < \rho_2 < \rho_3 < \rho_4 $ and $\widehat{A}_2:= A(\rho_1,\rho_4)$, $\widehat{A}_2:= A(\rho_2, \rho_3)$. 
Let $k \in \BbbN_0$, $\epsilon \in (0,1)$, $f \in H^{k+\epsilon}(\widehat{A}_2)$ 
and $u \in H^1(\widehat{A}_2)$ satisfy $-\Delta u = f$ on $\widehat{A}_2$ with additionally either homogeneous Dirichlet conditions, 
homogeneous Neumann conditions, or homogeneous mixed boundary conditions on the two straight parts of $\partial \widehat{A}_2$. Then, there is 
$C > 0$, depending only on $\omega$ and $\rho_i$, $i=1,\ldots,4$, such that 
$|\nabla^{k+2} u|_{H^\epsilon(\widehat{A}_1)} \leq C ( \|f\|_{H^{k+\epsilon}(\widehat{A}_2)} + \|u\|_{H^{1}(\widehat{A}_2)}). $
More generally, for any $q \in [1,\infty]$ and $s > 1/2$, 
one has $\|u\|_{B^{s+2}_{2,q}(\widehat{A}_1)} \leq C ( \|f\|_{B^{-1+s}_{2,q}(\widehat{A}_2)} + \|u\|_{H^1(\widehat{A}_2)})$. 
For Dirichlet boundary conditions, this estimate actually holds for $s > 0$. 

\end{lemma}
\begin{proof}
We restrict to $s >  1/2$ in the second part of the lemma in order to achieve a unified notation since the spaces
$B^{-1+s}_{2,q}$ and $\widetilde{B}^{-1+s}_{2,q}$ differ for $s \leq  1/2$. 

This is a rather standard elliptic regularity theorem. We sketch the proof to illuminate the point that also regularity assertions in Besov spaces 
are possible. We restrict to Dirichlet boundary conditions at one straight edge $\Gamma_1$ of $\widehat{A}_2$ and merely consider a local situation of three half-balls 
$H_\rho:= B_{\rho}(x_0) \cap \widehat{A}_2$, 
$H_{\rho'}:= B_{\rho'}(x_0) \cap \widehat{A}_2$, 
$H_{\rho^{\prime\prime}}:= B_{\rho^{\prime\prime}}(x_0) \cap \widehat{A}_2$ 
with $x_0 \in \Gamma_1$ 
and $0 < \rho^{\prime\prime} < \rho' < \rho$.  
Consider $\chi_\rho \in C^\infty_0(B_\rho(x_0))$, $\chi_{\rho'} \in C^\infty_0(B_{\rho'}(x_0))$
with $\chi_\rho \equiv 1$ on $H_{\rho'}$ and $\chi_{\rho'} \equiv 1$ on $H_{\rho^{\prime\prime}}$. 
Let $T$ be the solution operator for the Poisson problem on $H_{\rho}$ with Dirichlet boundary conditions on 
$\partial H_{\rho}$. Then, by elliptic regularity, the map $f \mapsto \chi_{\rho'} T \chi_{\rho}$ is bounded $H^{s}(H_{\rho}) \rightarrow H^{s+2}(H_{\rho'})$ 
for any $s \in \BbbN_0 \cup \{-1\}$.  By interpolation, the map $f \mapsto \chi_{\rho'} T \chi_{\rho}$ is bounded 
$B^{s}_{2,q}(H_{\rho}) \rightarrow B^{s+2}_{2,q}(B_{\rho'})$ for any $s > -1$, $q \in [1,\infty]$. 
The difference 
$u_0:= u - \chi_{\rho'} T \chi_{\rho} f$ satisfies homogeneous Dirichlet conditions on $\Gamma_1$ and is harmonic on $H_{\rho'}$. Hence, 
$u_0$ is smooth (up to $\Gamma_1$) on $H_{\rho'}$ and interior regularity provides $\|u_0\|_{H^{s'+2}(H_{\rho^{\prime\prime}})} \lesssim \|u_0\|_{H^1(H_{\rho'})}$
for any $s' \ge 0$. Finally, $\|u_0\|_{H^1(H_{\rho'})} \leq \|u\|_{H^1(H_{\rho'})} + \|\chi_{\rho'} T \chi_{\rho} f\|_{H^1(H_{\rho'})} 
\lesssim \|u\|_{H^1(H_{\rho'})} + \|\chi f\|_{H^{-1}(\widehat{A}_2)} 
\lesssim \|u\|_{H^1(H_\rho)} + \|f\|_{H^{s}_{2,q}(\widehat{A}_2)}$, where $\chi$ 
is yet another smooth cut-off function with $\chi \equiv 1$ on $B_{\rho}(x_0)$ and supported by a sufficiently small neighborhood of $B_\rho(x_0)$. 

The local estimates can be combined into a global one by a smooth partition of unity to result in the desired estimates for the sets $\widehat{A}_1$ 
and $\widehat{A}_2$. 
\end{proof}

The following lemma demonstrates a type of scaling arguments employed in connection with weighted spaces.

\begin{lemma}
\label{lemma:scaling-arguments-cone}
Let $k\in\mathbb{N}_0$, $\gamma\in\mathbb{R}$, and $\epsilon\in (0,1)$. Further let $f\in L_{loc}^2(\mathcal{C})$ and $u \in H^1_{loc}(\mathcal{C})$ 
satisfy $-\Delta u = f$ with either homogeneous Dirichlet boundary conditions, homogeneous Neumann boundary conditions, or homogeneous mixed boundary 
conditions. Then: 
\begin{enumerate}[label=(\roman*)]
\item \label{item:lemma:scaling-arguments-cone-i} For $f\in K_{\gamma}^k(\mathcal{C})$, $u\in K_{\gamma-k-2}^0(\mathcal{C})$ and $\beta\in(\mathbb{N}_0)^2$, $|\beta|\in [0,k+2]$, it holds
\begin{align}
\label{eq:lemma:scaling-arguments-cone-1}
\|r^{-|\beta|+\gamma} D^{k+2-|\beta|} u\|_{L^2(\mathcal{C})} \lesssim \sum_{|\alpha|\leq k} \|r^{|\alpha|-k+\gamma} D^\alpha f\|_{L^2(\mathcal{C})} + \|r^{-k-2+\gamma} u\|_{L^2(\mathcal{C})}.
\end{align}
\item \label{item:lemma:scaling-arguments-cone-ii} Let $f\in H^{k+\epsilon}(\mathcal{C}_1)$ with $\operatorname{supp} f \subset B_1(0)$. Further assume that there exists $R>1$ such that 
$u \in K^{k+2}_{-\epsilon}(\mathcal{C}_R)$. 
Then 
\begin{align*}
|D^{k+2} u|_{H^\epsilon(\mathcal{C}_1)} \lesssim |D^k f|_{H^\epsilon(\mathcal{C}_1)} + \sum_{|\alpha| \leq k} \|r^{|\alpha|-k-\epsilon} D^\alpha f\|_{L^2(\mathcal{C}_1)} \!+\!\! \sum_{|\alpha| \leq k+2} \|r^{|\alpha|-k-2-\epsilon} D^\alpha u\|_{L^2(\mathcal{C}_R)}.
\end{align*}
\end{enumerate}
\end{lemma}

\begin{proof}
Define the annuli $\widehat{A_1}:=A(1/2,2)$ and $\widehat{A_2}:=A(1/4,4)$, cf.\ \eqref{def:annuli}. For $\rho>0$ scaling yields for $A_{i,\rho}:=\rho\widehat{A_i}$, $i=1,2$, $\widehat{u}(\xi):=u(\rho\xi)$ and $\widehat{f}(\xi):=f(\rho\xi)$ that $-\Delta\widehat{u}=\rho^2 \widehat{f}$ on $\widehat{A_2}$.

We start the proof of \ref{item:lemma:scaling-arguments-cone-i} by noting the elliptic regularity estimate
\begin{align}
\label{eq:Kgammas-facts-vi}
\|\widehat{u}\|_{H^{k+2}(\widehat{A_1})} \lesssim \rho^2 \|\widehat{f}\|_{H^k(\widehat{A_2})} + \|\widehat{u}\|_{L^2(\widehat{A_2})},
\end{align}
cf. \cite[eq.~(4.3)]{Dauge_1997}. We multiply \eqref{eq:Kgammas-facts-vi} by $\rho^{\gamma-k-2}$ and obtain after scaling
\begin{align*}
\sum_{|\alpha|\leq k+2} \rho^{2|\alpha|+2\gamma-2k-4} \|D^\alpha u\|_{L^2(A_{1,\rho})}^2 \lesssim \sum_{|\alpha|\leq k} \rho^{2|\alpha|+2\gamma-2k} \|D^\alpha f\|_{L^2(A_{2,\rho})}^2 + \rho^{2\gamma-2k-4} \|u\|_{L^2(A_{2,\rho})}^2.
\end{align*}
The definition of the annuli implies $2^{-i}\rho<r<2^i\rho$, $i=1$, $2$, on $A_{i,\rho}$. Thus we get further
\begin{align*}
\sum_{|\alpha|\leq k+2} \|r^{|\alpha|+\gamma-k-2} D^\alpha u\|_{L^2(A_{1,\rho})}^2 \lesssim \sum_{|\alpha|\leq k} \|r^{|\alpha|+\gamma-k} D^\alpha f\|_{L^2(A_{2,\rho})}^2 + \|r^{\gamma-k-2} u\|_{L^2(A_{2,\rho})}^2.
\end{align*}
We now cover $\mathcal{C}$ by annuli $A_{1,2^{-j}}$, $j\in\mathbb{Z}$. Since they only have finite overlap, we obtain
\begin{align*}
\|r^{-|\beta|+\gamma} D^{k+2-|\beta|} u\|_{L^2(\mathcal{C})}^2 
&\lesssim \sum_{|\alpha|\leq k} \|r^{|\alpha|+\gamma-k} D^\alpha f\|_{L^2(\mathcal{C})}^2 + \|r^{\gamma-k-2} u\|_{L^2(\mathcal{C})}^2
\end{align*}
for $|\beta|\in [0,k+2]$, whereupon the result follows.

The proof of \ref{item:lemma:scaling-arguments-cone-ii} follows in a similar way. We prove it assuming $R > 2$. We have 
\begin{align}
\label{eq:lemma:scaling-arguments-cone-5}
|D^{k+2} \widehat{u}|_{H^\epsilon(\widehat{A_1})} \lesssim \rho^2 \|\widehat{f}\|_{H^{k+\epsilon}(\widehat{A_2})} + \|\widehat{u}\|_{H^{k+2}(\widehat{A_2})},
\end{align}
which follows from Lemma~\ref{lemma:regularity-lemma-annuli} after scaling with $\rho$. Then the usual scaling arguments yield
\begin{align*}
&|D^{k+2} u|_{H^\epsilon(A_{1,\rho})}^2 \lesssim \rho^{2-2(k+2+\epsilon)} |D^{k+2} \widehat{u}|_{H^\epsilon(\widehat{A_1})}^2 \\
&\qquad \lesssim \rho^{2-2(k+2+\epsilon)} \Bigl(\rho^4 \sum_{|\alpha| \leq k} \|D^\alpha \widehat{f}\|_{L^2(\widehat{A_2})}^2 + \rho^4 |D^k \widehat{f}|_{H^\epsilon(\widehat{A_2})}^2 + \sum_{|\alpha| \leq k+2} \|D^\alpha \widehat{u}\|_{L^2(\widehat{A_2})}^2\Bigr) \\
&\qquad \lesssim \rho^{2-2(k+2+\epsilon)} \Bigl(\rho^4 \sum_{|\alpha| \leq k} \rho^{-2+2|\alpha|} \|D^\alpha f\|_{L^2(A_{2,\rho})}^2 + \rho^4 \rho^{-2+2(k+\epsilon)} |D^k f|_{H^\epsilon(A_{2,\rho})}^2  
\\ & \qquad \mbox{}
+ \sum_{|\alpha| \leq k+2} \rho^{-2+2|\alpha|} \|D^\alpha u\|_{L^2(A_{2,\rho})}^2\bigr) \\
&\qquad \lesssim  \sum_{|\alpha| \leq k} \rho^{-2k-2\epsilon+2|\alpha|} \|D^\alpha f\|_{L^2(A_{2,\rho})}^2 \!+\! |D^k f|_{H^\epsilon(A_{2,\rho})}^2 \!+\!\!\! \sum_{|\alpha| \leq k+2} \rho^{-2k-2\epsilon-4+2|\alpha|} \|D^\alpha u\|_{L^2(A_{2,\rho})}^2.
\end{align*}
As in \ref{item:lemma:scaling-arguments-cone-i} we obtain by covering arguments
\begin{align*}
|D^{k+2} u|_{H^\epsilon(\mathcal{C}_1)}^2 &\lesssim \sum_{|\alpha| \leq k} \|r^{-k-\epsilon+|\alpha|} D^\alpha f\|_{L^2(\mathcal{C}_1)}^2 \!+\! |D^k f|_{H^\epsilon(\mathcal{C}_1)}^2 
 + \sum_{|\alpha| \leq k+2} \|r^{-k-\epsilon-2+|\alpha|} D^\alpha u\|_{L^2(\mathcal{C}_R)}^2.
\qedhere
\end{align*}
\end{proof}


\ifarxiv
\section{Some proofs}
\begin{numberedproof}{of Lemma~\ref{lemma:embedding-Wkp}}
The proof is structurally the same as that of Lemma~\ref{lemma:Sobolev-to-cone-second-version-appendix}/Lemma~\ref{lemma:Sobolev-to-cone-second-version}. 
We use an analogous notation as there. The case $k = 0$ is trivial, so we concentrate on $k \ge 1$. 

\emph{Step 1:} We claim for $k \ge 1$ the estimate 
\begin{align}
\label{eq:norm-equivalence-A12-Lp}
\|f\|_{W^{k,p}(A(1,2))} &\leq C |\nabla^k f|_{L^p(A(1,2))} +
\sum_{\ell=1}^2 \sum_{0 \leq j \leq k-1 }\|\nabla^j f\|_{L^p(\Gamma^1_\ell)}
\end{align}
It suffices to ascertain that 
$f \mapsto \|f\|_\ast:= \sum_{\ell=1}^2 \sum_{j \leq k-1} \|\nabla^j f\|_{L^p(\Gamma^1_\ell)}$ is a norm
on ${\mathcal P}_{k-1}$. This follows in the same way as for the case $p=2$ in the proof  
Lemma~\ref{lemma:Sobolev-to-cone-second-version}.

\emph{Step 2: $k \ge 1$:}
Scaling to $A(1,2)$, using the norm equivalence (\ref{eq:norm-equivalence-A12-Lp}), and scaling back yields
\begin{align}
\label{eq:lemma:Sobolev-to-cone-second-version-1-Lp}
&\int_{A(d,2d)} r^{-pk} |f|^p \lesssim d^{-pk+2} \|\widehat{f}\|_{L^p(A(1,2))}^p
\lesssim d^{-pk+2} \|\widehat{f}\|^p_{W^{k,p}(A(1,2))}
\\ &
\nonumber
\qquad \lesssim d^{-pk+2} \left(|D^k \widehat{f}|_{L^p(A(1,2))}^2 + \sum_{j=0}^{k-1} \sum_{l=1}^2 \|\nabla^j \widehat{f}\|_{L^p(\Gamma_l^1)}^p\right) \\
&\qquad \lesssim |D^k f|_{L^p(A(d,2d))}^2 + \sum_{j=0}^{k-1} \sum_{l=1}^2 \|r^{-(k-(1/p+j))} \nabla^j f\|_{L^p(\Gamma_l^d)}^2,
\nonumber
\end{align}
where we used $d \sim r$ on $A(d,2d)$.  Covering $\mathcal{C}$ by annuli of the form $A(d,2d)$, we obtain
\begin{align*}
\int_{\mathcal{C}} r^{-pk} |f|^p \lesssim |D^k f|_{L^p(\mathcal{C}_1)}^p 
+ \sum_{j=0}^{k-1} \sum_{l=1}^2 \|r^{-(k-(1/p+j))} \nabla^j f\|_{L^p(\Gamma_\ell^{\mathcal{C}_1})}^p.
\end{align*}
As $f\in W^{k,p}(\mathcal{C}_1)$, the trace theorem gives
$\nabla^j f|_{\Gamma_\ell^{\mathcal{C}_1}}\in W^{k-j-1/p}(\Gamma_\ell^{\mathcal{C}_1})$ for $j \in \{0,\ldots,k-1\}$.
We now distinguish the cases $p > 2$ and $p < 2$.

\emph{The case $p > 2$:} Then, $\lfloor k-2/p\rfloor = k-1$. 
Since $\partial_x^i \partial_y^{j'} f(0)=0$ for $i+j'\leq \lfloor k-2/p\rfloor = k-1$, we even have
$\nabla^j f|_{\Gamma_\ell^{\mathcal{C}_1}}\in W_0^{k-j-1/p,p}(\Gamma_l^{\mathcal{C}_1})$. 
It follows by \cite[Thm.~{1.4.4.4}]{Grisvard_2011}
\begin{align}
\label{eq:lemma:Sobolev-to-cone-second-version-2-Lp}
\sum_{j=0}^{k-1} \|r^{-(k-(1/p+j))} \nabla^j f\|_{L^p(\Gamma_\ell^{\mathcal{C}_1})}^p \lesssim \|f\|_{W^{k-1/p}(\Gamma_\ell^{\mathcal{C}_1})}^p 
\lesssim \|f\|_{W^{k,p}(\mathcal{C}_1)}^p.
\end{align}
\emph{The case $p < 2$:} Then, $\lfloor k-2/p\rfloor = k-2$. We have $\nabla^{k-1} f \in W^{1-1/p,p}(\Gamma_l^{\mathcal{C}_1})$. 
In view ovew of $p < 2$, which implies $1-1/p < 1/p$, we have have   
$\nabla^{k-1} f \in W^{1-1/p,p}(\Gamma_l^{\mathcal{C}_1}) = 
\widetilde{W}^{1-1/p,p}(\Gamma_l^{\mathcal{C}_1})$, \cite[Cor.~{1.4.4.5}]{Grisvard_2011}.
Since $\partial_x^i \partial_y^{j'} f(0)=0$ for $i+j'\leq \lfloor k-2/p\rfloor = k-2$, taking 
$k-1-j$ (tangential) antiderivatives shows for $j=0,\ldots,k-1$ that 
$\chi \nabla^j f \in \widetilde{W}^{(k-1-j)+1-1/p,p}(\Gamma_l^{\mathcal{C}_1})$, where $\chi$ is smooth cut-off functions
with $\chi \equiv 1$ near the origin, which we merely introduce for notational convenience to localize near the origin. 
Hence, in view of \cite[Thm.~{1.4.4.4}]{Grisvard_2011} we arrive at 
\begin{align}
\label{eq:lemma:Sobolev-to-cone-second-version-2-Lp}
\sum_{j=0}^{k-1} \|r^{-(k-(1/p+j))} \nabla^j f\|_{L^p(\Gamma_\ell^{\mathcal{C}_1})}^p \lesssim \|f\|_{W^{k-1/p}(\Gamma_\ell^{\mathcal{C}_1})}^p 
\lesssim \|f\|_{W^{k,p}(\mathcal{C}_1)}^p.
\end{align}

\end{numberedproof}

\section{Solving elliptic equations with the Mellin calculus: the $L^2$-setting}
\label{appendix:L2}
We present classical material about solving elliptic equation using the 
Mellin calculus. In this section, we focus on the $L^2$-setting, which meshes
well with the Mellin transformation due Plancherel's equality. We follow 
the presentation \cite{Dauge_1997}. 
\subsection{Fourier transformation and Paley-Wiener theorem} 
Define 
\begin{equation}
{\mathcal F}f(\xi):= \widehat f(\xi) := \frac{1}{\sqrt{2\pi}} \int_{x} e^{-i x \xi} f(x)\, dx
\end{equation}
with inverse formula
\begin{equation}
f(x)= \frac{1}{\sqrt{2\pi}} \int_{\xi} e^{i  x \xi} \widehat f(\xi)\, d\xi. 
\end{equation}
\emph{Facts:}
\begin{enumerate}[label=(\roman*)]
\item
\label{fourier-facts-i}
for $f \in L^2$, the Fourier transformation is defined as 
$$
\widehat f(\xi) = \operatorname{l.i.m.}_{R \rightarrow \infty} \frac{1}{\sqrt{2\pi}} \int_{-R}^R e^{-i x \xi} f(x)\, dx, 
$$
where $\operatorname{l.i.m.}$ stands for ``limit in the mean'', i.e., in $L^2$; this definition is consistent with
defining the Fourier transformation first on the Schwartz space ${\mathcal S}$, 
then by duality on the space of tempered distributions ${\mathcal S}^\prime$---see, e.g., \cite[Chap.~6, Sec.~2]{yosida80a}. 
\item
\label{fourier-facts-ii}
${\mathcal F}:L^2(\BbbR) \rightarrow L^2(\BbbR)$ isometric isomorphism by Plancherel's theorem. 
\item
\label{fourier-facts-iii}
If $f \in L^2$ with $\operatorname{supp} f \subset [-B,B]$, then $\widehat f$ is an entire function with $|\widehat f(z)| \leq C e^{B |\operatorname{Im} z|}$
\item 
\label{fourier-facts-iv}
If $f \in L^2$ with $\operatorname{supp} f \subset [0,\infty)$, then $\widehat f$ is holomorphic in $\{z \in \BbbC\,|\, \operatorname{Im} z < 0\}$. 
\end{enumerate}
A generalization of \ref{fourier-facts-iv} is 
\begin{lemma}
\label{lemma:7.1} 
Let $a < b$. Then: 
\begin{enumerate}[label=(\roman*)]
\item
\label{item:lemma:7.1-i}
$e^{ax} f \in L^2$ and $e^{bx} f \in L^2$ implies 
\begin{enumerate}[label=(\alph*)]
\item 
\label{item:lemma:7.1-ia}
$\widehat f$ holomorphic in the strip  $\{z \in \BbbC\,|\, a < \operatorname{Im} z < b\}$
\item
\label{item:lemma:7.1-ib}
$\|e^{a x} f \|^2_{L^2}  = \|\widehat f(\cdot + ia)\|^2_{L^2} = \lim_{\eta \rightarrow a+} \|\widehat f(\cdot + i \eta)\|^2_{L^2}$
\item
\label{item:lemma:7.1-ic}
$\|e^{b x} f \|^2_{L^2}  = \|\widehat f(\cdot + ib)\|^2_{L^2} = \lim_{\eta \rightarrow b-} \|\widehat f(\cdot + i \eta)\|^2_{L^2}$
\item
\label{item:lemma:7.1-id}
$\widehat f(\cdot + i \eta) \rightarrow \widehat f(\cdot + ia)$ in $L^2$ as $\eta \rightarrow a$ and 
$\widehat f(\cdot + i \eta) \rightarrow \widehat f(\cdot + ib)$ in $L^2$ as $\eta \rightarrow b$
\item
\label{item:lemma:7.1-ie}
for $\theta \in [0,1]$ there holds 
$$
\|\widehat f(\cdot + i (\theta a + (1-\theta) b)\|_{L^2} \leq \|\widehat f(\cdot + ia)\|^\theta_{L^2} \|\widehat f(\cdot + i b)\|^{1-\theta}_{L^2}
$$
\end{enumerate}
\item 
\label{item:lemma:7.1-ii}
If a function $\widehat f$ is holomorphic in $\{z \in \BbbC\,|\, a < \operatorname{Im} z < b\}$ and $\sup_{\eta \in (a,b)} \|\widehat f(\cdot + i \eta)\|_{L^2} < \infty$, 
then the inversion formula 
$$
f(x) = \frac{1}{\sqrt{2\pi}} \int_{\operatorname{Im} \zeta = \eta}\widehat f(\zeta) e^{i \zeta x} \, d\zeta
$$
holds for any $a < \eta < b$. 
(The integral is to be understood as $\operatorname*{l.i.m.}_{R \rightarrow \infty} \int_{-R}^R$.)
Moreover, $e^{a x} f \in L^2$ and $e^{bx} f \in L^2$. 
\end{enumerate}
\end{lemma}
\begin{proof}
\emph{Proof of \ref{item:lemma:7.1-ia}:}  For compact subsets $K \subset S:= \{z \in \BbbC\,|\, a < \operatorname{Im} z < b\}$ the integral 
of the Fourier transform is absolutely uniformly (in $z$) bounded: 
\begin{align*}
\int_{|x| \leq R} |e^{-i x z} f(x)| &\leq 
\int_{x \in \BbbR} |e^{-i x z} f(x)| \leq 
\int_{x} |f(x)| e^{x \operatorname{Im} z} = \int_{x < 0} |f(x)| e^{x \operatorname{Im} z} + \int_{x > 0} |f(x)| e^{x \operatorname{Im} z}
\\
&\leq \int_{x < 0} |f(x)| e^{ax}  e^{x (\operatorname{Im} z - a)}+ 
     \int_{x > 0} |f(x)| e^{bx} e^{x (\operatorname{Im} z - b)}  \\
&\stackrel{\operatorname{Im} z - a > 0, \operatorname{Im} z - b < 0}{\lesssim} 
\|e^{ax} f\|_{L^2} + \|e^{bx} f\|_{L^2}
\end{align*}
From this, it is easy to see that $z \mapsto \widehat f(z)$ is a continuous function and then that
path integrals $\int_{\Gamma} \widehat f(z)\,dz = 0$ for closed paths $\Gamma$ since $z \mapsto e^{i x z}$ is holomorphic. 
Hence, by Morera's theorem, $\widehat f$ is holomorphic. 

\emph{Proof of \ref{item:lemma:7.1-ib}:}  
\begin{itemize}
\item 
$\|e^{ax} f\|_{L^2} = \|\widehat f(\cdot + ia)\|_{L^2}$ is Plancherel
\item 
$\displaystyle 
\|\widehat f(\cdot + i \eta)\|_{L^2} \stackrel{\text{Plancherel}}{ = } \|e^{\eta x} f\|_{L^2} \stackrel{\eta \rightarrow a+}{\rightarrow }
\|e^{a x} f\|_{L^2} 
$
by Lebesgue dominiated convergence since $e^{\eta x} \leq \max\{e^{ax}, e^{bx}\}$. 
(Note: $e^{\eta x} \leq \max\{e^{ax}, e^{bx}\}$)
\end{itemize}
\emph{Proof of \ref{item:lemma:7.1-id}:}  
\begin{align*}
\|\widehat f(\cdot + i \eta) - \widehat f(\cdot + i a)\|_{L^2} 
\stackrel{\text{Plancherel}}{ = } \|(e^{\eta x} - e^{ax}) f\|_{L^2} \stackrel{\eta \rightarrow a}{\rightarrow } 0 \mbox{ by dominated convergence}
\end{align*}
\emph{Proof of \ref{item:lemma:7.1-ie}:}  
\begin{align*}
\|\widehat f(\cdot + i (\theta a + (1-\theta) b)\|^2_{L^2} &\stackrel{\text{Plancherel}}{ = } \|e^{(a \theta + (1-\theta) b)x} f\|^2_{L^2} 
 = \int_{x} e^{2 \theta a x}|f|^{2\theta} e^{2(1-\theta) b x} |f|^{2(1-\theta)}  \\
& \stackrel{\text{H\"older}}{\leq} \left( \int_{x} e^{2ax} |f|^2\right)^{\theta} \left(\int_x e^{2bx} |f|^2\right)^{1-\theta}
\end{align*}
\emph{Proof of \ref{item:lemma:7.1-ii}:}  
\emph{idea:} define $f$ by inversion formula for \emph{fixed} $\eta \in (a,b)$ and show independence of $\eta$ with Cauchy integral formula.

Fix $\eta \in (a,b)$ and \emph{define}
$$
f(x) = \frac{1}{\sqrt{2\pi}} \int_{\operatorname{Im} \zeta = \eta}  e^{i x \zeta} \widehat f(\zeta)\, d\zeta. 
$$
More precisely, using that $\widehat f(\cdot + i \eta) \in L^2$, we have 
$$
e^{\eta x} f(x):= \operatorname*{l.i.m.}_{R \rightarrow \infty} \frac{1}{\sqrt{2\pi}} \int_{\xi=-R}^R e^{i x \xi} \widehat f(\xi + i \eta)\, d\xi. 
$$
By holomorphy of $\widehat f$, we have by Cauchy's integral theorem by setting 
$\Gamma_{\eta} = \{z\,|\, \operatorname{Im} z = \eta, |\operatorname{Re} z| \leq \xi_n\}$, 
$\Gamma_{\eta'} = \{z\,|\, \operatorname{Im} z = \eta', |\operatorname{Re} z| \leq \xi_n\}$, 
$\Gamma_l(\xi_n) = \{z\,|\, \operatorname{Re} z = -\xi_n, \quad \operatorname{Im} z \in (\eta,\eta')\}$ with paths oriented appropriately: 
\begin{align*}
\int_{z \in \Gamma_\eta} e^{i x z} \widehat f(z) & = \int_{z \in \Gamma_{\eta'}} e^{i x z} \widehat f(z) + \int_{z \in \Gamma_l} e^{i x z} \widehat f(z) + 
\int_{z \in \Gamma_r} e^{i x z} \widehat f(z) 
\end{align*}
\emph{Claim:} there is a sequence $(\xi_n)_n$ with $\xi_n \rightarrow +\infty$ such that 
$\int_{z \in \Gamma_l(\xi_n)} |\widehat f(z)|^2 + \int_{z \in \Gamma_r(\xi_n)} |\widehat f(z)|^2 \rightarrow 0$. 
To see this, note the assumption $\sup_{\eta \leq y \leq \eta'} \int_{\xi \in \BbbR} |\widehat f (\xi + i y)|^2 < \infty$. Hence, 
\begin{align*}
\infty & > \int_{y \in (\eta,\eta')} \int_{\xi} |\widehat f(\xi + i y)|^2 \stackrel{\text{Fubini}}{  = }
\int_{\xi} \underbrace{ \int_{y=\eta}^{\eta'} |\widehat f(\xi+ iy)|^2\, dy}_{=:g(\xi) \in C(\BbbR)}\,d\xi
\end{align*}
Hence, there exists a sequence $(\xi_n)_n$ with $\xi_n \rightarrow \infty$ such that $g(\pm \xi_n) \rightarrow 0$. 
(Use $\int_{\xi=0}^\infty g(\xi) + g(-\xi)\,d\xi < \infty$, which implies that $g(\xi_n) + g(-\xi_n)\rightarrow 0$.)
We conclude  for fixed $x$ 
\begin{equation}
\label{eq:7.1}
\lim_{n \rightarrow \infty} \int_{\xi = -\xi_n}^{\xi_n} e^{i x (\xi + i \eta)} \widehat f(\xi + i \eta)\, d\xi 
= \lim_{n \rightarrow \infty} \int_{-\xi_n}^{\xi_n} e^{i x (\xi + i \eta')}\widehat f(\xi + i \eta')\, d\xi. 
\end{equation}
Since $\widehat f(\cdot + i \eta) \in L^2$ and $\widehat f(\cdot + i \eta') \in L^2$, we have 
\begin{align*}
&\lim_{n \rightarrow \infty} \int_{-\xi_n}^{\xi_n} e^{i x \xi} \widehat f(\xi + i \eta)\, d\xi \mbox{ exists in $L^2$ and equals $e^{\eta x} f(x)$ }\\
&\lim_{n \rightarrow \infty} \int_{-\xi_n}^{\xi_n} e^{i x \xi} \widehat f(\xi + i \eta')\, d\xi \mbox{ exists in $L^2$.}
\end{align*}
By passing, if necessary, to subsequences, the convergence is also pointwise a.e.. We conclude using (\ref{eq:7.1})
$$
\int_{\operatorname{Im} \zeta = \eta} e^{i x \zeta} f(\zeta) = \int_{\operatorname{Im} \zeta = \eta'} e^{i x \zeta} \widehat f(\zeta), 
$$
where the convergence of both integrals is understood as a Fourier transformation in $\xi \in \BbbR$, 
i.e., as $\lim_{R \rightarrow \infty} \int_{-R}^R$. 

We now show $e^{ax} \in L^2$ and $e^{bx} f \in L^2$: By Plancherel, for $\eta \in (a,b)$, we have 
$$
\|e^{\eta x} f\|_{L^2} = \|\widehat f(\cdot + i \eta)\|_{L^2} \leq C \quad \mbox {uniformly in $\eta$}. 
$$
For $x > 0$ and $\eta \rightarrow b-$ we have $e^{\eta x} |f| \uparrow e^{bx} |f|$ and $\|e^{\eta x} f\|_{L^2(0,\infty)} \leq C$. By the monotone convergence theorem, 
$e^{b x} f \in L^2(0,\infty)$. Analogously, $e^{ax} f \in L^2(-\infty,0)$. This shows $e^{ax} f$, $e^{bx} f \in L^2(\BbbR)$. 
\end{proof}
By letting $a \rightarrow -\infty$ and $b=0$ (or, similarly, $a = 0$ and $b \rightarrow \infty$) we obtain
\begin{corollary}
\label{cor:7.2}
\begin{enumerate}[label=(\roman*)]
\item 
\label{item:cor:7.2-i}
Let $f \in L^2$ with $\operatorname{supp} f \subset [0,\infty)$. Then $f$ is holomorphic on $\{z \in \BbbC\,|\, \operatorname{Im} z < 0\}$ and 
$\sup_{\eta < 0} \|\widehat f(\cdot + i \eta)\|_{L^2} \leq \|\widehat f\|_{L^2}$ and $\lim_{\eta \rightarrow 0} \|\widehat f(\cdot + i \eta) - \widehat f\|_{L^2} =0$. 
\item
\label{item:cor:7.2-ii}
Conversely, a function $\widehat f$ that is holomorphic on $\{\operatorname{Im} z < 0\}$ and satisfies $\sup_{\eta < 0} \|\widehat f(\cdot + i \eta)\|_{L^2} < \infty$
is the FT of a function $f \in L^2$ with $\operatorname{supp} f \subset [0,\infty)$. 
\end{enumerate}
\end{corollary}
\begin{proof}
\emph{Proof of \ref{item:cor:7.2-i}:} 
$\operatorname{supp} f \subset [0,\infty)$ implies that we have may take $a = -\infty$ and $b= 0$ in Lemma~\ref{lemma:7.1}. 

\emph{Proof of \ref{item:cor:7.2-ii}:} 
Lemma~\ref{lemma:7.1}, (\ref{item:lemma:7.1-ii}) shows $e^{ax} f \in L^2$ for every $a < 0$ and 
$\|e^{ax} f\|_{L^2}$ is bounded uniformly in $a$. Hence, $f(x) = 0$ for $x < 0$. 
\end{proof}
\subsection{Mellin transformation}
\subsubsection{Mellin transformation in 1D}
The Mellin transformation is the appropriate tool to analyze functions of the form $u(r) = r^\beta U(r)$, $r > 0$. It is defined
by a) substitution $r = e^t$: $u(r) \leftrightarrow \check u(t) := u(e^t)$ (``Euler transformation'')  and 
b) Fourier transformation of $\check u$: 
\begin{align}
\label{eq:mellin}
({\mathcal M}u)(\xi) &:= ({\mathcal F} \check u)(\xi)  = \frac{1}{\sqrt{2\pi}} \int_{t = -\infty}^\infty e^{-i \xi t} \check u(t)\, dt
 = \frac{1}{\sqrt{2\pi}} \int_{r=0}^\infty  e^{-i \xi \ln r} u(r) \frac{dr}{r} \\
\nonumber 
& = \frac{1}{\sqrt{2\pi}} \int_{r=0}^\infty r^{-i \xi} u(r) \frac{dr}{r}
\end{align}
with the (formal) inverse 
$$
u(r) = \frac{1}{\sqrt{2\pi}} \int_{\xi=-\infty}^\infty r^{i\xi} ({\mathcal M} u)(\xi) \,d\xi. 
$$
\begin{lemma} There holds the norm equivalence 
$$
\|r^\beta u\|_{L^2(\BbbR^+)} \sim \|{\mathcal M} u(\cdot + (\beta+1/2) i)\|_{L^2}
$$
\end{lemma}
\begin{proof}
\begin{itemize}
\item 
$e^{at} \check u \in L^2 \quad \Longleftrightarrow \quad {\mathcal M} u(\cdot + i a) \in L^2$ 
\item
$r^\beta  u \in L^2(\BbbR^+) \quad \Longleftrightarrow \quad \infty > \int_r r^{2\beta} u^2 = \int_t e^{2 \beta t} \check u^2(t) e^t\, dt$
\end{itemize}
\end{proof}
The analog of Corollary~\ref{cor:7.2} is the case of $u$ with $u(r) = 0$ for $r \ge 1$. Then, ${\mathcal M} u$ is holomorphic 
in the half-space $\{\operatorname{Im} z > 0\}$. However, the decay of $\check u$ at $t = -\infty$ or, of $u$ at $r = 0$, allows one to extend ${\mathcal M} u$
meromorphically into the lower half plane: 
\begin{lemma}
\label{lemma:8.2}
Let $u \in C^{k,\sigma}([0,\infty))$ with $u(r) = 0$ for $r \ge 1$. Then ${\mathcal M} u$ has a meromorphic extension to 
$\{ z \in \BbbC\,|\, \operatorname{Im} z > - (k + \sigma)\}$ with simple poles at $0, -i$, $-2i,\ldots,-ki$. 
\end{lemma}
\begin{proof}
Write 
\begin{align*}
u(r) & = \sum_{j=0}^k \frac{u^{(j)}(0)}{j!}  r^j \chi_{(0,1)}(r) + R(r),  \\
|R(r)| &\leq 
\begin{cases}
C r^{r + \sigma} & r \in (0,1) \\
0 & r \ge 1
\end{cases}
\end{align*}
Then
$$
{\mathcal M} u(z) = \sum_{j=0}^k \frac{u^{(j)}(0)}{j!}  {\mathcal M}(r^j \chi_{(0,1)}(r)) (z) + {\mathcal M} R(z)
$$
Since $|\check R(t)| \leq C e^{t (k+\sigma)}$ and $\operatorname{supp} \check R \subset (-\infty,0]$ we get that ${\mathcal M} R$
is holomorphic on $\{\operatorname{Im} z > - (k + \sigma)\}$. 

We compute (for $\operatorname{Im} z > 0$)
$$
{\mathcal M} (r^j \chi_{(0,1)}(r))(z) = \frac{1}{\sqrt{2\pi}} \int_{r=0}^1 r^{j-i z} \frac{dr}{r} = \frac{1}{\sqrt{2\pi}} \frac{1}{j-i z}
$$
so that 
$$
{\mathcal M} u(z) = \frac{1}{\sqrt{2\pi}} \sum_{j=0}^k \frac{u^{(j)}(0)}{j!}  \frac{1}{j - i z} + {\mathcal M} R(z)
$$
This formula is valid for $\operatorname{Im} z >0$. (This is where \textsl{a priori} the left-hand side is defined.)
The right-hand side is meromorphic function on $\{\operatorname{Im} z > -(k+\sigma)\}$. Hence, the left-hand side has a meromorphic
extension to $\{\operatorname{Im} z > -(k+\sigma)\}$. 
\end{proof}
Analogous to Lemma~\ref{lemma:7.1}, we have 
\begin{lemma}
If $\beta_1 < \beta_2$ and $r^{\beta_1} u \in L^2$, $r^{\beta_2} u \in L^2$, then ${\mathcal M} u$ is holomorphic in the strip
$\{\beta_1 + 1/2 < \operatorname{Im} z < \beta_2 + 1/2\}$. 
\end{lemma}
\subsection{Mellin transformation and $K^s_\gamma$}
\label{sec:mellin-transformation-and-Ks}

\underline{Recall:} $\mathcal{C}= \{x \in \BbbR^d\,|\, \frac{x}{|x|} \in G\}$ for a $G \subset \partial B_1(0)$ 

\underline{Notation:} We denote by $\widetilde{\cdot}$ a function in polar coordinates and by $\check{\cdot}$ the subsequent 
Euler transformation:
\begin{align*}
u(x) & = \widetilde u(r,\theta) \mbox{ (change of variables: Cartesian $\leftrightarrow$ polar coordiate)} \\
\check u(t,\theta) & = u(e^t,\theta) \mbox{ (Euler transformation $r = e^t$.)}
\end{align*}
Note: $r\partial_r \widetilde u = \partial_t \check u$. 

In order to understand more generally how derivatives transform under changes the variables $x \leftrightarrow (r,\theta) \leftrightarrow (t,\theta)$, 
we consider
\begin{example} 
for $d = 2$ with $x_1 = r \cos \theta$, $x_2 = r \sin\theta$ we have 
\begin{align*}
\partial_{x_1} u & = (\cos \theta \partial_r  - \frac{1}{r} \sin \theta \partial_\theta) \widetilde u
 = e^{-t} (\cos \theta \partial_t  - \sin \theta \partial_\theta) \check u \\
\partial_{x_2} u & = (\sin  \theta \partial_r  + \frac{1}{r} \cos \theta \partial_\theta) \widetilde u
 = e^{-t} (\sin \theta \partial_t  + \cos \theta \partial_\theta) \check u.  
\end{align*}
{To see this, note: 
$\widetilde u(r,\theta) = u(r \cos \theta, r \sin \theta) \quad \Longrightarrow \quad 
\partial_r \widetilde u = \partial_{x_1} u \cos\theta + \partial_{x_2} u \sin \theta$ and 
$\partial_\theta \widetilde u = - r \partial_{x_1} u \sin\theta + r \partial_{x_2} u \cos \theta$. 
}
\eex
\end{example} 
Generally, we have 
\begin{lemma}
\label{lemma:xrt-change-of-variables}
\begin{enumerate}[label=(\roman*)]
\item
\label{item:lemma:xrt-change-of-variables-i}
Let $|\alpha | \ne 0$. Then, writing $\beta = (\beta_1,\beta')$ 
\begin{align*}
D^{\alpha}_x u & = r^{-|\alpha|} \sum_{0 < |\beta| \leq |\alpha|} d_{\alpha\beta}(\theta) (r \partial_r)^{\beta_1} D^{\beta'}_\theta \widetilde u. 
\end{align*}
with $d_{\alpha\beta}$ smooth on $\partial B_1(0)$. Hence, 
\begin{align*}
D^{\alpha}_x u & = e^{-|\alpha| t} \sum_{0 < |\beta| \leq |\alpha|} d_{\alpha\beta}(\theta) D^{\beta}_{t,\theta} \check u
\end{align*}
\item
\label{item:lemma:xrt-change-of-variables-ii}
let $|\beta| \ne 0$. Then 
\begin{align*}
D^\beta_{t,\theta} \check u & = \sum_{0 < |\alpha| \leq |\beta|} r^{|\alpha|} d^\star_{\alpha\beta}(\theta) D^\alpha_x u
\end{align*}
\end{enumerate}
\end{lemma}
\begin{proof}
Induction on the differentiation order. 
\end{proof}
Using Lemma~\ref{lemma:xrt-change-of-variables}, we can show
\begin{lemma}
\label{lemma:uKsgamma-Mu}
For $s \in \BbbN_0$ and $\gamma \in \BbbR$ there holds 
$$
u \in K^s_\gamma(\mathcal{C}) \quad \Longleftrightarrow \quad e^{t (\gamma - s + d/2)} \check u \in H^s(\BbbR \times G). 
$$
\end{lemma}
\begin{proof}
``$\Longrightarrow$'': $u \in K^s_\gamma(\mathcal{C})$ implies $r^{-s + |\sigma| + \gamma} D^\sigma u \in L^2$ and 
\begin{align*}
e^{t (\gamma - s +d/2)} |D^\beta \check u| & \leq \sum_{|\alpha| \leq |\beta|} r^{|\alpha| + \gamma - s + d/2} |d^\star_{\beta\alpha}|\,|D^\alpha_x u| . 
\end{align*}
With the Jacobian $r dt = dr$ we get 
\begin{align*}
\|e^{t (\gamma - s + d/2)} D^\beta \check u\|^2_{L^2} 
\lesssim \sum_{|\alpha | \leq |\beta|} \int_{r,\theta} r^{2(|\alpha|+ \gamma - s + d/2)} |D^\alpha_{x} u|^2 \frac{dr}{r} \, d\theta
=  \sum_{|\alpha | \leq |\beta|} \int_{r,\theta} r^{2(|\alpha|+ \gamma - s)} |D^\alpha_{x} u|^2 r^{d-1} dr \, d\theta
\end{align*}
\end{proof}
The Mellin transform of $u \in K^s_\gamma(\mathcal{C})$ is taken as the Mellin transform w.r.t.\ $r$ leaving the variable $\theta$
as a parameter. A more elegant point of view is to consider Fourier transforms of functions $u \in L^2(\BbbR; X)$ with some
Hilbert space $X$ (e.g., $H^s(G)$). We note that the Paley-Wiener theorem also holds in that setting. 

\begin{theorem}
\label{thm:8.7}
Let $s \in \BbbN_0$, $\gamma \in \BbbR$. Set $\eta:= s - \gamma - d/2$ (Note: this is the negative of the exponent in 
Lemma~\ref{lemma:uKsgamma-Mu}.) Then:
\begin{enumerate}[label=(\roman*)]
\item
\label{item:thm:8.7-i}
Let $u \in K^s_\gamma(\mathcal{C})$. Then $e^{t (\gamma - s +d/2)} \check u \in L^2(\BbbR; H^s(G))$ and ${\mathcal M} u(\cdot - i \eta) \in 
L^2(\BbbR; H^s(G))$. Additionally, 
\begin{align}
\label{eq:thm:8.7-10}
\|u\|^2_{K^s_\gamma(\mathcal{C})} &\sim \int_{\xi \in \BbbR} \|{\mathcal M} u(\xi - i \eta)\|^2_{H^s(G; |\xi|)}\, d\xi
\end{align}
with 
\begin{align}
\nonumber 
\|v\|^2_{H^s(G; \rho)} & := \sum_{|\beta | \leq s} \rho^{2\beta_1} \|D^{\beta'} v\|^2_{L^2(G)}, \qquad \beta = (\beta_1,\beta') \\
\label{eq:thm:8.7-20}
& \sim \sum_{|\beta'| \leq s} (1+ \rho)^{2 (s - |\beta'|)}\|D^{\beta'} v\|^2_{L^2(G)}. 
\end{align}
\item
\label{item:thm:8.7-ii}
Conversely, let $U (\cdot - i \eta) \in L^2(\BbbR; H^s(G))$ with 
$$
\int_{\xi \in \BbbR} \|U (\cdot - i \eta)\|^2_{H^s(G; |\xi|)} \, d\xi < \infty. 
$$
Then $U (\cdot - i\eta)$ is the Mellin transform of a function $u \in K^s_\gamma(\mathcal{C})$.  
\end{enumerate}
\end{theorem}
\begin{proof}
\emph{Proof of \ref{item:thm:8.7-i}:} 
\begin{align*}
u \in K^s_\gamma(\mathcal{C}) \quad \stackrel{\text{Lemma~\ref{lemma:uKsgamma-Mu}}}{\Longrightarrow} \quad 
e^{-\eta t} \check u \in H^s(\BbbR \times G) \quad \Longrightarrow \quad e^{-\eta t} \check u \in L^2(\BbbR; H^s(G))
\end{align*}
Hence, ${\mathcal M} u(\cdot - i\eta) \in L^2(\BbbR; H^s(G))$. From Parseval and $e^{-\eta t}\check u \in H^s(\BbbR\times G)$ we 
get 
\begin{align*}
\|\partial_t^{\beta_1} D^{\beta'}_\theta (e^{-\eta t} \check u)\|^2_{L^2(\BbbR)} 
\sim \|\,|\xi|^{\beta_1} \partial^{\beta'}_\theta {\mathcal M} u(\cdot - i \eta)\|^2_{L^2(\BbbR)}. 
\end{align*}
Hence, 
\begin{align*}
\|u\|^2_{K^s_\gamma(\mathcal{C})} &\sim \|e^{-\eta t} \check u\|^2_{H^s(\BbbR \times G)}  
= \sum_{|\beta| \leq s} \|\partial^{\beta_1}_t \partial^{\beta'}_\theta (e^{-\eta t} \check u)\|^2_{L^2} \\
& \sim  \sum_{|\beta| \leq s} \int_{\xi} |\xi|^{2\beta_1} \|\partial^{\beta'}_\theta {\mathcal M} u(\cdot - i\eta)\|^2_{L^2(G)} . 
\end{align*}
This shows (\ref{eq:thm:8.7-10}). 
The norm equivalence in (\ref{eq:thm:8.7-20})  is best seen for the special case $s = 1$ and $d=2$: 
\begin{align*}
\sum_{|\beta| \leq 1}\rho^{2 \beta_1} \|D^{\beta'} v\|^2_{L^2(G)} & = \|v\|^2_{L^2(G)} + \rho^{2 \cdot 0} \|D^1 v\|^2_{L^2(G)} + \rho^{2 \cdot 1} \|D^0 v\|^2_{L^2(G)}
\sim (1 + \rho^2) \|v\|^2_{L^2(G)} + \|D^{\beta_2} v\|^2_{L^2(G)}. 
\end{align*}
\emph{Proof of \ref{item:thm:8.7-ii}:} 
The inverse FT of $U(\cdot - i \eta)$ is $e^{-\eta t} \check u \in L^2(\BbbR; H^s(G))$. By Parseval, in fact, 
$$
\|e^{-\eta t} \check u\|^2_{H^s(\BbbR \times G)} \sim \int_{\xi} \|U(\cdot - i \eta)\|^2_{H^s(G; |\xi|)}\, d\xi. 
$$ 
By norm equivalence (\ref{eq:thm:8.7-10}), we see $u \in K^s_\gamma$. 
\end{proof}
\subsection{Solving elliptic equations with Mellin transformation in the case $d = 2$}

We consider for $d = 2$ and $s \ge 2$ 
on $\mathcal{C}= \{(r\cos \theta, r \sin \theta)\,|\, r > 0, \theta \in G:= (0,\omega)\}$
\begin{align}
\label{eq:8.4} 
-\Delta u  = f \in K^{s-2}_\gamma \quad \mbox{ in $\mathcal{C}$}, 
& & u = 0 \quad \mbox{ for $\theta = 0$ and $\theta = \omega$}
\end{align}
\underline{goal:} show unique solvability in $K^s_\gamma$

\underline{warning:} (\ref{eq:8.4}) does not prescribe ``boundary conditions''
at $r=0$ and $r = \infty$. Hence, one  cannot expect uniqueness unless
one prescribes an ``energy space'', e.g., $K^s_\gamma$. 
In fact: for $\gamma \ne \delta$ solutions $u_\gamma \in K^s_\gamma$
and $u_\delta \in K^s_\delta$ of (\ref{eq:8.4}) may be different. 

Writing (\ref{eq:8.4}) in polar coordinates gives 
\begin{align}
\label{eq:8.5} 
- \left( (r \partial_r)^2 \widetilde u + \partial^2_\theta \widetilde u \right) & = r^2 \widetilde f, 
\end{align}
or, upon setting $r = e^t$,  
\begin{subequations}
\label{eq:8.6} 
\begin{align}
- \left( \partial^2_t \check u + \partial^2_\theta \check u \right) & = e^{2t} \check f
=: \check g \quad \mbox{ in $\BbbR \times G$},  \\
\check u(\cdot,0) & = \check u(\cdot,\omega)  = 0
\end{align}
\end{subequations}
Fourier transformation in $t$ yields 
\begin{align}
\label{eq:8.7} 
- (- \xi^2 + \partial^2_\theta) {\mathcal M} u(\xi) & = {\mathcal M} g(\xi) \ = {\mathcal F} (\check g)(\xi),
\end{align}
or, with the operator ${\mathcal L}(\xi):= (\xi^2 \operatorname{I} - \partial^2_\theta): H^2(G) \cap H^1_0(G) \rightarrow L^2(G)$ 
\begin{subequations}
\label{eq:8.8}
\begin{align}
{\mathcal L}(\xi) {\mathcal M} u(\xi) & = {\mathcal M}g(\xi),  \\
({\mathcal M}u)(\xi,0) & = ({\mathcal M} u)(\xi,\omega) = 0
\end{align}
\end{subequations}
Note: for $s \ge 2$ we have 
\begin{align}
\label{eq:cone-strip-equivalence}
f \in K^{s-2}_\gamma & \Longleftrightarrow e^{t (\gamma-s+2+1)} \check f \in H^{s-2}(\BbbR \times G)  \\
\nonumber 
& \Longleftrightarrow e^{t(\gamma-s+1)} \check g \in H^{s-2}  \\
\nonumber 
& 
\Longleftrightarrow {\mathcal M} g(\cdot-i \eta) \in L^2(\BbbR; H^{s-2}(G))
\end{align}
with 
$\eta = -(\gamma-s+1)$ and 
$$
\int_\xi \|{\mathcal M} g(\cdot - i \eta)\|^2_{H^{s-2}(G; |\xi|)}\, d\xi \sim \|f\|^2_{K^{s-2}_\gamma(\Gamma)} < \infty. 
$$
Next, we solve (\ref{eq:8.8}) for ${\mathcal M} u$. We observe 
\begin{lemma}
\label{lemma:8.8}
Consider 
\begin{subequations}
\label{eq:8.9} 
\begin{align}
- \partial^2_\theta \widehat  u + \zeta^2 \widehat u & = F \in L^2 \quad \mbox{ in $G$}, 
\\
\widehat u(0) & = \widehat  u(\omega)  = 0. 
\end{align}
\end{subequations}
Set 
\begin{equation}
\sigma^D:= \{\lambda_n:= \frac{\pi}{\omega} n\,|\, n \in \BbbN\}. 
\end{equation}
Then:
\begin{enumerate}[label=(\roman*)]
\item 
\label{item:lemma:8.8-i}
For $\zeta \in \BbbC \setminus (\pm i \sigma^D)$ problem (\ref{eq:8.9}) has 
a  unique solution $\widehat u_\zeta \in H^2(G) \cap H^1_0(G)$.  
\item
\label{item:lemma:8.8-ii}
Let $\{\xi + i \eta \,|\, \xi \in \BbbR\} \cap (\pm i \sigma^D) = \emptyset$. 
Then there is $C = C(\eta)$ such that for all $\zeta \in \{\xi + i \eta\,|\, \xi \in \BbbR\}$ the solution $\widehat u_\zeta$ of (\ref{eq:8.9}) satisfies 
\begin{align*}
\|\widehat u_\zeta\|_{H^2(G; |\xi|)} & \sim 
(1 + |\xi|^2)^2 \|\widehat u_\zeta\|_{L^2} + (1 + |\xi|^2) |\widehat u_\zeta|_{H^1} + |\widehat u_\zeta|_{H^2} \leq C \|F\|_{L^2}. 
\end{align*}
\end{enumerate}
\end{lemma}
\begin{proof}
We only show \ref{item:lemma:8.8-ii}. For simplicity, we work with the 
solution formula,noting  $\int_0^\omega \sin^2(\lambda_n \theta) = \frac{\omega}{2}$, 
\begin{align*}
\widehat u_\zeta &= \sum_n \frac{F_n}{\lambda^2_n + \zeta^2} \sin (\lambda_n \theta)
\\
F & = \sum_n F_n \sin (\lambda_n \theta); 
& \|F\|^2_{L^2} & = \sum_n |F_n|^2 
\end{align*}
and 
\begin{align*}
\|\widehat u_\zeta\|^2_{L^2(G)} & = \sum_n \frac{|F_n|^2}{|\zeta^2 + \lambda^2_n|^2}, \\
|\widehat  u_\zeta|^2_{H^1(G)} & = \sum_n \frac{|F_n|^2}{|\zeta^2 + \lambda^2_n|^2} \lambda^2_n,\\
|\widehat u_\zeta|^2_{H^2(G)} & = \sum_n \frac{|F_n|^2}{|\zeta^2 + \lambda^2_n|^2} \lambda^4_n
\end{align*}
The condition $\operatorname{Im} \zeta  \cap \sigma^D = \emptyset$ (i.e., $\eta \not\in \pm \sigma^D$) implies 
for fixed $\eta$
\footnote{for large $x^2 = \xi^2 + \lambda_n^2$ we have $x^2 - \eta^2 \sim x^2$ so the desired estimate follows; 
for $\xi$, $\lambda_n$ from a compact set, note $|\zeta^2 + \lambda_n^2| = |\zeta - i \lambda_n| \cdot |\zeta + i \lambda_n|$
i.e., the term is determined by the distance of $\pm i \sigma^D$ from the line $\operatorname{Im} \zeta = \eta$, which is assumed positive}
$$
\left| \frac{1}{(\xi + i \eta)^2 + \lambda^2_n}
\right|^2 = \frac{1}{(\xi^2 -\eta^2 + \lambda^2_n)^2 + 4 \xi^2 \eta^2}
\lesssim \min\{|\xi|^{-4}, \lambda^{-4}_n\}
$$
so that 
\begin{align*}
(1 + |\xi|^2)^2 \|\widehat u_\zeta\|^2_{L^2} + 
(1 + |\xi|^2) |\widehat u_\zeta|^2_{H^1} + 
|\widehat u_\zeta|^2_{H^2} \lesssim \sum_{n} |F_n|^2 \sim \|F\|^2_{L^2}. 
\end{align*}
\end{proof}
Lemma~\ref{lemma:8.8} allows us to estimate for (fixed) $\eta$ with $\{\eta \} \cap \sigma^D  = \emptyset$ and arbitrary 
$\zeta = \xi - i \eta$ 
the solution ${\mathcal M}u(\zeta)$ of (\ref{eq:8.7}): 
\begin{equation}
\label{eq:8.10}
\|{\mathcal M} u(\zeta)\|_{H^2(G; |\xi|)} \lesssim \|{\mathcal M}g(\zeta)\|_{L^2(G)}
\quad \mbox{uniformly in $\xi \in \BbbR$}. 
\end{equation}
Hence, by Theorem~\ref{thm:8.7} with $s = 2$, $\eta = 2-\gamma-1$ we get 
\begin{align}
\label{eq:proof-of-thm-8.9}
\|u\|^2_{K^2_\gamma(\mathcal{C})} &\sim 
\int_{\xi \in \BbbR} \|{\mathcal M} u(\cdot - i \eta) \|^2_{H^2(G; |\xi|)}
\lesssim \int_{\xi} \|{\mathcal M} g(\cdot - i \eta)\|^2_{L^2(G)}\, d\xi 
\sim \|e^{-t\eta} \check g\|^2_{L^2(\BbbR \times G)}  \\
\nonumber 
&\sim \|e^{t (-\eta+2)} \check f\|_{L^2(\BbbR \times G)} 
=  \|e^{t (\gamma+1)} \check f\|_{L^2(\BbbR \times G)} 
\sim \|f\|^2_{K^0_\gamma}. 
\end{align}
We have thus ``explicitly'' constructed the solution in $K^2_\gamma$ using the Fourier inversion formula 
on $\operatorname{Im} \zeta = -\eta = \gamma-1$. 
In general, one has
\begin{theorem}
\label{thm:8.9}
Let $s \in \BbbN$, $s \ge 2$, $\gamma \in \BbbR$. Set $\eta = -(\gamma - s + d/2)$ with $d = 2$. 
Assume that the line $\{\zeta = \xi + i \eta\,|\, \xi \in \BbbR\} \cap \pm i \sigma^D = \emptyset$. 
Then for every $f \in K^{s-2}_\gamma$
the problem (\ref{eq:8.4}) has a unique solution $u \in K^s_\gamma$ with the \textsl{a priori} estimate 
$$
\|u\|_{K^s_\gamma} \lesssim \|f\|_{K^{s-2}_\gamma}. 
$$
\end{theorem}
For future reference in the context of the $L^p$-setting in Appendix~\ref{appendix:Lp}, we formulate 
Theorem~\ref{thm:8.9} as a result in the strip 
\begin{equation*} 
D:= \BbbR \times G
\end{equation*}
in terms of the weighted norm for functions $\check{v} = \check{v}(t,\phi)$ and $\eta \in \BbbR$
\begin{equation}
\|\check{v}\|_{W^{s,p}_\eta(D)}:= \|e^{\eta t} \check{v} \|_{W^{s,p}(D)}. 
\end{equation}
\begin{corollary} 
\label{cor:strip-solution}
Let $D = \BbbR \times G$. 
Let $d = 2$, and let $s$, $\gamma$ satisfy the assumptions of Theorem~\ref{thm:8.9}. Set $\eta = -(\gamma - s +d/2)$ with $d = 2$. 
For every $\check{g} \in W^{s-2,2}_\eta(D)$
there is a unique solution $\check{u} \in W^{s,2}_\eta(D)$ of (\ref{eq:8.6}) and 
\begin{equation}
\label{eq:cor:strip-solution-10}
\|\check{u}\|_{W^{s,2}_{-\eta}(D)} \lesssim \|\check{g}\|_{W^{s-2,2}_{-\eta}(D)}. 
\end{equation}
Furthermore, let $\varepsilon > 0$ be such that $\{\zeta = \xi + i \eta'\,|\, \xi \in \BbbR, 
|\eta' - \eta| \leq \varepsilon\} \cap \pm i \sigma^D = \emptyset$.
Then, this solution $\check{u}$ satisfies for all $\eta' \in \BbbR$ with $|\eta' - \eta| \leq \varepsilon$ 
\begin{equation}
\label{eq:cor:strip-solution-20}
\|\check{u}\|_{W^{s,2}_{-\eta'}(D)} \lesssim \|\check{g}\|_{W^{s-2,2}_{-\eta'}(D)}. 
\end{equation}
\end{corollary}
\begin{proof} Given $\check{g}$, let $\check{f} = e^{-2t} \check{g}$ be defined by (\ref{eq:8.6}). Let $u$ be the unique solution 
given by Theorem~\ref{thm:8.9}. Then, the function $\check{u}$ in the $(t,\phi)$-variables corresponding to $u$ satisfies 
by the norm equivalence (\ref{eq:cone-strip-equivalence})
\begin{align*}
\|\check{u}\|_{W^{s,2}_{\gamma-s+1}(D)} & \sim \|u\|_{K^s_\gamma(\mathcal{C})} \lesssim \|f\|_{K^{s-2}_\gamma(\mathcal{C})} 
\sim \|\check{f}\|_{W^{s-2,2}_{\gamma-s+2+1}(D)}  = \|\check{g}\|_{W^{s-2,2}_{\gamma-s+1}(D)}. 
\end{align*}
This shows (\ref{eq:cor:strip-solution-10}). To see
(\ref{eq:cor:strip-solution-20}), one start with the case of $\check{g}$ having compact support so that its Fourier transform
is entire; the general case then follows by density. Recall that $\check{u}$ is obtained by a Fourier inversion formula. The Cauchy integral
theorem allows one to move line of integration in the Fourier inversion formula as long as one does not cross the spectrum
of ${\mathcal L}$ --- this argument will be worked out in the following Section~\ref{sec:solution-decomposition}. This is to say:
\begin{align*}
\sqrt{2 \pi} \check{u}(t,\phi) &= \int_{\operatorname{Im} \zeta = - \eta} e^{-i t \zeta} {\mathcal M} u(\zeta,\phi)\,d\zeta  
= \int_{\operatorname{Im} \zeta = -\eta'} e^{-i t \zeta} {\mathcal M} u(\zeta,\phi)\,d\zeta  
= \int_{\operatorname{Im} \zeta = -\eta'} e^{-i t \zeta} {\mathcal L}^{-1} \check{g} (\zeta,\phi)\,d\zeta.  
\end{align*}
The argument is then concluded with norm equivalences as done in the proof of Theorem~\ref{thm:8.9} in 
(\ref{eq:proof-of-thm-8.9}). 

\end{proof}
\section{Solution decomposition}
\label{sec:solution-decomposition}
Consider 
\begin{align}
\label{eq:8.11}
-\Delta u_1 &= f \in L^2, & u_1 &= 0 \quad \mbox{ for $\theta = 0$ and $\theta= \omega$}
\end{align}
with 
\begin{subequations}
\label{eq:8.12}
\begin{align}
u_1 & \in H^1(\mathcal{C}) & \operatorname{supp} u_1 & \subset B_1(0), \\
f  & \in L^2(\mathcal{C}) & \operatorname{supp} f & \subset B_1(0). 
\end{align}
\end{subequations}
We note: 
\begin{enumerate}[label=(\alph*)]
\item
$f \in K^0_0(\mathcal{C})$, $u_1 \in K^1_0(\mathcal{C}) \cap K^2_1(\mathcal{C})$ 
\item 
$\exists u_0 \in K^2_0(\mathcal{C})$ solution of (\ref{eq:8.11}) by Thm.~\ref{thm:8.9}. 
(The solution is obtained by the Fourier inversion formula on $\operatorname{Im} \zeta = -1$! the condition that the spectrum of ${\mathcal L}$ doesn't
meet this line is satisfied unless $\omega = \pi$--but this is not an interesting case! )
\end{enumerate}
\underline{question:} relation between $u_1$ and $u_0$?

\underline{answer:} $u_0$ and $u_1$ differ by at most a singularity function


Mellin transformation yields 
\begin{align*}
(-\partial^2_\theta + \zeta^2){\mathcal M} u_1 & = {\mathcal M} g \quad \mbox{ on $\{\operatorname{Im} \zeta > 0\}$}, \\
(-\partial^2_\theta + \zeta^2){\mathcal M} u_0 & = {\mathcal M} g \quad \mbox{ on $\{\operatorname{Im} \zeta =-1\}$} 
\end{align*}
\begin{itemize}
\item
$f \in K^0_0(\mathcal{C})$ \quad $\Longrightarrow$ \quad $e^{t(1-2)}\check g \in L^2(\BbbR \times G)$ \quad $\Longrightarrow$ \quad 
${\mathcal M} g$ holomorphic on $\{\operatorname{Im} \zeta > -1\}$ with values in $L^2(G)$
\item
$u_1 \in K^2_1(\mathcal{C})$ \quad $\Longrightarrow$ \quad $e^{t(-2+ 1+1)} \check u_1 \in H^2(\BbbR \times G)$ \quad $\Longrightarrow$ \quad
${\mathcal M} u_1$ holomorphic on $\{\operatorname{Im} \zeta > 0\}$ with values in $H^2(G)$
\item
${\mathcal M} g(\cdot-i) \in L^2(\BbbR; L^2(G))$ 
\item
${\mathcal M} u_1(\cdot) \in L^2(\BbbR; H^2(G))$ 
\item
${\mathcal M} u_0(\cdot-i) \in L^2(\BbbR; H^2(G))$ 
\item
$\zeta \mapsto ({\mathcal L}(\zeta))^{-1}$ is meromorphic on $\BbbC$ with poles at $\pm i \lambda_n$, $n \in \BbbN$.  
\end{itemize}
From 
$$
\underbrace{{\mathcal L}(\zeta)}_{\text{mero.}} \underbrace{{\mathcal M} u_1(\zeta)}_{\text{ hol. on } \operatorname{Im} \zeta > 0}
= \underbrace{{\mathcal M} g(\zeta)}_{\text{hol. on } \operatorname{Im} \zeta > -1}
$$
we define the meromorphic extension of ${\mathcal M} u_1$ to $\{\operatorname{Im} \zeta > -1\}$ by 
$$
{\mathcal U} (\zeta) := {\mathcal M} u_1(\zeta):= {\mathcal L}^{-1} {\mathcal M} g(\zeta)
$$
We note that 
$$
{\mathcal U}|_{\operatorname{Im} \zeta = -1} = {\mathcal M} u_0|_{\operatorname{Im} \zeta = -1}
$$
\begin{theorem}
\label{thm:8.10}
\begin{align*}
-u_1 + u_0 &= \sum_{\substack{ \zeta \in \pm i \sigma^D \\ \operatorname{Im} \zeta \in (-1,0)}} 
\frac{2 \pi i}{\sqrt{2 \pi}} \operatorname{Res} r^{i \zeta} {\mathcal L}(\zeta)^{-1} {\mathcal M} g(\zeta)
\end{align*}
(Note that the sum is either empty or has one exactly one term)
\end{theorem}
\begin{proof}
The Fourier inversion formulas (along the real line) yield 
\begin{align*}
\check u_1 & = {\mathcal F}^{-1} ({\mathcal M} u_1(\cdot)) = {\mathcal F}^{-1} {\mathcal U}(\cdot), \\
\check u_0 & = {\mathcal F}^{-1} ({\mathcal M} u_0(\cdot - i)) = {\mathcal F}^{-1} {\mathcal U}(\cdot-i). 
\end{align*}
As in the proof of Lemma~\ref{lemma:7.1}, the Fourier inversion formula is a line integral that can be shifted. 
We now assume $\lambda_1 < 1$ so that when shifting the integral from the $\operatorname{Im} \zeta = -1$ to the line
$\operatorname{Im} \zeta = 0$, one passes across the pole at $- i\lambda_1$, and a residue of 
$r^{i \zeta} {\mathcal U}(\zeta)$ at $\zeta = -i \lambda_1$ is picked up, i.e., 
$$
\operatorname{Res}_{\zeta = -i \lambda_1}  {\mathcal L}(\zeta)^{-1} {\mathcal M} g(\zeta) r^{i \zeta}. 
$$
The formal calculation is: 
\begin{align*}
\check u_1 (t,\theta) & = {\mathcal F}^{-1}({\mathcal U})(t,\theta)  
= \frac{1}{\sqrt{2\pi}} \int_{\operatorname{Im} \zeta = 0} e^{i \zeta t} {\mathcal U}(\zeta)\, d\zeta, \\
\check u_0 (t,\theta) & 
= \frac{1}{\sqrt{2\pi}} \int_{\operatorname{Im} \zeta = -1} e^{i \zeta t} {\mathcal U}(\zeta)\, d\zeta
\end{align*}
so that by Lemma~\ref{lemma:7.1} 
\begin{align*}
-\check u_1 + \check u_0 & = \frac{2 \pi i}{\sqrt{2\pi}} \operatorname{Res}_{\zeta = -i \lambda_1} e^{i \zeta t} {\mathcal U}(\zeta) 
\stackrel{r = e^t}{ =} \frac{2 \pi i}{\sqrt{2\pi}} \operatorname{Res}_{\zeta = -i \lambda_1} r^{i\zeta} {\mathcal U}(\zeta) . 
\end{align*}
The residue $\operatorname{Res}_{\zeta = -i \lambda_1} r^{i \zeta} {\mathcal U}(\zeta) = 
\operatorname{Res}_{\zeta = -i \lambda_1} r^{i \zeta} {\mathcal L}(\zeta)^{-1} {\mathcal M} g(\zeta)$  can be evaluated explicitly. 
We start from 
\begin{align*}
\underbrace{(-\partial^2_\theta + \zeta^2)}_{={\mathcal L}(\zeta)}{\mathcal U}(\zeta) & = {\mathcal M}g(\zeta)=: 
F(\zeta) = \sum_{n} F_n(\zeta)\sqrt{\frac{2}{\omega}} \sin(\lambda_n \theta). 
\end{align*}
We use the solution formula
\begin{align*}
{\mathcal U}(\zeta) &= \sum_n \frac{F_n(\zeta)}{\zeta^2 + \lambda^2_n} \sqrt{\frac{2}{\omega}} \sin(\lambda_n \theta) \\
&= \underbrace{ \frac{F_1(\zeta)}{\zeta^2 + \lambda^2_1}}_{ 
               = \frac{F_1(\zeta)}{2 i \lambda_1}((\zeta - i\lambda_1)^{-1} - (\zeta + i \lambda_1)^{-1})}
\sqrt{\frac{2}{\omega}} \sin(\lambda_1 \theta) 
+ 
\underbrace{ \sum_{n \ge 2} \frac{F_n(\zeta)}{\zeta^2 + \lambda^2_n} \sqrt{\frac{2}{\omega}} \sin(\lambda_n \theta) }_{\text{hol. in $\operatorname{Im} \zeta \in (-1,0)$}}
\end{align*}
Hence (note: $r^{i \zeta}$ is holomorphic near $\zeta = -i \lambda_1$)
\begin{align*}
\operatorname{Res}_{\zeta = -i \lambda_1} r^{i \zeta} {\mathcal L}(\zeta)^{-1} {\mathcal M} g(\zeta) 
&=  \frac{-F_1(-i \lambda_1)}{2 i \lambda_1} r^{i (-i \lambda_1)} \sqrt{\frac{2}{\omega}} \sin(\lambda_1 \theta)
\end{align*}
We further compute 
\begin{align*}
F_1(-i \lambda_1) & = \sqrt{\frac{2}{\omega}} \int_G \sin (\lambda_1 \theta) {\mathcal F}(e^{2t} \widetilde f(e^t,\theta)) |_{\xi = -i\lambda_1}\, d\theta \\
& = \sqrt{\frac{2}{\omega}} \int_G \sin (\lambda_1 \theta) \frac{1}{\sqrt{2\pi}} \int_{t} e^{2t} \widetilde f(e^t,\theta) e^{-i(-i \lambda_1) t}\, dt\, d\theta \\
& \stackrel{r = e^t}{=} \sqrt{\frac{2}{2 \pi \omega}} \int_G \int_r \sin (\lambda_1 \theta)  r  \widetilde f(r,\theta) r^{-\lambda_1}\, dr\, d\theta \\
& = \sqrt{\frac{2}{\omega}} \frac{1}{\sqrt{2\pi}} \int_\mathcal{C}r^{-\lambda_1} \sin(\lambda_1 \theta) f\, dx 
\end{align*}
\end{proof}
The procedure in Theorem~\ref{thm:8.10} shows that the difference $u_0  - u_1$ is given by 
$$
u_1 - u_0 = \ell(f) r^{\lambda_1} \sin(\lambda_1 \theta)
$$
where $r^{\lambda_1} \sin(\lambda_1 \theta)$ is a solution of the homogeneous equation that satisfies the boundary conditions. 
The linear functional $\ell$ is given explicitly by 
$$
\ell(f) = -\frac{1}{\pi} \int_\mathcal{C}r^{-\lambda_1} \sin(\lambda_1 \theta) f\, dx
$$
We also note that the function $u_0 \in H^2(\mathcal{C}\cap B_1(0))$, i.e., is ``regular'' near $r = 0$. 

As a corollary of Theorem~\ref{thm:8.10}, we have 
\begin{corollary}
\label{cor:8.11}
Let $\Omega\subset \BbbR^2$ be a polygon with vertices $A_1,\ldots,A_J$ and interior angles $\omega_j \in (0,2\pi)/\setminus \{\pi\}$. Then, for 
every $f \in L^2(\Omega)$ the solution $u \in H^1_0(\Omega)$ of $-\Delta u = f$ can be written as 
$u = u_{H^2} + \sum_{j=1}^J c_j \chi_j r_j^{\pi/\omega_j} \sin(\frac{\pi}{\omega_j} \theta)$, where:  
$\chi_j$ is a smooth cut-off function with $\chi_j \equiv 1$ near $A_j$; 
the coefficients $c_j = 0$ if $\omega_j < \pi$ and for $\omega_j > \pi$ it is given by 
\begin{align*}
c_j (f) &= -\frac{1}{\pi} \int_\Omega f s_j \eta_j + u \Delta (s_j \eta_j),  \\
s_j &:= r_j^{-\pi/\omega_j} \sin \frac{\pi}{\omega_j} \theta, 
\end{align*}
where $\eta_j$ is an arbitrary  smooth cut-off function with $\eta_j \equiv 1$ near $A_j$; $u_{H^2} \in H^2$ with $\|u_{H^2}\|_{H^2(\Omega)} \leq C \|f\|_{L^2(\Omega)}$. 
\end{corollary}
\begin{proof}
localize near $A_j$ using $\eta_j$: Then $(\eta_j u)$ solves 
$$
-\Delta (\eta_j u) = - (\Delta \eta_j u+ 2 \nabla \eta_j \cdot \nabla u + \eta_j \underbrace{\Delta u}_{ = -f})
$$
From Theorem~\ref{thm:8.10}, we get 
\begin{align*}
c_j & = -\frac{1}{\pi} \int_\Omega s_j (-\Delta (\eta_j u) ) = \cdots \stackrel{\Delta s_j = 0}{=} 
-\frac{1}{\pi} \int_\Omega f \eta_j s_j + u \Delta (\eta_j s_j) \,dx. 
\end{align*}
\end{proof}
\subsection{Solution decomposition for the Neumann problem}
Consider 
\begin{equation}
-\Delta u_1 = f \in L^2, \quad \partial_n u_1 = 0 \quad \mbox{ for $\theta=  0$ and $\theta = \omega$}
\end{equation}
with 
\begin{align*}
u_1 &\in H^1(\mathcal{C}), 
& \operatorname{supp} u_1 & \subset B_1(0), 
& \operatorname{supp} f & \subset B_1(0). 
\end{align*}
Note: $f \in K^0_0(\mathcal{C})$, $u_1 \in K^1_\delta$ for every $\delta > 0$ ($\delta =  0$ not possible)

Mellin transformation yields 
$$
\underbrace{(-\partial^2_\theta + \zeta^2 )}_{=:{\mathcal L}^N(\zeta)}  {\mathcal M} g_1  = {\mathcal M} g \quad \mbox{ on $\{\operatorname{Im} \zeta > \delta\}$} 
$$
and the boundary conditions are $\partial_\theta {\mathcal M} u_1(0) = \partial_\theta {\mathcal M} u_1(\omega) = 0$. 

One can check: 
\begin{lemma}
Set $\sigma^N:= \{\lambda^N_n:= n \frac{\pi}{\omega}\,|\, n \in \BbbN_0\}$. Then 
\begin{align}
(-\partial^2_\theta + \zeta^2) \widehat  u =  F \in L^2(G), 
\quad \widehat u^\prime(0) = \widehat u^\prime(\omega) = 0
\end{align}
has a unique solution $\widehat u_\zeta$ for $\zeta \not\in \pm i \sigma^N$ and 
$$
\|\widehat u_\zeta\|_{H^2(G; |\zeta|)} \lesssim \|F\|_{L^2(G)}. 
$$
\end{lemma}
\begin{corollary}
For $f \in K^0_0(\mathcal{C})$ there exists a unique $u_0 \in K^2_0$ solution of 
\begin{align*}
-\Delta u_0 = f \quad \mbox{ in $\mathcal{C}$}, 
\qquad \partial_n u_0 = 0 \quad \mbox{ on $\partial\mathcal{C}$}. 
\end{align*}
\end{corollary}
\begin{proof}
We proceed as in the Dirichlet case. We have to make sure that the line $\operatorname{Im} \zeta = -1 = 0-2+1 (= \gamma-s+d/2)$ does
not meet $\pm i \sigma^N$. This is the case for $\omega \ne \pi$. 
\end{proof}
As in the proof of Theorem~\ref{thm:8.10}, we get 
$$
-u_1 + u_0 = \frac{2 \pi i}{\sqrt{2 \pi}} \sum_{\substack{ \zeta \in -i \sigma^N \\ \operatorname{Im} \zeta \in (-1,\delta)}} 
\operatorname{Res} r^{i \zeta} {\mathcal L}(\zeta)^{-1} {\mathcal M} g(\zeta)
$$
\begin{itemize}
\item
The residue at $-i \lambda_1 = - i \frac{\pi}{\omega}$ is computed as in the Dirichlet case and yields 
$$
- r^{\lambda_1} \cos (\lambda_1 \theta) \frac{1}{\pi} \int_{\mathcal{C}}r^{-\lambda_1} \cos (\lambda_1 \theta) f
$$
\item
$\zeta = 0$ is a double pole of ${\mathcal L}(\zeta)^{-1}$. We have 
\begin{align*}
{\mathcal U}(\zeta) & = \frac{F_0(\zeta)}{\zeta^2} \sqrt{\frac{1}{\omega}} 
+ \underbrace{ \sum_{n=1}^\infty \frac{F_n(\zeta)}{\zeta^2 + \lambda_n^2} \sqrt{\frac{2}{\omega}} \cos (\lambda_n \theta) }_
             {\rightarrow \text{ Residue at $- i \lambda_1$}}
\end{align*}
and 
\begin{align*}
\frac{r^{i\zeta} F_0(\zeta)}{\zeta^2}\sqrt{\frac{1}{\omega}} = 
& \sqrt{\frac{1}{\omega}} \frac{1}{\zeta^2} F_0(\zeta) e^{i \zeta \ln r} 
\stackrel{\text{Taylor}}{=} \sqrt{\frac{1}{\omega}} \frac{1}{\zeta^2} (F_0(0) + F^\prime_0(0) \zeta + \cdots ) ( 1  + i \ln r{\zeta} + \cdots ) \\
& 
\sqrt{\frac{1}{\omega}} \left( F_0(0) \frac{1}{\zeta^2} + (F^\prime_0(0) + F_0(0) i \ln r) \frac{1}{\zeta} + \cdots 
\right)
\end{align*}
Therefore 
\begin{align*}
\operatorname*{Res}_{\zeta = 0} r^{i\zeta} \frac{F_0(\zeta)}{\zeta^2} \sqrt{\frac{1}{\omega}} 
& =
\sqrt{\frac{1}{\omega}} \left( i \ln r F_0(0) + F^\prime_0(0)\right) 
 = a \ln r + b
\end{align*}
for suitable $a$, $b \in \BbbC$. 
\footnote{calculation of $a$: 
\begin{align*}
a &= i \sqrt{\frac{1}{\omega}} F_0(0) 
=  i \sqrt{\frac{1}{\omega}}  \sqrt{\frac{1}{\omega}} \int_G {\mathcal F}(e^{2t} f(e^t,\theta))|_{\xi = 0}  
=  i\frac{1}{\omega} \frac{1}{\sqrt{2\pi}} \int_{\theta=0}^\omega\int_t e^{2t} f(e^t,\theta)
=  i \frac{1}{\omega} \frac{1}{\sqrt{2\pi}} \int_{\mathcal{C}}f
\end{align*}
}
We conclude 
$$
u_1 = \underbrace{u_0}_{\in H^2} + a  \ln r + b + \underbrace{\gamma}_{\in \BbbR} r^{\pi/\omega} \cos (\frac{\pi}{\omega} \theta)
$$
If we inject the information that $u_1 \in H^1(\mathcal{C})$, then we see that $a = 0$. Hence, 
$$
u_1 = u_0 + b+ \gamma r^{\pi/\omega} \cos (\frac{\pi}{\omega} \theta). 
$$
\end{itemize}
\begin{remark}
The example of the Neumann problem illustrates several things: 
\begin{itemize}
\item the operator ${\mathcal L}$ may have poles of higher order. This leads to the introduction of the factors $r^{\lambda} \log^k r$ 
in the singularity function: To evaluate the residue at $\lambda$ of 
$(\zeta - \lambda)^{-n} F(\zeta) e^{i \ln r \zeta}$, 
one expands 
$F(\zeta) e^{i \ln r \zeta}$ as a Taylor series at $\zeta = \lambda$ to get 
\begin{align*}
F(\zeta) e^{i \ln r \zeta} &= 
e^{ i \lambda \ln r} F(\zeta) e^{i \ln r (\zeta-\lambda)} = 
e^{i \lambda \ln r} \sum_{j=0}^\infty F_j (\zeta - \lambda)^j \sum_{k=0}^\infty \frac{1}{k!} (i \ln r)^k (\zeta - \lambda)^k  \\
& = e^{i \lambda \ln r} \sum_{j=0}^\infty (\zeta - \lambda)^j \sum_{k=0}^j \frac{1}{k!} (i \ln r)^k F_{k-j}, 
\end{align*}
so that the residue of $(\zeta-\lambda)^{-n} F(\zeta) e^{i \ln r \zeta}$ at $\lambda$ is 
$$
2 \pi i e^{i \lambda \ln r} \sum_{k=0}^{n-1} \frac{1}{k!} (i \ln r)^k F_{k-j}
$$
\item 
In general, not only ${\mathcal L}$ has poles but also the right-hand side ${\mathcal M}g$. This arises for smooth right-hand sides
as discussed in Lemma~\ref{lemma:8.2}. These poles in the solution correspond to polynomials (since they are at $-i n$, $n \in \BbbN$). 
\eremk
\end{itemize}
\end{remark}

\section{Solving elliptic equations with the Mellin calculus: the $L^p$-setting}
\label{appendix:Lp}
We present classical material about solving elliptic equation using the 
Mellin calculus. In this section, we focus on the $L^p$-setting. Extensive references
paying special attention to the 3D case including edges are the monographs 
\cite{nazarov-plamenevsky-1994, mazya-rossmann2010, KozlovMazyaRossmann2001}. 
For certain aspects, we follow the original work 
\cite{mazya-plamenevskii78, mazya-plamenevskii78-translated}. 
\subsection{$L^p$-analysis of parameter-dependent problems}
The analysis of the parameter-dependent problem (\ref{eq:8.9}) in an $L^p$-setting 
is conveniently done in the analog of the the norms $H^k(G; |\xi|)$ given,  
for $1 < p < \infty$, by 
\begin{equation}
\|v\|^p_{W^{s,p}(G; \rho)}:= \sum_{|\beta'| \leq s} (1 + \rho)^{2(s-|\beta'|)} \|D^{\beta'} v\|^p_{L^p(G)}. 
\end{equation}

We start with the analog of Lemma~\ref{lemma:8.8} in an $L^p$-setting: 
\begin{lemma}
\label{lemma:8.8-lp}
Let $\widehat u$ solve 
\begin{subequations}
\label{eq:8.9-lp}
\begin{align}
- \partial^2_\theta \widehat  u + \zeta^2 \widehat u & = F \in L^2 \quad \mbox{ in $G$},
\\
\widehat u(0) & = \widehat  u(\omega)  = 0.
\end{align}
\end{subequations}
Set 
\begin{equation}
\sigma^D:= \{\lambda_n:= \frac{\pi}{\omega} n\,|\, n \in \BbbN\} 
\end{equation}
and let $p \in (1,\infty)$. Then:
\begin{enumerate}[label=(\roman*)]
\item
\label{item:lemma:8.8-i-lp}
For $\zeta \in \BbbC \setminus (\pm i \sigma^D)$ problem (\ref{eq:8.9}) has
a  unique solution $\widehat u_\zeta \in W^{2,p}(G) \cap W^{1,p}_0(G)$.
\item
\label{item:lemma:8.8-ii-lp}
Let $\{\xi + i \eta \,|\, \xi \in \BbbR\} \cap (\pm i \sigma^D) = \emptyset$.
Then there is $C = C(\eta)$ such that for all $\zeta \in \{\xi + i \eta\,|\, \xi \in \BbbR\}$ the solution $\widehat u_\zeta$ of (\ref{eq:8.9}) satisfies
\begin{align}
\label{lemma:8.9-lp}
\|\widehat u_\zeta\|_{W^{2,p}(G; |\xi|)} & \lesssim 
(1 + |\xi|^2)^2 \|\widehat u_\zeta\|_{L^p(G)} + (1 + |\xi|^2) |\widehat u_\zeta|_{W^{1,p}(G)} + |\widehat u_\zeta|_{W^{2,p}(G)} \leq C \|F\|_{L^p(G)}.
\end{align}
\end{enumerate}
\end{lemma}
\begin{proof}
\emph{1.~Step ($L^p$-estimate in full space, $\operatorname{Re} \zeta^2 > 0$):} 
Consider $-\Delta v  + \zeta^2 v = f $ in $\BbbR^d$. By $\operatorname{Re} \zeta^2 > 0$, we consider solutions that decay at $\infty$. 
By Fourier transformation (with Fourier variable $\xi$), the solution $\widehat{v}$ is given by 
\begin{equation*}
\widehat v(\xi) = \frac{1}{|\xi|^2  + \zeta^2} \widehat{f}. 
\end{equation*} 
By the Mikhlin multiplier theorem (see, e.g., \cite[Chap.~{IV}, Thm.~{3}]{emstein}) 
The mappings $f \mapsto |\zeta|^2 v$, $f \mapsto |\zeta| D v$, and $f \mapsto D^2 v$ are uniformly-in-$\zeta$ bounded maps in $L^p$ with 
\begin{equation}
\label{eq:result-mikhlin-multiplier-thm}
|\zeta|^2 \|v\|_{L^p(\BbbR^d)} + |\zeta| \|D v\|_{L^p(\BbbR^d)} + \|D^2 v\|_{L^p(\BbbR^d)} \leq C_p \|f\|_{L^2(\BbbR^d)}. 
\end{equation}
\emph{2.~Step ((\ref{lemma:8.9-lp}) for $\zeta = \xi + i \eta$ with large $|\xi|$):} To make use of the full-space result of Step~1, we localize the solution 
$\widehat{u}$ of (\ref{eq:8.9-lp}) by considering $\varphi \widehat{u}$ for a smooth cut-off function $\varphi$. Near the boundary, 
we antisymmetrically extend $\widehat{u}$ so as to preserve the $W^{2,p}$-regularity. The function $\varphi \widehat{u}$ then 
satisfies an equation of the form studied in Step~1 with right-hand side $f = \varphi F + 2 \varphi^\prime \widehat{u}^\prime + \varphi^{\prime\prime} \widehat{u}$. 
From Step~1, we then get the localized estimate 
\begin{equation*}
\|D^2 (\varphi \widehat u)\|_{L^p(\BbbR^d)} + 
|\zeta| \|D (\varphi \widehat u)\|_{L^p(\BbbR^d)} + 
|\zeta|^2 \|\varphi \widehat u\|_{L^p(\BbbR^d)} 
\lesssim \|\varphi F\|_{L^p(\BbbR^d)} + \|\varphi^\prime \widehat u^\prime\|_{L^p(\BbbR^d)} + \|\varphi^{\prime\prime} \widehat u\|_{L^p(\BbbR^d)}. 
\end{equation*}
Using a partition of unity and summing the local contributions, we obtain from this 
\begin{equation*}
\|D^2 \widehat u\|_{L^p(G)} + 
|\zeta| \|\widehat u\|_{L^p(G)} + 
|\zeta|^2 \|\widehat u\|_{L^p(G)} 
\lesssim \|F\|_{L^p(G)} + \|\widehat u^\prime\|_{L^p(G)} + \|\widehat u\|_{L^p(G)}. 
\end{equation*}
For $|\zeta|$ sufficiently large, the terms 
$\|D \widehat {u} \|_{L^p(G)}$  and $\|\widehat {u} \|_{L^p(G)}$ of the right-hand can be absorbed by the left-hand side, and the claimed estimate is shown. 

\emph{3.~Step ((\ref{lemma:8.9-lp}) for finite $\zeta$):} Consider for sufficiently large $\lambda > 0$ the equation 
\begin{equation}
\label{eq:lemma:8.9-lp-10}
-v^{\prime\prime} + \lambda^2 v  = F + \lambda^2 v - \zeta^2 v, \quad v(0) = v(\omega) = 0. 
\end{equation}
Since by Step~2, the left-hand side operator is a bounded linear operator $L^p(G) \rightarrow W^{2,p}(G)$, we have the Fredholm alternative
for (\ref{eq:8.9-lp}). Since by assumption $\zeta \not\in \pm i \sigma^D$, we have uniqueness of solutions of (\ref{eq:8.9-lp}). We conlude from the Fredholm alternative existence
and thus, for each $\zeta \not\in \pm i \sigma^D$, a constant $C_\zeta > 0$ with 
\begin{equation*}
\|\widehat{u}\|_{W^{2,p}(G)} \leq C_\zeta \|F\|_{L^p(G)}. 
\end{equation*}
We next show uniformity of $C_\zeta$ if $\zeta$ ranges in a compact set $K$ satisfying $K \cap \pm i \sigma^D  = \emptyset$. This follows by 
contradiction. Otherwise, there is a sequence $(\widehat{u}_n,\zeta_n)_{n} \subset W^{2,p}(G)\setminus \{0\} \times K $ such that $F_n:= -\widehat{u}^{\prime\prime}_n + \zeta_n^2 \widehat{u}_n$
satisfies 
\begin{align*}
\forall n \in \BbbN \quad \colon \quad 
\|\widehat{u}_n \|_{W^{2,p}(G)} > n \|F_n\|_{L^p(G)} . 
\end{align*}
We may normalize $\|\widehat{u}_n \|_{W^{2,p}(G)} = 1$. This implies $F_n \rightarrow 0$ in $L^p(G)$. Furthermore, by Banach-Alaoglu, 
there is $\widehat{u}_\infty \in W^{2,p}(G)$ with $\widehat{u}_n \rightharpoonup \widehat{u}_\infty$ weakly in $W^{2,p}(G)$. By compact
embedding $W^{2,p}(G) \subset W^{1,p}(G)$, we may assume $\widehat{u}_n \rightarrow \widehat{u}_\infty$ in $W^{1,p}(G)$. 
By compactness of $K$, we may also assume $\zeta_n \rightarrow \zeta \in K$ as $n \rightarrow \infty$. The limit $\widehat{u}_\infty$
satisfies in view of $F_n \rightarrow 0$ 
\begin{equation}
-\widehat{u}^{\prime\prime}_\infty + \zeta^2 \widehat{u}_\infty  = 0, \quad \widehat{u}_\infty(0) = \widehat{u}_\infty(\omega) = 0. 
\end{equation}
Since $\zeta \in K$ and $K \cap \pm i \sigma^D = \emptyset$, we conclude $\widehat{u}_\infty = 0$. Hence, $\widehat{u}_n \rightarrow 0$ in $W^{1,p}(G)$. 
Finally, from (\ref{eq:lemma:8.9-lp-10}) for fixed $\lambda$ sufficiently large we have the \textsl{a priori} bound 
\begin{align*}
\|\widehat{u}_n \|_{W^{2,p}(G)} \lesssim \|F_n\|_{L^p(G)} + \|\widehat{u}_n\|_{W^{1,p}(G)} \rightarrow 0 \quad \mbox{ as $n \rightarrow \infty$.}
\end{align*}
We conclude that $\widehat{u}_n \rightarrow 0 = \widehat{u}_\infty$ in $W^{2,p}(G)$. But this contradicts $\|\widehat{u}_n\|_{W^{2,p}(G)} = 1$ 
for all $n \in \BbbN$.  
\end{proof}
\begin{remark}
The explicit $\zeta$-dependence in (\ref{eq:result-mikhlin-multiplier-thm})
relies on the H\"ormander-Mikhlin multiplier theorem in full-space. 
In the present 1D setting, an alternative proof could be based on Fourier series representations of the solution as done in the 
proof of the corresponding $L^2$-result in Lemma~\ref{lemma:8.8}. For this, the basic ingredient is that by De Leeuw's theorem (see, e.g., \cite[Chap.~{VII}, Thm.~{3.8}]{stein-weiss71}), the 
Mikhlin multiplier theorem also holds for Fourier series if the Fourier multiplier is the restriction to ${\mathbb Z}$ of a full-space Fourier multiplier satisfying
the conditions of the Mikhlin multiplier theorem. 
\eremk
\end{remark}
\subsection{Elliptic problems in a strip}
Following \cite{mazya-plamenevskii78,mazya-plamenevskii78-translated}, we consider equation (\ref{eq:8.6}) in the strip 
$D:= \BbbR \times G$. The procedure to obtain estimates in exponentially weighted spaces differs from that in the $L^2$-based
setting as Parseval's equality is not available in an $L^p$-setting. Instead, in the analysis of (\ref{eq:8.6}) one localizes
in such a way that exponential functions have bounded variation and connects with the $L^2$-results through lower order terms. 
In the strip $D:= \BbbR \times G = \BbbR \times (0,\omega)$, it is convenient to introduce the norms 
\begin{equation}
\|f\|_{W^{s,p}_\beta(D)}:= \|\exp(\beta t) f\|_{W^{s,p}(D)}
\end{equation}

The point of the following lemma is that the lower order norm of the solution is measured in $L^2$ rather than $L^p$, which allows 
one to employ $L^2$-based results from Appendix~\ref{appendix:L2}. 
\begin{lemma}[\protect{\cite[(4.6)]{mazya-plamenevskii78-translated}}]
\label{lemma:mazya-plamenevskii78-eq.4.6}
Let $\check{u}$ solve (\ref{eq:8.6}), i.e., 
\begin{subequations}
\label{eq:8.6-Lp}
\begin{align}
- (\partial_t^2 \check{u} + \partial_\theta^2 \check{u})  & = \check{g} \quad \mbox{ in $D = \BbbR \times G$}, \\
\check{u}(\cdot,0) & = \check{u}(\cdot,\omega)  = 0 
\end{align}
\end{subequations}
Let $\eta_1 = \eta_1(t)$, $\eta_2  = \eta_2(t) \in C^\infty_0(\BbbR)$ with $\eta_2 \equiv 1$ on $\operatorname{supp} \eta_1$. 
Then, for $k \in \BbbN_0$, 
\begin{align}
\label{eq:lemma:mazya-plamenevskii78-eq.4.6-10}
\|\eta_1 \check{u}\|_{W^{k+2,p}}(D)  \leq C \|\eta_2 \check{g}\|_{W^{k,p}(D)} + \|\eta_2 \check{u}\|_{L^2(D)}, \\  
\label{eq:lemma:mazya-plamenevskii78-eq.4.6-15}
\mbox{ and } \qquad \|\eta_1 \check{u}\|_{W^{k+2,p}}(D)  \leq C \|\eta_2 \check{g}\|_{W^{k,p}(D)} + \|\eta_2 \check{u}\|_{L^p(D)}. 
\end{align}
\end{lemma}
\begin{proof}
For the proof of (\ref{eq:lemma:mazya-plamenevskii78-eq.4.6-15}), we refer to 
\cite[Thm.~{9.11}]{Gilbarg} (interior estimate) and \cite[Thm.~{9.13}]{Gilbarg} (boundary estimates). 

We now show (\ref{eq:lemma:mazya-plamenevskii78-eq.4.6-10}). 
We restrict to the case $k = 0$ and, for simplicity, on the case of interior regularity of a $D' \subset\subset D$. 
We select $D^{\prime\prime}$, $D^{\prime\prime\prime}$ with 
$D' \subset\subset D^{\prime\prime} \subset\subset D^{\prime\prime\prime} \subset\subset D$. 
We split $\check{u} = w + z$, where $w$ is given by the Newton potential of $\eta \check{g}$ for a suitable cut-off function $\eta$ 
with $\eta \equiv 1$ on $D^{\prime\prime}$. 
By, e.g., \cite[Thm.~{9.9}]{Gilbarg}, we may assume $\|w\|_{W^{2,p}(\BbbR^2)} \lesssim \|\check{g}\|_{L^p(D^{\prime\prime\prime})}$. 
Since $-\Delta z = 0$ on $D^{\prime\prime}$, interior regularity gives 
\begin{equation*}
\|z\|_{H^{s}(D')} \leq C_s \|z\|_{L^2(D^{\prime\prime})} 
\end{equation*}
for any $s \ge 0$.  Choosing $s$ so that the Sobolev embedding theorem $H^s(D') \subset W^{2,p}(D')$ holds, we get 
\begin{align*} 
\|\check{u}\|_{W^{2,p}(D')} \leq \|w\|_{W^{2,p}(D')} + \|z\|_{W^{2,p}(D')} 
& \lesssim \|\check{g}\|_{L^p(D^{\prime\prime})} + \|z\|_{L^2(D^{\prime\prime})}
\lesssim \|\check{g}\|_{L^p(D^{\prime\prime\prime})} + \|\check{w}\|_{L^2(D^{\prime\prime})} + \|\check{u}\|_{L^2(D^{\prime\prime})} \\
\lesssim \|\check{g}\|_{L^p(D^{\prime\prime\prime})} + \|\check{u}\|_{L^2(D^{\prime\prime})}.  
\end{align*} 
\end{proof}
The following lemma is the key ingredient for introducing exponential weights using localization: 
\begin{lemma}[\protect{\cite[Lemma~{4.1}]{mazya-plamenevskii78-translated}}] 
\label{lemma:mazya-plamenevskii78-lemma-4.1}
Let $(\zeta_\ell)_{\ell \in \BbbZ}$ be a smooth partition of unity subordinate to the covering of $\BbbR$ by the intervals 
$[\ell-1,\ell+1]$, $\ell \in \BbbZ$. Let $\eta \in \BbbR$ be such that $\{\xi + i \eta\,|\, \xi \in \BbbR\} \cap \pm i \sigma^D  = \emptyset$. 
Let $\check{g} \in L^{p}(D) \cap L^2(D)$ satisfy $\operatorname{supp} \check{g} \subset [m-1,m+1] \times G$ for some $m \in \BbbZ$. 
Then there is $\varepsilon > 0$ such that the solution $\check{u}$ of (\ref{eq:8.6-Lp}) provided by Corollary~\ref{cor:strip-solution} 
satisfies 
\begin{align}
\label{eq:lemma:mazya-plamenevskii78-lemma-4.1-100}
\|e^{\eta t} \zeta_\ell \check{u}\|_{L^p(D)} \lesssim e^{-|m - \ell| \varepsilon} \|e^{\eta t} \check{g} \|_{L^p(D)}. 
\end{align}
\end{lemma}
\begin{proof}
We follow the proof of \cite[Lemma~{4.1}]{mazya-plamenevskii78-translated}. Let $\varepsilon'>0$ be so small that 
$\{\zeta \in \BbbC\,|\, \operatorname{Im} \zeta \in [\eta-\varepsilon',\eta+\varepsilon']\} \cap \pm i \sigma^D  = \emptyset$. 
We start by noting that $\check{g} \in L^2(D)$ implies that the 
solution $\check{u}$ given by Corollary~\ref{cor:strip-solution} 
satisfies for any $\eta_1 \in [\eta- \varepsilon',\eta+\varepsilon']$ 
\begin{align}
\label{eq:lemma:mazya-plamenevskii78-lemma-4.1-10}
\|\check{u}\|_{W^{2,2}_{\eta_1}(D)} \lesssim \|\check{g}\|_{W^{0,2}_{\eta_1}(D)}. 
\end{align}
Define 
\begin{equation*}
M_{p,\eta}:= \|e^{\eta t}\check{g}\|_{L^p(D)}. 
\end{equation*}
\emph{The case $p \ge 2$:} For any $\eta_1 \in (\eta - \varepsilon', \eta+\varepsilon')$ we estimate 
\begin{align*}
\left(\int_{\ell-1}^{\ell+1} \|e^{\eta t} \check{u}\|^2_{L^2(G)}\,dt\right)^{1/2} 
&\stackrel{(\ref{eq:lemma:mazya-plamenevskii78-lemma-4.1-10})}{ \lesssim } e^{(\eta - \eta_1)\ell} M_{2,\eta_1} 
 \stackrel{\text{$p \ge 2$, $\operatorname{supp} \check{g} \subset [m-1,m+1]$}}{\lesssim } e^{(\eta - \eta_1)\ell} M_{p,\eta_1} 
\lesssim e^{(\eta - \eta_1)(\ell-m)} M_{p,\eta}. 
\end{align*}
For $|\ell - m |\leq 2$, the estimate (\ref{eq:lemma:mazya-plamenevskii78-eq.4.6-10}) leads to 
\begin{align}
\label{eq:lemma:mazya-plamenevskii78-lemma-4.1-20}
\|e^{\eta t} \zeta_\ell \check{u}\|_{L^p(D)} \lesssim e^{(\eta - \eta_1)(\ell-m)} M_{p,\eta} + \left(\int_{\ell-2}^{\ell+2} \|e^{\eta t} \check{u}\|^2_{L^2(G)}\, dt\right)^{1/2}. 
\end{align}
For $|\ell - m| > 2$, the estimate (\ref{eq:lemma:mazya-plamenevskii78-lemma-4.1-20}) is even true without the term 
$e^{(\eta - \eta_1)(\ell-m)} M_{p,\eta}$ by the interior regularity arguments used in the proof of Lemma~\ref{lemma:mazya-plamenevskii78-eq.4.6}. 
Selecting $\eta_1=\eta+\varepsilon'$ for $m < \ell$ and $\eta_1 = \eta-\varepsilon'$ for $m > \ell$  shows 
(\ref{eq:lemma:mazya-plamenevskii78-lemma-4.1-100}). 

\emph{The case $p \leq 2$:} 
By the boundedness of $\zeta_\ell$ and the compact support of $\zeta_\ell$, the H\"older inequality gives with the conjugate exponent $p' = p/(p-1)$ 
\begin{align*}
\|e^{\eta_1 t} \zeta_\ell \check{u}\|^p_{L^p(D)} \lesssim \int_{G} \left(\int_{t \in \BbbR} |e^{\eta_1 t} \check{u}(x,t)|^{p'}\,dt\right)^{p/p'}\,dx
\end{align*}
Applying the Hausdorff-Young inequality to the inner integral yields for the (partial) Fourier transform $\widehat u(x,\zeta):={\mathcal F}_{t \rightarrow\zeta} \check{u}(x,t)$  
\begin{align}
\label{eq:lemma:mazya-plamenevskii78-lemma-4.1-30}
\|e^{\eta_1 t} \check{u}\|^p_{L^p(D)} &\lesssim \int_{\operatorname{Im} \zeta = \eta_1} \|\widehat{u}(\cdot,\zeta)\|^p_{L^p(G)} \,d\zeta. 
\end{align}
From Lemma~\ref{lemma:8.8-lp} we get with the Fourier transform $\widehat{g}$ of $\check{g}$ for $\operatorname{Im} \zeta = \eta_1$ 
\begin{align*}
\|\widehat{u}(\cdot,\zeta)\|^p_{L^p(G)} &\lesssim (1 + |\zeta|^2)^{-p} \|\widehat{g}(\cdot,\zeta)\|^p_{L^p(G)} 
\leq (1 + |\zeta|^2)^{-p} \int_{x \in G} \int_{t \in \BbbR} \left|e^{-i \zeta t} \check{g}(x,t)\right| \, dt\, dx  \\
&\stackrel{\operatorname{Im} \zeta = \eta_1}{\lesssim} (1 + |\zeta|^2)^{-p} \int_{t \in \BbbR} \|e^{\eta_1 t} \check{g}(\cdot,t)\|^p_{L^p(G)}. 
\lesssim (1 + |\zeta|^2)^{-p} \|e^{\eta_1 t} \check{g}\|^p_{L^p(D)}. 
\end{align*}
Inserting this in (\ref{eq:lemma:mazya-plamenevskii78-lemma-4.1-30}) yields in view of $p > 1$ 
\begin{align*}
\|e^{\eta_1 t} \check{u}\|^p_{L^p(D)} &\lesssim \int_{\operatorname{Im} \zeta = \eta_1} 
(1 + |\zeta|^2)^{-p} \|e^{\eta_1 t} \check{g}\|^p_{L^p(D)}
\,d\zeta 
\lesssim \|e^{\eta_1 t} \check{g}\|^p_{L^p(D)}. 
\end{align*}
In total, we have arrived at 
\begin{align*}
\|e^{\eta_1 t} \zeta_\ell \check{u}\|_{L^p(D)} \lesssim \|e^{\eta_1 t} \check{g}\|_{L^p(D)} \lesssim M_{p,\eta_1}. 
\end{align*}
In turn, this implies in view of the support properties of $\zeta_\ell$ and $\check{g}$ 
\begin{align*}
\|e^{\eta t} \zeta_\ell \check{u}\|_{L^p(D)} &\lesssim e^{(\eta-\eta_1) \ell} \|e^{\eta_1 t} \zeta_\ell \check{u}\|_{L^p(D)} 
\lesssim e^{(\eta - \eta_1)\ell} M_{p,\eta_1} \lesssim e^{(\eta - \eta_1)(\ell-m)} M_{p,\eta}. 
\end{align*}
Again, selecting $\eta_1 = \eta + \varepsilon'$ for $m < \ell$ and $\eta_1 = \eta - \varepsilon'$ for $m \ge \ell$ yields 
the desired result. 
\end{proof}
We now come to the analog of Corollary~\ref{cor:strip-solution}: 
\begin{theorem}[\protect{\cite[Thm.~{4.1}]{mazya-plamenevskii78-translated}}]
\label{thm:mazya-plamenevskii78-thm-4.1}
Let $k \in \BbbN_0$, $\eta \in \BbbR$ be such that $\{\xi + i \eta\,|\, \xi \in \BbbR\} \cap \pm i \sigma^D  = \emptyset$ and $p \in (1,\infty)$. 
Then, for every $\check{g} \in W^{k,p}_\eta(D)$ there is a unique solution $\check{u} \in W^{k+2,p}_{\eta}(D) $ of (\ref{eq:8.6-Lp}) with the 
\sl{a priori} estimate 
\begin{align}
\label{eq:thm:mazya-plamenevskii78-thm-4.1-10}
\|\check{u}\|_{W^{k+2,p}_{\eta}(D)} \lesssim \|\check{g}\|_{W^{k,p}_{\eta}(D)}. 
\end{align}
\begin{proof}
We restrict to the case $k = 0$. By density, it suffices to assume that $\check{g} \in W^{k,p}_\eta(D) \cap W^{k,2}_\eta(D)$. Let ${\mathcal R}$
be the solution operator $\check{g} \mapsto \check{u}$ from Corollary~\ref{cor:strip-solution}. With a smooth and uniformly (in $\ell$) bounded partition of unity $(\zeta_\ell)_{\ell}$ 
subordinate to the the partitioning of $\BbbR$ by intervals $[\ell-1,\ell+1]$, we can define the solution $\check{u} \in W^{k+2,2}_\eta(D)$ as 
\begin{equation*}
\check{u} = {\mathcal R} \check{g} = \sum_{\ell,m} \zeta_\ell {\mathcal R} (\zeta_m \check{g}). 
\end{equation*}
\emph{Step 1:}
The map given by $(u_m)_{m \in \BbbZ} \mapsto (\sum_{m} e^{-|\ell - m|\varepsilon} u_m)_{\ell \in \BbbZ}$ is a bounded, linear 
maps $\ell^p \rightarrow \ell^p$ 
for $ p \in [1,\infty]$ if $\varepsilon > 0$.  This follows from Young's convolution theorem. Alternatively, one may consider 
the extremal cases $p=1$ and $p=\infty$
\begin{align*}
p&=1\colon  &\qquad 
\sum_{\ell} |\sum_{m} e^{-|\ell-m|\varepsilon} u_m| & \leq 
\sum_{m} |u_m| \sum_{\ell} e^{-|\ell-m|\varepsilon}  \lesssim \|(u_m)_m\|_{\ell^1} , \\
p&=\infty\colon & \qquad 
\sup_{\ell} |\sum_{m} e^{-|\ell-m| \varepsilon} u_m| &\lesssim 
\sup_{\ell} \|(u_m)_m\|_{\ell^\infty} \sum_{m} e^{-|\ell-m|\varepsilon} \lesssim \|(u_m)_m\|_{\ell^\infty}. 
\end{align*}
Interpolating between $p=1$ and $p=\infty$ concludes the proof. 

\emph{Step 2:} We claim that 
\begin{align}
\label{eq:thm:mazya-plamenevskii78-thm-4.1-20}
\|e^{\eta t} \check{u}\|_{L^p(D)} \lesssim \|e^{\eta t} \check{g}\|_{L^p(D)}. 
\end{align}
To see this, we note that the functions $\zeta_m \check{g}$ have support properties as required in Lemma~\ref{lemma:mazya-plamenevskii78-lemma-4.1}. 
We therefore conclude from Lemma~\ref{lemma:mazya-plamenevskii78-lemma-4.1} that 
\begin{align}
\label{eq:thm:mazya-plamenevskii78-thm-4.1-30}
\|e^{\eta t} \zeta_\ell \check{u} \|_{L^p(D)} \lesssim e^{-|m - \ell|\varepsilon} \|e^{\eta t} \zeta_m \check{g}\|_{L^p(D)}. 
\end{align}
Since $(\zeta_\ell)_\ell$ is a partition of unity with the finite the overlap property
$\sup_{t \in \BbbR} \operatorname{card} \{\ell \in \BbbZ\,|\, \zeta_\ell(t) \ne 0\} \leq 3$, we get,  
denoting by $\chi_A$ the characteristic function of the set $A$, 
\begin{align*}
\|e^{\eta t} \check{u} \|^p_{L^p(D)} & = 
\int_{D} \left| \sum_{\ell} \chi_{\{\zeta_\ell \ne 0\}} \cdot  \zeta_\ell e^{\eta t} \check{u}\right|^p 
\leq \int_D \left( \sum_{\ell} \chi_{\{\zeta_\ell \ne 0\}}^{p'} \right)^{p/p'} \sum_{\ell} |\zeta_\ell e^{\eta t} \check{u} |^p
\leq 3^{p/p'}
\|(\|e^{\eta t} \zeta_\ell \check{u}\|_{L^p(D)})_\ell\|^p_{\ell^p}. 
\end{align*}
(\ref{eq:thm:mazya-plamenevskii78-thm-4.1-30}) and Step~1 then imply 
\begin{align*}
\|e^{\eta t} \check{u} \|_{L^p(D)} & 
\lesssim \| (\|e^{\eta t} \zeta_m \check{g}\|_{L^p(D)} )_m\|_{\ell^p} 
\lesssim \|e^{\eta t} \check{g}\|_{L^p(D)}, 
\end{align*}
where we used again the support properties of $(\zeta_\ell)_\ell$. The proof (\ref{eq:thm:mazya-plamenevskii78-thm-4.1-20}) is complete. 

\emph{Step 3:} 
From (\ref{eq:lemma:mazya-plamenevskii78-eq.4.6-15}) in Lemma~\ref{lemma:mazya-plamenevskii78-eq.4.6} we infer 
\begin{align*}
\|e^{\eta t} (\zeta_\ell \check{u}) \|_{W^{k+2,p}(D)} 
\lesssim \|e^{\eta t} \zeta_\ell^\prime \check{g}\|_{W^{k,p}(D)} + \|e^{\eta t} \zeta_\ell^\prime \check{u}\|_{L^p(D)}, 
\end{align*}
for a cut-off function $\zeta_\ell$ with slightly larger support than $\zeta_\ell$. Combining these local estimates produces 
\begin{align*}
\|e^{\eta t} \check{u}\|_{W^{k+2,p}(D)} \lesssim \|e^{\eta t} \check{g}\|_{w^{k,p}(D)} + \|e^{\eta t} \check{u}\|_{L^p(D)}.  
\end{align*}
Inserting the estimate from Step~3, we get 
\begin{align*}
\|\check{u}\|_{W^{k+2,p}_\eta(D)} \lesssim \|\check{g}\|_{W^{k,p}(D)}. 
\end{align*}
\emph{Step 4:} We now show uniqueness in $W^{k+2,p}_\eta(D)$ of the solution (\ref{eq:8.6-Lp}). Let $u \in W^{k+2,p}_{\eta}(D)$
solve (\ref{eq:8.6-Lp}) with $\check{g} = 0$. Hence, $\check{u}$ is smooth up to the boundary and thus locally in any Sobolev space. 
Define the anuli $D_M:=\{(x,t) \in D\,|\, M < t < M+1\}$ cut-off functions $\varphi_M \in C^\infty_0(\BbbR)$ with 
$\varphi_M \equiv 1$ for $|t| \leq M$ and (uniformly in $M$) bounded derivatives. Define 
$\widetilde D_M:= \left(\overline{D}_{M-1} \cup \overline{D}_{M} \cup \overline{D}_{M+1}\right)^\circ$ 
and 
$D^\prime_M:= \left(\overline{D}_{M-2} \overline{D}_{M-1} \cup \overline{D}_{M} \cup \overline{D}_{M+1} \overline{D}_{M+2}\right)^\circ$. 
From (\ref{eq:lemma:mazya-plamenevskii78-eq.4.6-15}) and $-\Delta u = 0$, we infer for any $k$,
\begin{equation}
\label{eq:thm:mazya-plamenevskii78-thm-4.1-50}
\|u\|_{W^{k,p}(\widetilde{D}_M)} \lesssim \|u\|_{L^p(D^\prime_M)}. 
\end{equation}
The function $u_M:= \varphi_M u$ satisfies 
\begin{align*}
-\Delta u_M & = -2 \nabla \varphi_M \cdot \nabla u - \Delta \varphi_M u =:\check{g}_M, \quad u_M = 0 \quad \mbox{ on $\partial D$}. 
\end{align*} 
We may select by Sobolev's embedding theorem $k$ with $W^{k,p}(\widetilde{D}_M) \subset L^2(\widetilde{D}_M)$ (with embedding constant independent of $M$). 
Thus, with (\ref{eq:thm:mazya-plamenevskii78-thm-4.1-50}) we get 
\begin{align*}
\|e^{\eta t} \check{g}_M \|_{L^{2}(D)} \lesssim \|e^{\eta t} u\|_{L^p(D^\prime_M)}. 
\end{align*}
We conclude from the $L^2$-theory of Corollary~\ref{cor:strip-solution}
$$
\|u_M\|_{W^{2,2}_\eta (D)} \lesssim \|e^{\eta t} \check{g}_M\|_{L^2(D)} \lesssim \|e^{\eta t} u\|_{L^p(D^\prime_M)}. 
$$
Since $e^{\eta t} u \in L^p(D)$ (by $u \in W^{k+2,p}_\eta(D)$)  the right-hand side tends to zero as $M \rightarrow \infty$. We conclude that 
$u = 0$ pointwise on $D$. Thus $u = 0$. 
\end{proof}
\end{theorem}
\subsection{Solution decomposition in a cone} 
For the relation between the spaces on the cone ${\mathcal C}$ and the strip $D$, we have the analog of Lemma~\ref{lemma:uKsgamma-Mu}: 
\begin{lemma} 
\label{lemma:uKsgamma-Mu-Lp}
For $s \in \BbbN_0$, $\gamma \in \BbbR$, $p \in (1,\infty)$ there holds with $d = 2$
\begin{align}
u \in K^s_\gamma(\mathcal{C}) \quad \Longleftrightarrow \quad e^{t (\gamma - s + d/p)} \check{u}  \in W^{s,p}(D). 
\end{align}
Here, $u$ and $\check{u}$ are related via the Euler transformation $r  = e^t$. 
\end{lemma} 
\begin{proof}
Analogous to the proof of Lemma~\ref{lemma:uKsgamma-Mu}. 
\end{proof}
\begin{numberedproof}{of Proposition~\ref{prop:Lp}}
One essentially repeats the arguments of Proposition~\ref{prop:solution_u1}. 
Note that $\operatorname{supp} f \subset B_1(0)$ by the support properties of $u_1$. 

\emph{Step 1:} We show that we may assume $f \in W^{k,p} \cap C^\infty(\overline{{\mathcal C}})$ together with $\operatorname{supp} f \subset B_{2}(0)$
and the conditions 
\begin{equation}
\label{eq:proof-of-prop-Lp-10}
\partial^i_x \partial^j_y f(0) = 0, 
\qquad i+j \leq k-2/p. 
\end{equation}

We start by noting that for any finite cone ${\mathcal C}_r$, $r > 1$, the solution $u_1$ (restricted to ${\mathcal C}_r$) 
is the solution of the Dirichlet problem
\begin{align}
\label{eq:proof-of-prop-Lp-20}
-\Delta u_1 & = f \quad \mbox{ on ${\mathcal C}_r$}, \qquad u_1 = 0 \quad \mbox{ on $\partial {\mathcal C}_r$}. 
\end{align}
Next, we approximate $f \in W^{k,p}({\mathcal C})$ by a function $\widetilde f \in C^\infty(\overline{\mathcal{C}})$. We may assume that $\widetilde{f} = 0$ 
outside ${\mathcal C}_{1+\varepsilon}$, $\varepsilon > 0$, and we may assume that $\widetilde{f}$ satisfies 
(\ref{eq:proof-of-prop-Lp-10}) since the point evaluations in (\ref{eq:proof-of-prop-Lp-10}) are continuous on $W^{k,p}$ by assumption and we thus 
may subtract a suitable polynomial.  Let $\widetilde{u}_1$ solve  
\begin{align}
\label{eq:proof-of-prop-Lp-30}
-\Delta \widetilde{u}_1 & = f \quad \mbox{ on ${\mathcal C}_r$}, \qquad \widetilde{u}_1 = 0 \quad \mbox{ on $\partial {\mathcal C}_r$}. 
\end{align}
Then, 
$$\|u_1 - \widetilde{u}_1\|_{W^{1,p}({\mathcal C}_r)} \lesssim \|f - \widetilde{f}\|_{L^p({\mathcal C}_r)} 
$$ 
by, e.g., \cite[Thm.~{7.1}]{giaquinta-martinazzi12}. (The condition that $\mathcal{C}_r$ be bilipschitz equivalent to a cube may be addressed
by replacing ${\mathcal C}_r$ by a curvilinear quadrilateral with an additional corner on the curved part of $\partial{\mathcal C}_r$.) 
Next, for a cut-off function $\widetilde{\chi} \in C^\infty_0(B_r)$ with $\widetilde{\chi} \equiv 1$ on $B_{1+2 \epsilon}(0)$, we consider 
$\widehat{u}_1:= \widetilde{\chi} \widetilde{u}_1$, which solves on ${\mathcal C}$ 
\begin{align*}
-\Delta \widehat{u}_1 & = \widetilde{\chi} \widetilde{f} - 2 \nabla \widetilde{\chi} \cdot \nabla \widetilde{u}_1 - \Delta \widetilde{\chi} \widetilde{u}_1 
=: \widehat{f}. 
\end{align*}
We note that 
\begin{enumerate}
\item $\widehat{f} \in C^\infty(\overline{\mathcal{C}}) \cap W^{k,p}({\mathcal C})$ 
and satisfies the condition (\ref{eq:proof-of-prop-Lp-10});  
\item $\widehat{u}_1 = \widetilde{u}_1$ near the origin; 
\item $\widetilde{u}_1$ can be assumed smooth on $\mathcal{C}_r \setminus \mathcal{C}_{1+\epsilon}$  and there $\widetilde{u}_1$ 
is controlled by $\|\widetilde{f}\|_{L^p}$. 
\end{enumerate} 
We conclude that $\widehat{u}_1$ corresponds to an $\widehat{f} \in W^{k,p}(\mathcal{C}) \cap C^\infty(\overline{\mathcal{C}})$ and, as $\widetilde{f} \rightarrow f$, it approaches 
$u_1$ on ${\mathcal C}_{1+2 \epsilon}$. 

\emph{Step 2:} By Step~1, we may assume $f \in W^{k,p}(\mathcal{C}) \cap C^\infty(\overline{\mathcal C})$. Hence, the representation formulas from
Proposition~\ref{prop:solution_u1} are valid. In particular, the Mellin transform of $f$ is meromorphic with possible poles 
at $\{- i m\,|\, m \in \BbbN_0, m \ge \lfloor k-2/p\rfloor\}$. We may therefore define $u_0$ as in the proof of Proposition~\ref{prop:solution_u1}
by the inverse Fourier transformation on the line $\operatorname{Im} \zeta = - (\gamma - s + 2/p)$ with $\gamma = 0$ and $s = k+2$. 
Theorem~\ref{thm:mazya-plamenevskii78-thm-4.1} and Lemma~\ref{lemma:uKsgamma-Mu-Lp} provide
\begin{align*}
\|u_0\|_{K^{k+2,p}_0(\mathcal{C})} & \sim \| \check{u}\|_{W^{k+2,p}_{-(k+2)+2/p}(D)} \lesssim \|\check{g} \|_{W^{k,p}_{-(k+2)+2/p}(D)} 
 = \|\check{f}\|_{W^{k,p}_{-k+2/p}(D)} \sim \|f\|_{K^{k,p}_{0}(\mathcal{C})} \lesssim \|f\|_{W^{k,p}(\mathcal{C}) }. 
\end{align*}
The solution $u_1$ can be represented as the inverse Mellin transform on the line $\operatorname{Im} \zeta = 0$. By the residue theorem, we conclude 
the stated representation (\ref{eq:representation-formula-Lp-1}).  

It is worth pointing out that Lemma~\ref{lemma:embedding-Wkp} shows that the 
functionals $f \mapsto \int_{\mathcal{C}} r^{-\lambda^D_j} f\, d\mathbf{x}$ are continuous on 
$\{f \in W^{k,p}({\mathcal C})\,|\, \partial^i_x \partial^j f(0) = 0 \quad \mbox{ for $i+j\leq k-2/p$}\}$. 
\end{numberedproof}

\fi 
\bibliographystyle{abbrv}
\bibliography{maxwell}
{\small
{\em Authors' addresses}:
{\em Jens Markus Melenk}, {\em Claudio Rojik}, Technische Universit\"at Wien, Vienna, Austria
 e-mail: \texttt{melenk@\allowbreak tuwien.ac.at}.

\end{document}